\documentclass{article}

\PassOptionsToPackage{numbers, compress}{natbib}

\usepackage[preprint]{neurips_2022}

\usepackage[utf8]{inputenc} 
\usepackage[T1]{fontenc}    
\usepackage{hyperref} 
\usepackage{url}            
\usepackage{booktabs}       
\usepackage{amsfonts}       
\usepackage{nicefrac}       
\usepackage{microtype}      
\usepackage{xcolor}         
\usepackage{subfig}
\usepackage{enumitem}
\usepackage{float}
\usepackage[textsize=tiny]{todonotes}
\usepackage{amssymb, amsmath, amsthm}
\usepackage[scr=rsfso]{mathalpha}
\usepackage{cleveref}
\usepackage{xspace}
\Crefname{assumption}{Assumption}{Assumptions}
\Crefname{lemma}{Lemma}{Lemmas}
\theoremstyle{plain}
\newtheorem{theorem}{Theorem}[section]
\newtheorem{lemma}[theorem]{Lemma}
\newtheorem{corollary}[theorem]{Corollary}
\newtheorem{proposition}[theorem]{Proposition}

\theoremstyle{definition}
\newtheorem{definition}{Definition}[section]
\newtheorem{assumption}{Assumption}

\theoremstyle{remark}

\newcommand{\saferef}[2]{%
	#1~\texorpdfstring{\ref{#2}}{\ref*{#2}}}

\usepackage{multirow}
\usepackage{thmtools}

\usepackage{mathtools}
\DeclarePairedDelimiter\norm{\lVert}{\rVert}
\DeclarePairedDelimiter\fnorm{\lVert}{\rVert_{\rm{F}}}
\DeclarePairedDelimiter\opnorm{\lVert}{\rVert_{\rm{op}}}

\def\tD{{\tilde D}}

\def\defas{\stackrel{\text{\tiny def}}{=}}
\def\df{\hat{\mathsf{df}}}
\def\Rem{{\mathrm{Rem}}}

\def\argmin{\mathop{\rm arg\, min}}

\def\iid{i.i.d.\@\xspace}
\def\eg{e.g.\@\xspace}
\def\ie{i.e.\@\xspace}
\def\resp{resp.\@\xspace}

\def\R{{\mathbb{R}}}
\def\E{{\mathbb{E}}}
\def\N{{\mathbb{N}}}
\def\P{{\mathbb{P}}}

\DeclareMathOperator{\trace}{Tr}
\DeclareMathOperator{\range}{range}
\DeclareMathOperator{\rank}{rank}
\let\vec\relax
\DeclareMathOperator{\vec}{\mathbf{vec}}
\DeclareMathOperator{\diag}{\mathbf{diag}}
\let\div\relax
\DeclareMathOperator{\div}{div}
\DeclareMathOperator{\supp}{supp}

\def\mathbold{\boldsymbol} 

\def\ba{\mathbold{a}}

\def\bA{\mathbold{A}}
\def\hbA{\smash{\widehat{\mathbf A}}}

\def\bb{\mathbold{b}}

\def\bB{\mathbold{B}}
\def\hbB{{\widehat{\bB}}}

\def\bC{\mathbold{C}}

\def\bE{\mathbold{E}}

\def\bff{\mathbold{f}}

\def\bF{\mathbold{F}}\def\calF{{\cal F}}

\def\bg{\mathbold{g}}

\def\bG{\mathbold{G}}

\def\bH{\mathbold{H}}
\def\hbH{{\widehat{\bH}}}\def\tbH{{\tilde{\bH}}}

\def\bI{\mathbold{I}}

\def\bJ{\mathbold{J}}

\def\bK{\mathbold{K}}

\def\bL{\mathbold{L}}

\def\bM{\mathbold{M}}

\def\bN{\mathbold{N}}
\def\calN{{\cal N}}

\def\bP{\mathbold{P}}

\def\bQ{\mathbold{Q}}

\def\bR{\mathbold{R}}\def\hbR{{\widehat{\bR}}}

\def\bS{\mathbold{S}}\def\hbS{{\widehat{\bS}}}

\def\bT{\mathbold{T}}

\def\be{\mathbold{e}}

\def\bu{\mathbold{u}}
\def\tbu{{\widetilde{\bu}}}
\def\bU{\mathbold{U}}
\def\calU{{\cal U}}

\def\bv{\mathbold{v}}

\def\bV{\mathbold{V}}

\def\bW{\mathbold{W}}

\def\bx{\mathbold{x}}

\def\bX{\mathbold{X}}

\def\by{\mathbold{y}}

\def\bY{\mathbold{Y}}

\def\bz{\mathbold{z}}

\def\bZ{\mathbold{Z}}

\def\bbeta{\mathbold{\beta}}

\def\hbbeta{\hat{\bbeta}{}}

\def\ep{\varepsilon}
\def\bep{ {\mathbold{\ep} }}

\def\bfeta{\mathbold{\eta}}

\def\bLambda{\mathbold{\Lambda}}

\def\bxi{\mathbold{\xi}}

\def\brho{\mathbold{\rho}}

\def\hsigma{\widehat{\sigma}}

\def\bSigma{\mathbold{\Sigma}}\def\hbSigma{{\widehat{\bSigma}}}

\def\hphi{\widehat{\phi}}

\graphicspath{{figures/}} 

\title{Noise Covariance Estimation in Multi-Task High-dimensional Linear Models}

\author{%
  Kai Tan \\
  Department of Statistics\\
  Rutgers University\\
  Piscataway, NJ 08854 \\
  \texttt{kai.tan@rutgers.edu} \\
   \And
   Gabriel Romon \\
   CREST, ENSAE, IP Paris \\
   Palaiseau 91120 Cedex, France \\
   \texttt{gabriel.romon@ensae.fr} \\
   \AND
   Pierre C Bellec\\
    Department of Statistics\\
    Rutgers University\\
    Piscataway, NJ 08854 \\
   \texttt{pierre.bellec@rutgers.edu} \\
}

\begin{document}
\maketitle

\begin{abstract}
  This paper studies the multi-task high-dimensional linear regression models where the noise among different tasks is correlated, in the moderately high dimensional regime where sample size $n$ and dimension $p$ are of the same order. 
  Our goal is to estimate the covariance matrix of the noise random vectors, or equivalently the correlation of the noise variables on any pair of two tasks. Treating the regression coefficients as a nuisance parameter, we leverage the multi-task elastic-net and multi-task lasso estimators to estimate the nuisance. By precisely understanding the bias of the squared residual matrix and by correcting this bias, we develop a novel estimator of the noise covariance that converges in Frobenius norm at the rate $n^{-1/2}$ when the covariates are Gaussian. This novel estimator is efficiently computable.
  
  Under suitable conditions, the proposed estimator of the noise covariance attains the same rate of convergence as the ``oracle'' estimator that knows in advance the regression coefficients of the multi-task model. The Frobenius error bounds obtained in this paper also illustrate the advantage of this new estimator compared to a method-of-moments estimator that does not attempt to estimate the nuisance.
  
  As a byproduct of our techniques, we obtain an estimate of the generalization error of the multi-task elastic-net and multi-task lasso estimators. Extensive simulation studies are carried out to illustrate the numerical performance of the proposed method.
\end{abstract}

\section{Introduction}

\subsection{Model and estimation target}
Consider a multi-task linear model with $T$ tasks and $n$ \iid observations $(\bx_i, Y_{i1}, Y_{i2},\dots, Y_{iT})$, $\forall i=1,...,n$, where $\bx_i\in \R^p$ is a random feature vector and $Y_{i1}, \ldots, Y_{iT}$ are responses
in the model 
\begin{equation}\label{eq: model}
    \begin{aligned}
        Y_{it} &= \bx_i^\top \bbeta^{(t)} + E_{it} &&\text{for each } t = 1, ..., T; i=1,...,n
               &&\text{(scalar form)},
    \\
        \by\smash{{}^{(t)}}                             &= \bX \bbeta\smash{{}^{(t)}} + \bep\smash{{}^{(t)}}
                                          &&\text{for each } t = 1, ..., T
                                          &&\text{(vector form)},
    \\
        \bY &= \bX\bB^* + \bE
        &&
        &&\text{(matrix form)},
    \end{aligned}
\end{equation}
where $\bX\in\R^{n\times p}$ is the design matrix with rows $(\bx_i^{\top})_{i=1,...,n}$, 
$\by^{(t)} = (Y_{1t},..., Y_{nt})^\top$ is the response vector for task $t$,
$\bep^{(t)} = (E_{1t},...,E_{nt})^\top$ is the noise vector for task $t$,
$\bbeta^{(t)} \in \R^{p}$ is an unknown fixed coefficient vector for task $t$. 
In matrix form, 
$\bY\in \R^{n\times T}$ is the response matrix with columns $\by^{{(1)}},...,\by^{(T)}$,
$\bE\in \R^{n\times T}$ has columns $\bep^{{(1)}},...,\bep^{(T)}$,
and $\bB^*\in\R^{p\times T}$ is an unknown coefficient matrix with
columns $\bbeta^{{(1)}},...,\bbeta^{(T)}$. The three forms in \eqref{eq: model} are equivalent.

While the $n$ vectors $(\bx_i^\top, y_i^{(1)}, \ldots, y_i^{(T)})_{i=1,...,n}$ of dimension $p+T$ are \iid, 
we assume that for each observation $i=1,...,n$, the noise random variables $E_{i1},...,E_{iT}$ are centered and correlated. 

The focus of the present paper is on estimation of the
noise covariance matrix $\bS\in\R^{T\times T}$, which has entries
$\bS_{tt'} = \E[\varepsilon_1^{(t)}\varepsilon_1^{(t')}]$ for any pair $t,t'=1,\ldots,T$, or equivalently
$$
\bS = \E[\tfrac1n\bE^\top\bE].
$$

The noise covariance plays a crucial role in multi-task linear models because it characterizes the noise level and correlation between different tasks:
if tasks $t=1,...,T$ represent time this captures temporal correlation; 
if tasks $t=1,...,T$ represent different activation areas in the brain (\eg, \cite{Salmon2019handlings}) this captures spatial correlation.

Since $\bS$ is the estimation target,
we view $\bB^*$ as an unknown nuisance parameter.
If $\bB^*=\mathbf0$,  then $\bY = \bE$, hence $\bE$ is directly observed and a natural estimator is the sample covariance
$\frac1n\bE^\top\bE$.
There are other possible choices for the sample covariance;
ours coincides with the maximum likelihood estimator of the centered Gaussian model  
where the $n$ samples are \iid from $\mathcal N_T(\bf 0, \bS)$.
In the presence of a nuisance parameter $\bB^*\neq \bf 0$,
	the above sample covariance is not computable since we only observe $(\bX, \bY)$ and do not have access to $\bE$. 
	Thus we will refer to 
	$\frac1n\bE^\top\bE \in\R^{T\times T}$
	as the \textit{oracle estimator} for $\bS$, and its error
	$\frac1n\bE^\top\bE - \bS$ 
	will serve as a benchmark.

The nuisance parameter $\bB^*$ is not of interest by itself, but if an estimator $\hbB$ is available that provides good estimation
of $\bB^*$, we would hope to leverage $\hbB$ to estimate the nuisance 
and improve estimation of $\bS$. 
For instance given an estimate $\hbB$ such that
$\fnorm{\bX(\hbB-\bB^*)}^2/n\to0$, one may use the estimator
\begin{equation}
	\textstyle
	\label{naive}
	\hbS_{(\text{naive})} = \frac1n (\bY - \bX\hbB)^\top(\bY - \bX\hbB)
\end{equation}
to consistently estimate $\bS$ in Frobenius norm.
We refer to this estimator as the \textit{naive estimator} since it is obtained by simply replacing the noise $\bE$ in the oracle estimator $\frac1n \bE^\top\bE$ with the residual matrix $\bY - \bX\hbB$. 
However,
in the regime $p/n\to\gamma$ of interest in the present paper,
the convergence $\fnorm{\bX(\hbB-\bB^*)}^2/n\to0$ is not true even for $T=1$
and common high-dimensional estimators such as Ridge regression \cite{dobriban2018high} or the Lasso \cite{bayati2012lasso,miolane2018distribution}.
Simulations in \Cref{sec: simulation} will show that
\eqref{naive} presents a major bias for estimation of $\bS$.
One goal of this paper is to develop estimator $\hbS$ of $\bS$ 
by exploiting a commonly used estimator $\hbB$ of the nuisance,
so that in the regime $p/n\to\gamma$ the error $\hbS-\bS$ is comparable to the benchmark $\frac1n\bE^\top\bE - \bS$.

\subsection{Related literature}
If $T=1$, the above model \eqref{eq: model} reduces to the standard linear model
with $\bX\in\R^{n\times p}$ and response vector $\by^{(1)}\in \R^n$.
We will refer to the $T=1$ case as the single-task linear model and drop the superscript $^{(1)}$ for brevity, \ie, 
	$y_i = \bx_i^\top \bbeta^* + \ep_i$,
where $\ep_i$ are \iid with mean $0$, and unknown variance $\sigma^2$. The coefficient vector $\bbeta^*$ is typically assumed to be $s$-sparse, \ie, $\bbeta^*$ has at most $s$ nonzero entries. In this single-task linear model, estimation of noise covariance $\bS$ reduces to estimation of the noise variance  $\sigma^2=\E[\ep_i^2]$, which has been studied in the literature. \citet{fan2012variance} proposed a consistent estimator for $\sigma^2$ based on a refitted cross validation method, which assumes the support of $\bbeta^*$ is correctly recovered; \cite{belloni2011square} and  \cite{sun2012scaled} introduced square-root Lasso (scaled Lasso) to jointly estimate the coefficient $\bbeta^*$ and noise variance $\sigma^2$ by 
\begin{equation}\label{eq: scaled-lasso}
    \textstyle
	(\hbbeta, \hsigma) =\argmin_{\bbeta\in \R^p, \sigma>0} \frac{\norm{\by - \bX\bbeta}^2}{2n\sigma} + \frac{\sigma}{2} + \lambda_0\norm{\bbeta}_1.
\end{equation}
This estimator $\hsigma$ is consistent only when the prediction error $\norm{\bX(\hbbeta - \bbeta^*)}^2 /n$ goes to 0, which requires $s\log(p)/n\to 0$. 
Estimation of $\sigma^2$ without assumption on $\bX$ was proposed in \cite{yu2019estimating} by utilizing natural parameterization of the penalized likelihood of the linear model. Their estimator can be expressed as the minimizer of the Lasso problem: 
$\hsigma^2_{\lambda} = \min_{\bbeta\in \R^p} \frac{1}{n} \norm{\by - \bX\bbeta}^2 + 2\lambda\norm{\bbeta}_1.$ Consistency of these estimators \cite{sun2012scaled,belloni2011square,belloni2014,yu2019estimating} requires $s\log(p)/n \to 0$ and does not hold in the high-dimensional proportional regime $p/n\to\gamma\in (0, \infty)$.
For this proportional regime $p/n \to \gamma \in (0, \infty)$, \cite{dicker2014variance} introduced a method-of-moments estimator $\hsigma^2$ of $\sigma^2$, 
\begin{equation}\label{eq: est-dicker14}
	\hsigma^2 = \frac{n+p+1}{n(n+1)} \norm{\by}^2 - \frac{1}{n(n+1)}\norm{\bSigma^{-\frac12}\bX^\top\by}^2, 
\end{equation}
which is unbiased, consistent, and asymptotically normal in high-dimensional linear models with Gaussian predictors and errors. Moreover, \cite{janson2017eigenprism} developed an EigenPrism procedure for the same task as well as confidence intervals for $\sigma^2$.
The estimation procedures in these two papers don't attempt to estimate the nuisance parameter $\bbeta^*$, and require no sparsity on $\bbeta^*$ and isometry structure on $\bSigma$, but assume $\|\bSigma^{\frac 12}\bbeta^*\|^2$ is bounded. 
Maximum Likelihood Estimators (MLEs) were studied in
\cite{dicker2016maximum} for joint estimation of noise level and signal strength in high-dimensional linear models with fixed effects; they showed that a classical MLE for random-effects models may also be used effectively in fixed-effects models.

In the proportional regime,
\cite{bayati2013estimating,miolane2018distribution} used the Lasso to estimate the nuisance $\bbeta^*$ and produce estimator for $\sigma^2$. Their approach requires an uncorrelated Gaussian design assumption with $\bSigma = \bI_p$. 
\citet{bellec2020out} provided consistent estimators of a similar nature for $\sigma^2$ using more general M-estimators with convex penalty without requiring $\bSigma = \bI_p$. In the special case of the squared loss, this estimator has the form \cite{bayati2013estimating,miolane2018distribution,bellec2020out}
\begin{equation}\label{eq: est-bellec20}
	\hsigma^2 = (n -\df)^{-2}\big\{ \norm{\by-\bX\hbbeta}^2 (n+p -2\df) - \norm{\bSigma^{-\frac12}(\by-\bX\hbbeta)}\big\},
\end{equation}
where $\df = \trace[(\partial/\partial \by) \bX\hbbeta]$ denotes the degrees of freedom. This estimator coincides with the method-of-moments estimator in \cite{dicker2014variance} when $\hbbeta = \bf0$. 

For multi-task high-dimensional linear model \eqref{eq: model} with $T\ge 2$, the estimation of $\bB^*$ is studied in \cite{lounici2011oracle}, \cite{obozinski2011support}, \cite{simon2013blockwise}. These works suggest to use a joint convex optimization problem over the tasks to estimate $\bB^*$. A popular choice is the multi-task elastic-net, which
solves the convex optimization problem 
\begin{equation}\label{eq: hbB}
	\hbB=\argmin_{\bB\in\R^{p\times T}}
	\Big(
	\frac{1}{2n}\fnorm*{\bY - \bX\bB }^2 + \lambda \norm{\bB}_{2,1}
	+ \frac{\tau}{2} \fnorm{\bB}^2
	\Big),
\end{equation}
where $\|\bB\|_{2,1} = \sum_{j=1}^p \|{\bB^{\top} \be_j}\|_2$, and $\fnorm{\cdot}$ denotes the Frobenius norm of a matrix. 
This optimization problem can be efficiently solved by existing statistical packages, for instance, scikit-learn \citep{scikit-learn}, and glmnet  \citep{friedman2010regularization}. 
Note that  \eqref{eq: hbB} is also referred to as multi-task (group) Lasso and multi-task Ridge if $\tau = 0$ and $\lambda =0$, respectively.
\citet{van2016chi} extended square-root Lasso \citep{belloni2011square} and scaled Lasso \citep{sun2012scaled} to multi-task setting by solving the following problem 
\begin{align}\label{eq: multi-ScaledLasso} 
	(\hbB, \hbS) = 
	\argmin_{\bB,\bS \succ 0} \Big\{\frac{1}{n} \trace\big((\bY - \bX\bB)\bS^{-\frac 12}(\bY - \bX\bB)^\top\big) + \trace(\bS^{\frac 12}) + 2\lambda_0\norm{\bB}_1 \Big\},
\end{align}
where $\norm{\bB}_1 = \sum_{j,t} |B_{jt}|$. 
Note that the covariance estimator in \eqref{eq: multi-ScaledLasso} is constrained to be positive definite.
\citet{molstad2019new} studied the same problem and proposed to estimate $\bS$ by \eqref{naive} with $\hbB$ in \eqref{eq: multi-ScaledLasso}, which is consistent under Frobenius norm loss when $\fnorm{\bX(\hbB-\bB^*)}^2/n\to0$. 
In a recent paper, \citet{bellec2021chi} studied the multi-task Lasso problem and proposed confidence intervals for single entries of $\bB^*$ and confidence ellipsoids for single rows of $\bB^*$ under the assumption that $\bS$ is proportional to the identity, which may be restrictive in practice. This literature generalizes degrees of freedom adjustments from single-task to multi-task models, which we will illustrate in \Cref{sec: NewEstimator}. 


Noise covariance estimation in the high dimensional multi-task linear model is a difficult problem. If the estimand $\bS$ is known to be diagonal, 
estimating $\bS$ reduces to the estimation of noise variance for each task, 
in which the existing methods for single-task high-dimensional linear models can be applied.  
Nonetheless, for general positive semi-definite matrix $\bS$, the noise among different tasks may be correlated, 
hence the existing methods are not readily applicable, and a more careful analysis is called for to incorporate the correlation between different tasks. \citet{fourdrinier2021covariance} considered estimating $\bS$ for the multi-task model \eqref{eq: model} where rows of $\bE$ have elliptically symmetric distribution and in the classical regime $p\le n$. However, their estimator has no statistical guarantee under Frobenius norm loss. 

Recently, for the proportional regime $p/n \to \gamma \in (0, \infty)$, \cite{celentano2021cad} generalized the estimator $\hsigma^2$ in \cite{bayati2013estimating} to the multi-task setting with $T=2$. Their work covers correlated Gaussian designs, where a Lasso or Ridge regression is used to estimate $\bbeta^{(1)}$ for the first task, and another Lasso or Ridge regression is used to estimate $\bbeta^{(2)}$ for the second task. 
In other words, they estimate the coefficient vector for each task separately instead of using a multi-task estimator like \eqref{eq: hbB}. 
It is not trivial to adapt their estimator from the setting $T=2$ to larger $T$, and allow $T$ to increase with $n$. 
This present paper takes a different route and aims to fill this gap by proposing a novel noise covariance estimator with theoretical guarantees. 
Of course, our method applies directly to the 2-task linear model considered in \cite{celentano2021cad}. 

\subsection{Main Contributions}
\label{sec:contributions}

The present paper introduces a novel estimator $\hbS$ in \eqref{eq: hbS} of the noise covariance $\bS$, which provides consistent estimation of $\bS$ in Frobenius norm,
in the regime where $p$ and $n$ are of the same order. 
The estimator $\hbS$ is based on the multi-task elastic-net 
estimator $\hbB$ in \eqref{eq: hbB} of the nuisance, and can be seen
as a de-biased version of the naive estimator \eqref{naive}. 
The naive estimator \eqref{naive} suffers from a strong bias in the regime
where $p$ and $n$ are of the same order, 
and the estimator $\hbS$
is constructed by precisely understanding this bias and correcting it.


After introducing this novel estimator $\hbS$ in \Cref{def: hbS} below,
we prove several rates of convergence for the Frobenius error
$\fnorm{\hbS-\bS}$, which is comparable, in terms of rate of convergence,
to the benchmark $\fnorm{\frac 1 n \bE^\top\bE - \bS}$ under suitable assumptions.
%

As a by-product of the techniques developed for the construction of $\hbS$,
we obtain estimates of the generalization error 
%
of $\hbB$, which {are} of independent interest and can be used for parameter tuning.

\subsection{Notation}
Basic notation and definitions that will be used in the rest of the paper are given here. Let $[n] = \{1, 2,\ldots, n \}$ for all $n\in\N$. 
The vectors $\be_i\in \R^n,\be_j\in\R^p, \be_t\in\R^T$ denote the canonical basis vector of the corresponding index. 
We consider restrictions of vectors (\resp, of matrices)
by zeroing the corresponding entries (\resp, columns). 
More precisely, for $\bv\in\R^p$ and index set $B\subset [p]$, $\bv_B\in\R^p$ is the vector
with $(\bv_B)_j = 0$ if $j\notin B$ and $(\bv_B)_j = v_j$ if $j\in B$.
If $\bX\in\R^{n\times p}$ and $B\subset[p]$, $\bX_B\in\R^{n\times p}$
is 
such that $(\bX_B)\be_j = {\mathbf 0}$ if $j\notin B$ and
$(\bX_B)\be_j= \bX\be_j$ if $j\in B$.
For a real vector $\ba \in \R^p$, 
$\norm*{\ba}$ 
denotes its Euclidean norm. 
For any matrix $\bA$, $\bA^\dagger$ is its Moore–Penrose inverse;  
$\fnorm*{\bA}$ 
,$\opnorm*{\bA}$, 
$\norm*{\bA}_*$ 
denote its Frobenius, operator and nuclear norm, respectively. 
Let $\norm{\bA}_0$ be the number of non-zero rows of $\bA$. 
Let $\bA\otimes \bB$ be the Kronecker product of $\bA$ and $\bB$, and $\langle \bA, \bB\rangle = \trace(\bA^\top\bB)$
is the Frobenius inner product for matrices of identical size. 
For $\bA$ symmetric, $\phi_{\min}(\bA)$ and $\phi_{\max}(\bA)$ denote its smallest and largest eigenvalues, respectively. 
Let $\bI_n$ denote the identity matrix of size $n$ for all $n\in\N$.
For a random sequence $\xi_n$, we write $\xi_n = O_P(a_n)$ if $\xi_n/a_n$ is stochastically bounded. 
$C$ denotes an absolute constant and 
$C(\tau, \gamma)$ stands for a generic positive constant depending on $\tau,\gamma$;
their expression may vary from place to place.



\subsection{Organization}
The rest of the paper is organized as follows. 
\Cref{sec: NewEstimator} introduces our proposed estimator for noise covariance. 
\Cref{sec: main-results} presents our main theoretical results on proposed estimator and some relevant estimators. 
\Cref{sec: simulation} demonstrates through numerical experiments that our estimator outperforms several existing methods in the literature, which corroborates our theoretical findings in \Cref{sec: main-results}. 
\Cref{sec: discussion} provides discussion and points out some future research directions. 
Proofs of all the results stated in the main body are given in the supplementary, which starts with an outline for ease of navigation. 

\section{Estimating noise covariance, with possibly diverging number of tasks  \texorpdfstring{$\bT$}{T}}\label{sec: NewEstimator}
Before we can define our noise covariance estimator, we need to introduce the following building blocks.
Let $\hat{\mathscr{S}} = \{k\in [p]: {\hbB{}^{\top} \be_k} \neq 0\}$  
denote the set of nonzero rows of $\hbB$ in \eqref{eq: hbB}, and  let $|\hat{\mathscr{S}}|$ denote the cardinality of $\hat{\mathscr{S}}$. 
For each $k\in \hat{\mathscr{S}}$, define $\bH^{(k)}=\lambda\|\hbB{}^\top \be_k\|^{-1}(\bI_T - \hbB{}^\top\be_k \be_k^\top\hbB ~  \|\hbB{}^\top\be_k\|^{-2} )$, which is the Hessian of the map $\bu \mapsto \lambda\norm*{\bu}$ at $\bu = \hbB{}^\top \be_k$ when $\bu\ne \mathbf0$.
Define $\bM,\bM_1\in\R^{pT \times pT}$ by
\begin{equation}
    \textstyle
\bM_1 = \bI_T \otimes (\bX_{\hat{\mathscr{S}}}^\top\bX_{\hat{\mathscr{S}}} + \tau n \bP_{\hat{\mathscr{S}}})
,
\qquad
    \bM = \bM_1 + n\sum_{k\in\hat{\mathscr{S}}} (\bH^{(k)} \otimes \be_k\be_k^\top) 
\end{equation}
where $\bP_{\hat{\mathscr{S}}} = \sum_{k\in\hat{\mathscr{S}}} \be_k\be_k^\top\in\R^{p\times p}$. 
Define the residual matrix $\bF$, the error matrix $\bH$, and $\bN$ by
\begin{equation}
   \bF=\bY-\bX\hbB, 
   \qquad
    \bH = \bSigma^{1/2}(\hbB - \bB),
   \qquad
\bN = (\bI_T \otimes \bX)\bM^\dagger (\bI_T \otimes \bX^\top)
\in \R^{Tn \times Tn}.
\label{eq:def F H N}
\end{equation}
To construct our estimator we also make use of the so-called interaction matrix $\hbA\in \R^{T\times T}$.
\begin{definition}[\cite{bellec2021chi}]\label{def: A}
	The \textit{interaction matrix} $\hbA\in \R^{T\times T}$ of the estimator $\hbB$ in  \eqref{eq: hbB} is defined by
	\begin{align}\label{eq: hbA-matrix}
		\hbA 
                = \sum_{i=1}^n (\bI_T \otimes \be_i^\top \bX) \bM^{\dagger} (\bI_T \otimes \bX^\top\be_i)
                = \sum_{i=1}^n (\bI_T \otimes \be_i^\top)\bN(\bI_T \otimes \be_i).
	\end{align}
	
\end{definition}

The matrix $\hbA$ was introduced in \cite{bellec2021chi}, where it is used alongside the multi-task Lasso estimator ($\tau=0$ in \eqref{eq: hbB}).
It generalizes the degrees of freedom from \citet{stein1981estimation} to the multi-task case. 
Intuitively, it captures the correlation between the residuals on different tasks \cite[Lemma F.1]{bellec2021chi}. 
Our definition of the noise covariance estimator involves $\hbA$, 
although our statistical purposes differ greatly from the confidence intervals developed in \cite{bellec2021chi}.

We are now ready to introduce our estimator $\hbS$ of the noise covariance $\bS$.
\begin{definition}[Noise covariance estimator]\label{def: hbS}
	 With $\bF=\bY-\bX\hbB$ and $\hbA$ as above, define
	 \begin{equation}\label{eq: hbS}
	 	 \hbS = (n\bI_T - \hbA)^{-1} \Bigl[\bF^\top \big( (p+n)\bI_n - \bX \bSigma^{-1}\bX^\top\big) \bF - \hbA\bF^\top\bF 
	 	- \bF^\top\bF\hbA 
	 	\Bigr](n\bI_T - \hbA)^{-1}.
	 \end{equation}
\end{definition}

Efficient solvers (\eg, \verb|sklearn.linear_model.MultiTaskElasticNet| in \cite[]{scikit-learn}) are available to compute $\hbB$.
Computation of $\bF$ is then straightforward, and computing the matrix $\hbA$
only requires inverting a matrix of size $|\hat{\mathscr{S}}|$ \cite[Section 5]{bellec2021chi}.
The estimator $\hbS$ generalizes the scalar estimator
\eqref{eq: est-bellec20} to the multi-task setting  
in the sense that for $T=1$, $\hbS$ is exactly equal to \eqref{eq: est-bellec20}.
Note that unlike in \eqref{eq: est-bellec20}, here $\bF^\top\bF$, $\hbA$ and $(n\bI_T - \hbA)$ are matrices of size $T\times T$: the order of matrix multiplication in $\hbS$ matters and should not be switched.
    This non-commutativity is not present for $T=1$ in \eqref{eq: est-bellec20} where matrices in $\R^{T\times T}$ are reduced to scalars.
Another special case of $\hbS$ can be seen in \cite{celentano2021cad} for $T=2$ where the matrix $\hbA\in\R^{2\times 2}$ is diagonal and the two columns of $\hbB\in\R^{p\times 2}$ are two Lasso or Ridge estimators computed independently of each other, one for each task. Except in these two special cases --- \eqref{eq: est-bellec20} for $T=1$, \cite{celentano2021cad} for $T=2$ and two Lasso/Ridge --- we are not aware of previously proposed estimators of the same form as $\hbS$. 

\section{Theoretical analysis}\label{sec: main-results}

\subsection{Oracle and method-of-moments estimator}

Before moving on to the theoretical analysis of $\hbS$, we state our randomness assumptions for $\bE, \bX$ and we study two preliminary estimators:
the oracle $\frac 1 n \bE^\top\bE$ and another estimator obtained by the method of moments.

\begin{assumption}[Gaussian noise]\label{assu: noise}
	$\bE\in \R^{n\times T}$ is a Gaussian noise matrix with \iid $\mathcal N_T(\bf0,\bS)$ rows, where $\bS\in \R^{T\times T}$ is an unknown positive semi-definite matrix. 
\end{assumption}
An oracle with access to the noise matrix $\bE$ may compute the oracle estimator $\hbS_{\rm{(oracle)}} \defas \frac 1n \bE^\top \bE$, with convergence rate given by the following theorem, which will serve as a benchmark.
\begin{proposition}[Convergence rate of $\hbS_{\text{(oracle)}}$] \label{prop: oracle}
Under \Cref{assu: noise},
	\begin{equation}
		\E\big[	\fnorm{\hbS_{\rm{(oracle)}}- \bS}^2 \big] = \tfrac1n [(\trace(\bS))^2 + \trace(\bS^2)].
                \label{eq: bound oracle}
	\end{equation}
	Consequently, $n^{-1} (\trace(\bS))^2 \le 
	\E\big[ \fnorm{\hbS_{\rm{(oracle)}} - \bS}^2 \big]  \le 2 n^{-1} (\trace(\bS))^2$. 
\end{proposition}
The next assumption concerns the design matrix $\bX$ with rows $\bx_1^\top,\ldots,\bx_n^\top$.
\begin{assumption}[Gaussian design]\label{assu: design}
	$\bX\in \R^{n\times p}$ is a Gaussian design matrix with \iid $\mathcal N_p(\mathbf 0,\bSigma)$ rows, where  $\bSigma$ is a known positive definite matrix. The matrices $\bE$ and $\bX$ are independent.
\end{assumption}

Under the preceding assumptions, we obtain the following method-of-moments estimator, which extends the estimator for noise variance in \cite{dicker2014variance} to the multi-task setting. Its error will also serve as a benchmark. 

\begin{proposition}\label{prop: mom}
        Under \Cref{assu: noise,assu: design}, the method-of-moments estimator defined as
	\begin{equation}
            \label{hbS_mm}
		\hbS_{\rm{(mm)}} = \frac{(n+1+p)}{n(n+1)} \bY^\top \bY - \frac{1}{n(n+1)}  \bY^\top\bX \bSigma^{-1}\bX^\top\bY 
	\end{equation}
        is unbiased for $\bS$, \ie, $\E [\hbS_{\rm{(mm)}} ] = \bS.$ Furthermore, the Frobenius error is bounded from below as
	\begin{align}
		\E [\fnorm{\hbS_{\rm{(mm)}} - \bS}^2 ] \ge \frac{p-2}{(n+1)^2} \big[\trace(\bS) + \fnorm{\bSigma^{\frac12}\bB^*}^2\big]^2. 
                \label{eq:lower-boud-mom}
	\end{align}
\end{proposition}

By \eqref{eq:lower-boud-mom}, a larger norm $\fnorm{\bSigma^{1/2}\bB^*}$ induces a larger variance for $\hbS_{\rm{(mm)}}$.
Our goal with an estimate $\hbS$, when a good estimator $\hbB$ of the nuisance is available, is to improve upon the right-hand side of \eqref{eq:lower-boud-mom}
when the estimation error $\fnorm{\bSigma^{1/2}(\hbB-\bB^*)}$ is smaller than
$\fnorm{\bSigma^{1/2}\bB^*}$. 

A high-probability upper bound of the form
$\fnorm{\hbS_{\rm{(mm)}} - \bS}^2
\le C \frac{n+p}{n^2}[\trace(\bS) + \fnorm{\bSigma^{\frac12}\bB^*}^2]^2
$,
that matches the lower bound \eqref{eq:lower-boud-mom} when $p>n$,
is a consequence of our main result below. 
Indeed, when $\hbB=\bf 0$ then $\hbA=\bf0$ and our estimator $\hbS$ from \Cref{def: hbS}
coincides with $\hbS_{\rm{(mm)}}$ up to the minor modification
of replacing $n+1$ by $n$ in \eqref{hbS_mm}.
This replacement is immaterial
compared to the right-hand side in \eqref{eq:lower-boud-mom}.
Furthermore, such $\hbS$ 
corresponds to one of $\tau$ or $\lambda$ being $+\infty$ in \eqref{eq: hbB} 
and the aforementioned upper bound
follows by taking $\tau=+\infty$ in the proof of \Cref{thm:covariance} below.
The empirical results in \Cref{sec: simulation} confirm that $\hbS$ has smaller variance compared to $\hbS_{\rm{(mm)}}$ in simulations.



\subsection{Theoretical results for proposed estimator \texorpdfstring{$\hbS$}{} }
We have established lower bounds for the oracle estimator and the
method-of-moments estimator that will serve as benchmarks. 
We turn to the analysis of the estimator $\hbS$ from \Cref{def: hbS}
under the following additional assumptions.

\begin{assumption}[High-dimensional regime]\label{assu: regime}
	$n,p$ satisfy $p/n \le \gamma$ for a constant $\gamma\in (0, \infty)$. 
\end{assumption}

For asymptotic statements such as those involving the stochastically
bounded notation $O_p(\cdot)$ or the convergence in probability in \eqref{eq: convergence proba generalization error} below, we implicitly consider a sequence
of multi-task problems indexed by $n$ where $p,T,\bB^*,\hbB,\bS$ all implicitly
depend on $n$. The Assumptions, such as $p/n\le\gamma$ above, are required to hold at
all points of the sequence. In particular, $p/n\to\gamma'$ is allowed for any limit $\gamma'\le \gamma$ under \Cref{assu: regime}, although our results do not require a specific value for the limit.

\begin{assumption}\label{assu: tau}
	Assume either one of the following:
	\begin{enumerate}[label = \roman*), topsep=0pt,itemsep=-1ex,partopsep=1ex,parsep=1ex]
            \item $\tau>0$ in the penalty of estimator \eqref{eq: hbB}, and let $\tau' = \tau/\opnorm{\bSigma}$. 
            \item $\tau=0$ and for $c>0$, $P(U_1) \ge 1-\frac 1T$ and $P(U_1)\to1$ as $n\to\infty$, where $U_1 = \{\norm{\hbB}_0 \le n(1-c)/2 \}$ is the event that $\hbB$ has at most $n(1-c)/2$ nonzero rows. Finally, $T \le e^{\sqrt{n}}$. 
	\end{enumerate}
\end{assumption}

\Cref{assu: tau}(i) requires that the Ridge penalty in \eqref{eq: hbB} be enforced,
so that the objective function is strongly convex.
\Cref{assu: tau}(ii), on the other hand, does not require strong convexity
but that the number of nonzero rows of $\hbB$ is small enough with high-probability,
which is a reasonable assumption when the tuning parameter $\lambda$ in \eqref{eq: hbB} is large enough and $\bB^*$ is sparse enough. While we do not prove
in the present paper that $\P(U_1)\to1$ under assumptions on the tuning parameter $\lambda$ and the sparsity of $\bB^*$, results of a similar nature have been obtained previously in several group-Lasso settings
\cite[Theorem 3.1]{lounici2011oracle},
\cite[Lemma 6]{liu2009estimation},
\cite[Lemma C.3]{bellec2021chi},
\cite[Proposition 3.7]{bellec2019first}.

\begin{theorem}
	\label{thm:covariance}
        Suppose that \Cref{assu: noise,assu: design,assu: regime,assu: tau} hold for all $n, p$ as $n\to\infty$, then almost surely
        \begin{equation}
	\fnorm{(\bI_T - \hbA/n) (\hbS -\bS)(\bI_T - \hbA/n)} \le \Theta_1 n^{-\frac 12} \big( \fnorm{\bF}^2/n + \fnorm{\bH}^2 +\trace(\bS)\big)
        \label{eq:thm33}
        \end{equation}
        for some non-negative random variable $\Theta_1$ of constant order, in the sense that $\E [\Theta_1^2] \le C(\tau')(T \wedge (1 + \frac pn))(1 + \frac pn))\le C(\gamma, \tau')$
         under \Cref{assu: tau}(i), 
        and ${\E [I(\Omega)\Theta_1^2]} \le C(\gamma,c)$ under \Cref{assu: tau}(ii), where $I(\Omega)$ is the indicator function of an event $\Omega$
        with $\P(\Omega)\to 1$.
\end{theorem}

Above, $\Theta_1\ge 0$ is said to be of constant order because $\Theta_1=O_P(1)$
follows from $\E[\Theta_1^2]\le C(\gamma,\tau')$ or from $\E[I(\Omega)\Theta_1^2]\le C(\gamma,c)$ if the stochastically bounded notation $O_P(1)$ is allowed to
hide constants depending on $(\gamma,\tau')$ or $(\gamma,c)$ only.
In the left-hand side of \eqref{eq:thm33}, multiplication by $\bI_T - \hbA/n$ on both sides of the error $\hbS -\bS$ can be further removed, as
\begin{equation}
\fnorm{\hbS -\bS}
\le
\fnorm{(\bI_T - \hbA/n) (\hbS -\bS)(\bI_T - \hbA/n)}
\opnorm{(\bI_T - \hbA/n)^{-1}}^2
\label{eq: ineq op norm}
\end{equation}
and the fact that $\opnorm{(\bI_T - \hbA/n)^{-1}}$ is bounded from above with high probability by a constant depending on $\gamma, \tau', c$ only. 
Upper bounds on $\opnorm{(\bI_T - \hbA/n)^{-1}}$ are formally stated in the supplementary material. 


\subsection{Understanding the right-hand side of \texorpdfstring{\eqref{eq:thm33}}{(14)}, and the multi-task generalization error}
\label{sec: gen error}

Before coming back to upper bounds on the error $\fnorm{\hbS-\bS}$,
let us study the quantities appearing in the right-hand side of \eqref{eq:thm33}. By \eqref{eq:def F H N}, $\fnorm{\bF}^2/n$
is the mean squared norm of the residuals and is observable, while the squared error
$\fnorm{\bH}^2=\fnorm{\bSigma^{1/2}(\hbB-\bB^*)}^2$ and $\trace[\bS]$ are unknown.
By analogy with single task models, we define the \textit{generalization error} as the matrix $\bH^\top\bH + \bS$ of size $T\times T$, whose $(t,t')$-th entry is
    $\E[(Y^{new}_t - \bx_{new}^T\hbB\be_t)(Y^{new}_{t'} - \bx_{new}^T\hbB\be_{t'})|(\bX,\bY)]$
    where $(Y^{new}_t,Y^{new}_{t'},\bx_{new})$ is independent of $(\bX,\bY)$
    and has the same distribution
    as $(Y_{it},Y_{it'},\bx_i)$ for some $i=1,...,n$.
    Estimating the generalization error is useful for parameter tuning:
    since 
    \begin{equation}
     \trace[\bH^T\bH + \bS] = \fnorm{\bSigma^{1/2}(\hbB-\bB^*)}^2 + \trace[\bS],  
     \label{eq:trace-gen}
    \end{equation}
    minimizing an estimator of $\trace[\bH^T\bH + \bS]$ is a useful proxy
    to minimize the Frobenius error $\fnorm{\bSigma^{1/2}(\hbB-\bB^*)}^2$
    of $\hbB$.
The following theorem gives an estimate for the generalization error matrix
as well as a consistent estimator for its trace \eqref{eq:trace-gen}.
\begin{theorem}[Generalization error]
	\label{thm:generalization error}
        Let \Cref{assu: noise,assu: tau,assu: design,assu: regime} be fulfilled. Then
	\begin{align*}
		\fnorm{ \bF^\top\bF/n - (\bI_T - \hbA/n) (\bH^\top\bH + \bS) (\bI_T - \hbA/n)} \le \Theta_2 n^{-\frac 12} \big(\fnorm*{\bF}^2/n + \fnorm{\bH}^2 + \trace(\bS)\big),
	\end{align*}
for some non-negative random variable $\Theta_2$ of constant order, in the sense that $\E [\Theta_2] \le C(\gamma,\tau')$ under \Cref{assu: tau}(i), 
and with ${\E [I(\Omega)\Theta_2]} \le C(\gamma,c)$ under \Cref{assu: tau}(ii),  where $I(\Omega)$ is the indicator function of an event $\Omega$
with $\P(\Omega)\to 1$.

Furthermore, if $T = o(n)$ as $n, p\to \infty$ while $\tau', \gamma, c$ stay constant,
 then 
 \begin{equation}
	\frac{\trace(\bS) + \fnorm{\bH}^2}{\fnorm{(\bI_T - \hbA/n)^{-1}\bF^\top}^2/n} \overset{p}{\to} 1.
        \label{eq: convergence proba generalization error}
 \end{equation}
\end{theorem}
In the above theorem, $\bS$ and $\bH$ are unknown, while $\hbA$ and $\bF$ can be computed from the observed data $(\bX, \bY)$.
Thus \eqref{eq: convergence proba generalization error} shows that $\fnorm{(\bI_T - \hbA/n)^{-1}\bF^\top}^2/n$ is a consistent estimate for 
the unobserved quantity $\trace(\bS) + \fnorm{\bH}^2$.

\subsection{Back to bounds on \texorpdfstring{$\fnorm{\hbS-\bS}$}{estimation error}}
We are now ready to present our main result on the error bounds for $\hbS$.
It is a consequence of \eqref{eq:thm33}, \eqref{eq: ineq op norm} and \eqref{eq: convergence proba generalization error}.
\begin{theorem}
	\label{thm:main}
     Let \Cref{assu: noise,assu: tau,assu: design,assu: regime} be fulfilled and $T = o(n)$. Then
	\begin{align}
		\fnorm{\hbS - \bS} &\le O_P(n^{-\frac 12}) (\fnorm*{\bF}^2/n),\\
		\fnorm{\hbS - \bS} &\le O_P(n^{-\frac 12}) [\trace(\bS) + \fnorm{\bH}^2].
                \label{eq:upper bound trS+H2}
	\end{align}
        Here the $O_P(n^{-\frac12})$ notation involves constants depending on $\gamma,\tau',c$.
\end{theorem}
It is instructive at this point to compare \eqref{eq:upper bound trS+H2}
with the lower bound \eqref{eq:lower-boud-mom} on the Frobenius error
of the method-of-moments estimator. When $p\ge n$ then $\E[\fnorm{\hbS_{\rm{(mm)}}-\bS}^2]\ge \frac{c}{n}[\trace[\bS] + \fnorm{\bSigma^{1/2}\bB^*}^2]^2$;
this is the situation where the Statistician does not attempt to estimate
$\bB^*$, and pays a price of $[\trace[\bS] + \fnorm{\bSigma^{1/2}\bB^*}^2]^2/n$.
On the other hand, by definition of $\bH$ in \eqref{eq:def F H N}, the right-hand side of \eqref{eq:upper bound trS+H2},
when squared, is of order
$n^{-1}[\trace[\bS] + \fnorm{\bSigma^{1/2}(\hbB-\bB^*)}^2]^2$.
Here the error bound only depends on $\bB^*$ through
the estimation error for the nuisance $\fnorm{\bSigma^{1/2}(\hbB-\bB^*)}^2$.
This explains that when $\hbB$ is a good estimator of $\bB^*$
and $\fnorm{\bSigma^{1/2}(\hbB-\bB^*)}^2$ is smaller compared to $\fnorm{\bSigma^{1/2}\bB^*}^2$, the estimator $\hbS$ that leverages $\hbB$ will outperform
the method-of-moments estimator $\hbS_{\rm{(mm)}}$ which does not attempt to estimate the nuisance. 

Finally, the next results show that under additional assumptions,
the estimator $\hbS$ enjoys Frobenius error bounds similar to the oracle estimator $\frac1n\bE^\top\bE$.

\begin{assumption}\label{assu: SNR}
	$\text{SNR }\le \mathfrak{snr}$ for some positive constant $\mathfrak{snr}$ independent of $n, p, T$, where  $\text{SNR}=  \fnorm{\bSigma^{\frac12}\bB^*}^2/\trace(\bS)$ denotes the signal-to-noise ratio of the multi-task linear model \eqref{eq: model}.
\end{assumption}

\begin{corollary}\label{cor36}
     Suppose that Assumptions~\ref{assu: noise}, \ref{assu: design}, \ref{assu: regime},  \ref{assu: tau}(i), \ref{assu: SNR} and $T=o(n)$ hold, then
	\begin{equation}
		\fnorm{\hbS - \bS} \le O_P(n^{-\frac 12}) \trace(\bS),
	\end{equation}
        where $O_P(\cdot)$ hides constants depending on $\gamma,\tau',\mathfrak{snr}$.
	Furthermore,  
	\begin{align*}
		&\fnorm{\hbS - \bS}^2 \le O_P(T/n) \fnorm{\bS}^2 = o_P(1)\fnorm{\bS}^2,\\
		&\big|\norm{\hbS}_* - \trace(\bS)\big| \le O_P(\sqrt{T/n}) \trace(\bS) = o_P(1) \trace(\bS). 
	\end{align*}
\end{corollary}

\begin{corollary}\label{cor37}
	Suppose that Assumptions~\ref{assu: noise}, \ref{assu: design}, \ref{assu: regime}, \ref{assu: tau}(ii)
        and $T = o(n)$ hold.
        If $\norm{\bB^*}_0\le (1-c)n/2$ and the tuning parameter $\lambda$ is of the form $\lambda=\mu\sqrt{\trace(\bS)/n}$ for some positive constant $\mu$, 
        then
	\begin{equation}
		\fnorm{\hbS - \bS} \le O_P(n^{-\frac12}) (1 + \mu^2) \trace(\bS), 
                \label{eq: cor37}
	\end{equation}
where $O_P(\cdot)$ hides constants depending on $c,\gamma, \phi_{\min}(\bSigma)$.

\end{corollary}

Comparing \Cref{cor36,cor37} with \Cref{prop: oracle}, we conclude
that $\fnorm{\hbS-\bS}^2$ is of the same order as the Frobenius error of the oracle estimator
in \eqref{eq: bound oracle} up to constants depending on the signal-to-noise ratio, $\gamma$, and $\tau'$ under \Cref{assu: tau}(i), and up to constants depending on $\mu$, $c, \gamma, \phi_{\min}(\bSigma)$
under \Cref{assu: tau}(ii).

The error bounds in \eqref{eq:upper bound trS+H2}-\eqref{eq: cor37} are measured in Frobenius norm, similarly to existing works on 
noise covariance estimation \cite{molstad2019new}. 
Outside the context of linear regression models, much work has been devoted to covariance estimation 
in the operator norm. 
By the loose bound $\opnorm*{\bM}\leq \fnorm*{\bM}$, our upper bounds carry over to the operator norm.
The same cannot be said for lower bounds, since for instance
$\E\big[\opnorm{\hbS_{\rm{(oracle)}}- \bS}^2 \big] \asymp n^{-1} \opnorm{\bS} \trace(\bS)$ (see, \eg, \cite[Corollary 2]{koltchinskii2017concentration}).

\begin{figure}[h]
	\centering
	\subfloat[\centering Boxplot for estimating a \emph{full rank} $\bS$ ]{{\includegraphics[width=.5\textwidth]{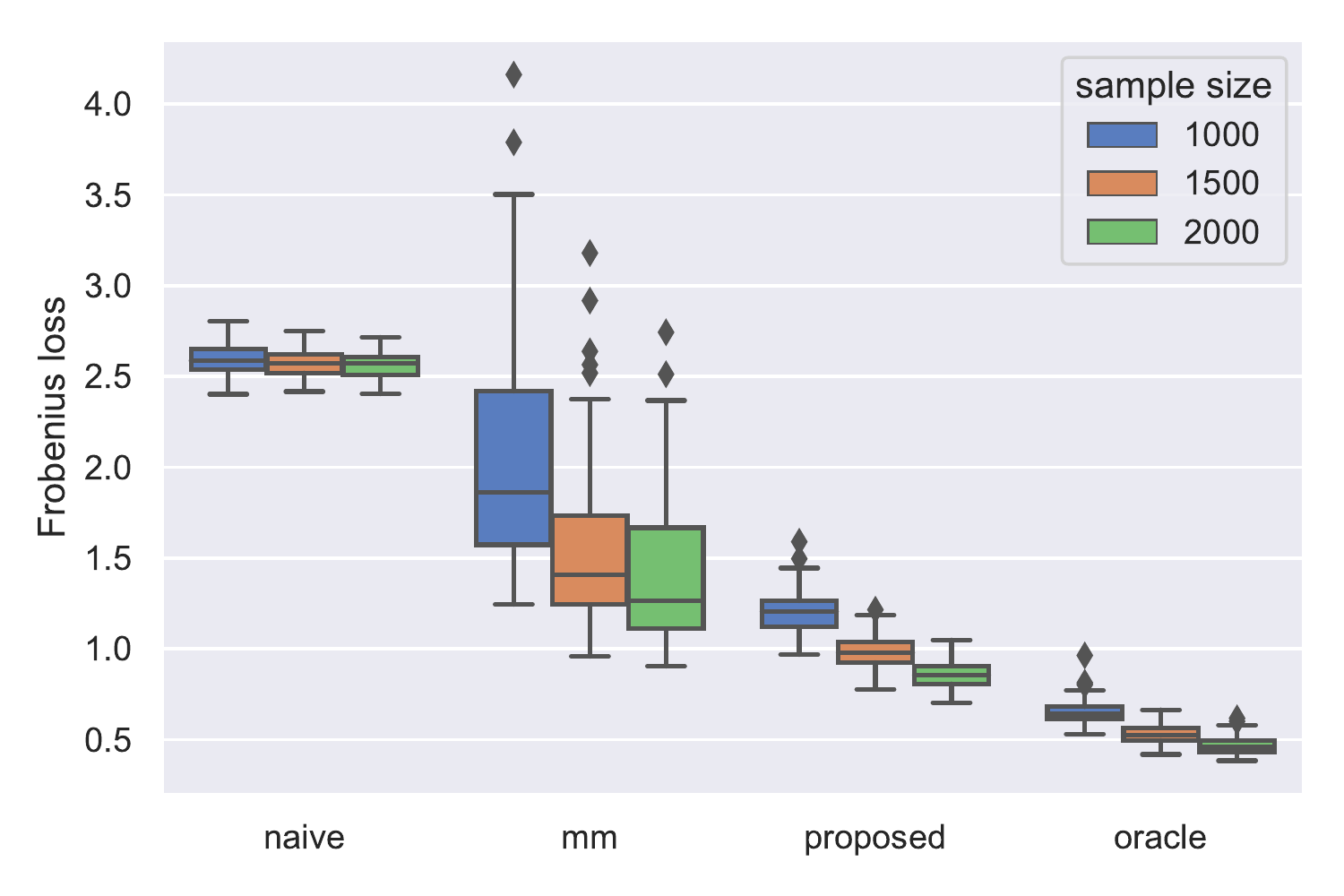} }}%
	~
	\subfloat[\centering Boxplot for estimating a \emph{low rank} $\bS$ ]{{\includegraphics[width=.5\textwidth]{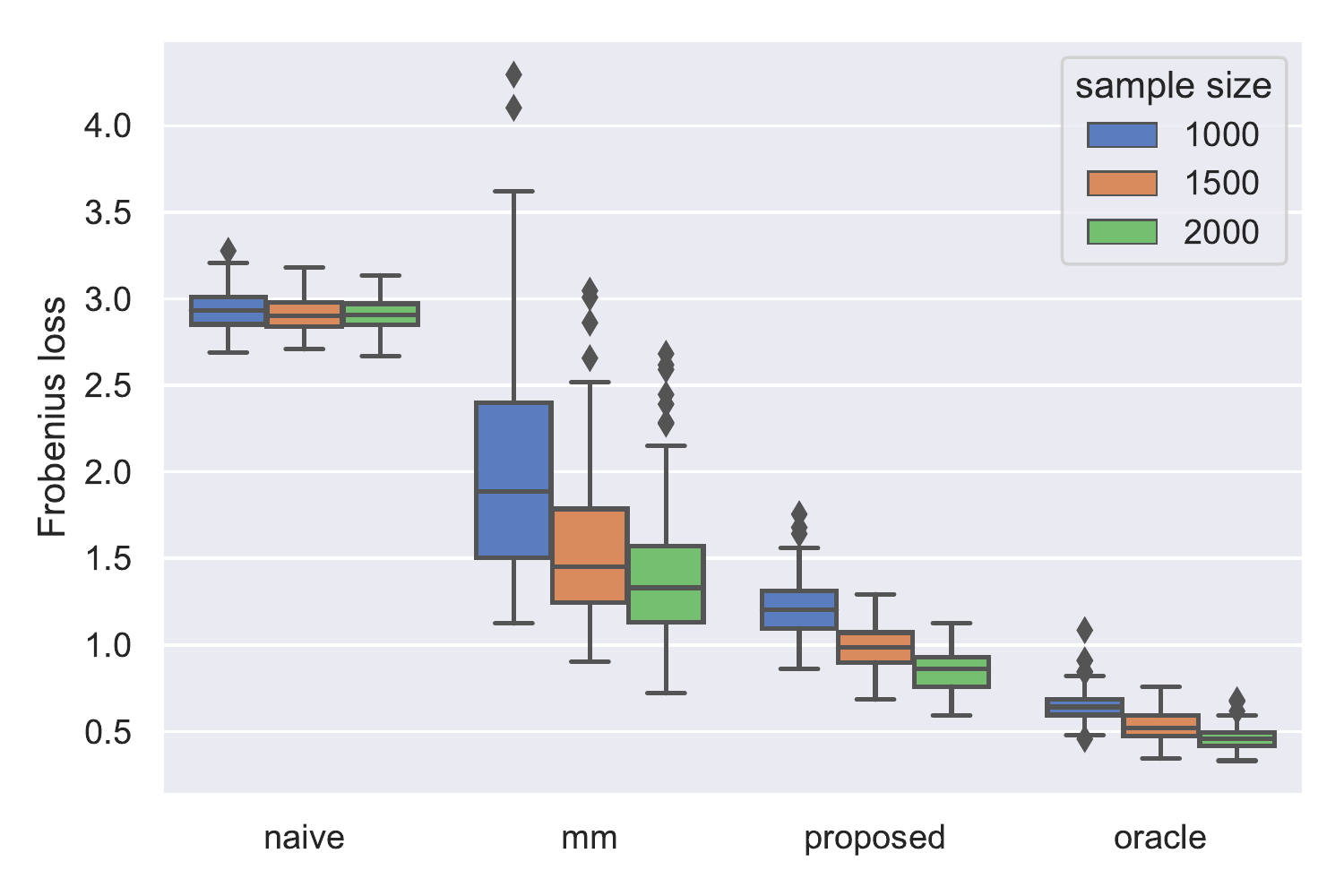} }}%
	\caption{Boxplots for Frobenius norm loss over 100 repetitions.}%
	\label{Fig:0}%
\end{figure}

\section{Numerical experiments} \label{sec: simulation}

Regarding parameters for our simulations, 
we set $T=20$, $p=1.5n$ and $n$ equals successively $1000,1500,2000$. 
We consider two settings for the noise covariance matrix:
$\bS$ is full-rank and $\bS$ is low-rank.
The complete construction of $\bS$, $\bB^*$ and $\bX$,
as well as implementation details are given in the supplementary material. 

We compare our proposed estimator $\hbS$ with relevant estimators including 
(1) the naive estimate $\hbS_{\text{(naive)}} = n^{-1}\bF^\top\bF$, 
(2) the method-of-moments estimate $\hbS_{\rm{(mm)}}$ defined in Proposition \ref{prop: mom}, and 
(3) the oracle estimate $\hbS_{\rm{(oracle)}} = n^{-1}\bE^\top\bE$. 
The performance of each estimator is measured in Frobenius norm: for instance,  $\fnorm{\hbS-\bS}$ is the loss for proposed estimator $\hbS$. 
\Cref{Fig:0} displays the boxplots of the Frobenius loss from the different methods over 100 repetitions. 
\Cref{Fig:0} shows that, besides the oracle estimator, our proposed estimator has the best performance with significantly smaller loss  compared to the naive and method-of-moments estimators. 

Since the estimation target is a $T\times T$ matrix, we also want to compare different estimators in terms of the bias and standard deviation for each entry of $\bS$.
\Cref{Fig:1} presents the heatmaps of bias and standard deviation from different estimators for full rank $\bS$ with $n=1000$. 
The remaining heatmaps for different $n$ and for estimation of low rank $\bS$ are available in the supplementary material.

As expected, the oracle estimator has best performance in \Cref{Fig:0}
and smallest bias and variance in \Cref{Fig:1}.
The naive estimator has large bias as we see in \Cref{Fig:1},
 though it has small standard deviation. 
The method-of-moments estimator is unbiased but its variance is relatively large, 
which means its performance is not stable, as was reflected in \Cref{Fig:0}. 
Our proposed estimator improves on both the naive and method-of-moments estimators 
because it has much smaller bias than the former, 
while having smaller standard deviation than the latter.

\begin{figure}[H]
	\centering
	\subfloat[\centering Heatmap of \emph{bias} for each entry ]{{\includegraphics[width=.48\textwidth]{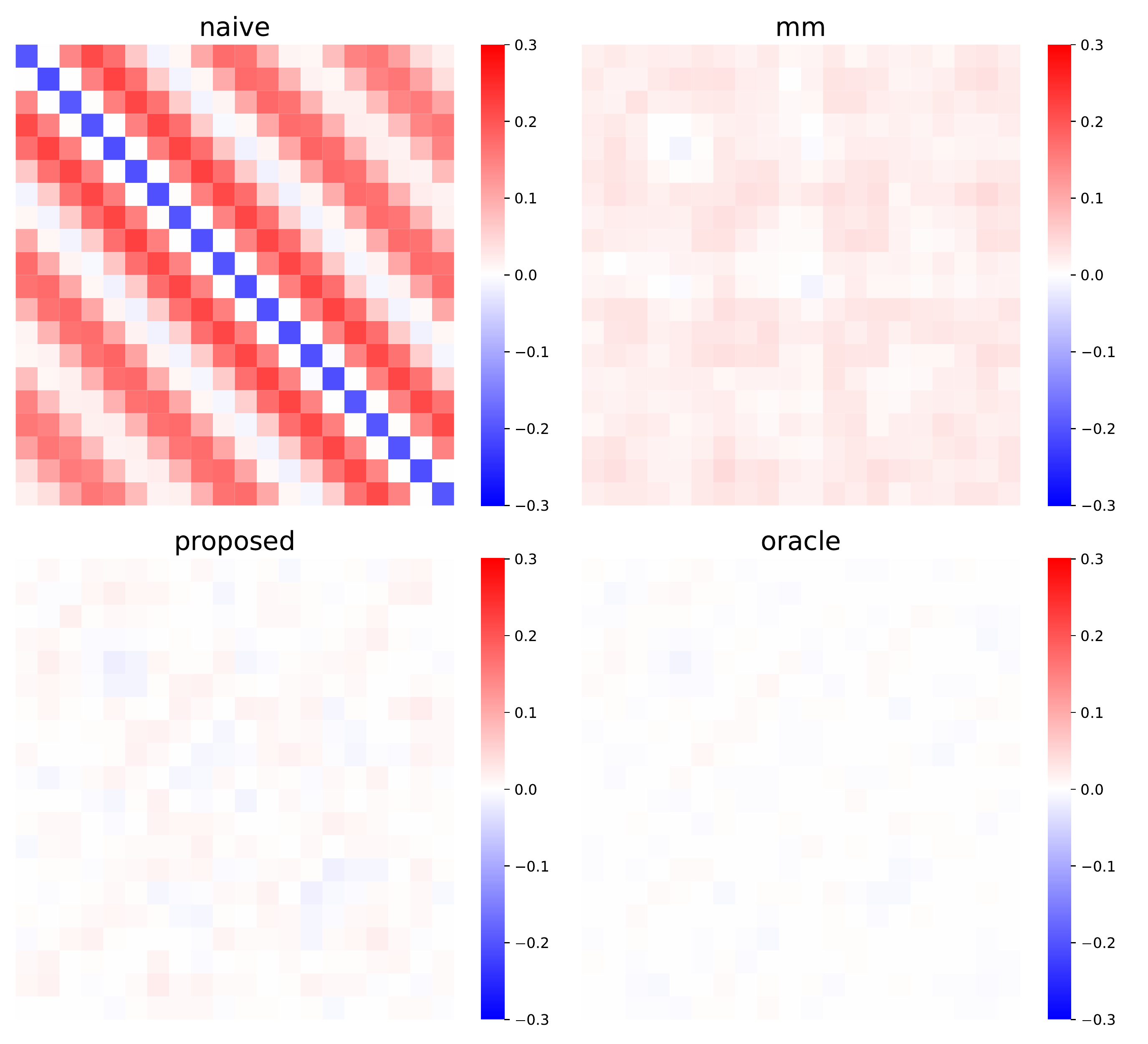} }}%
	~
	\subfloat[\centering Heatmap of \emph{standard deviation} for each entry ]{{\includegraphics[width=.48\textwidth]{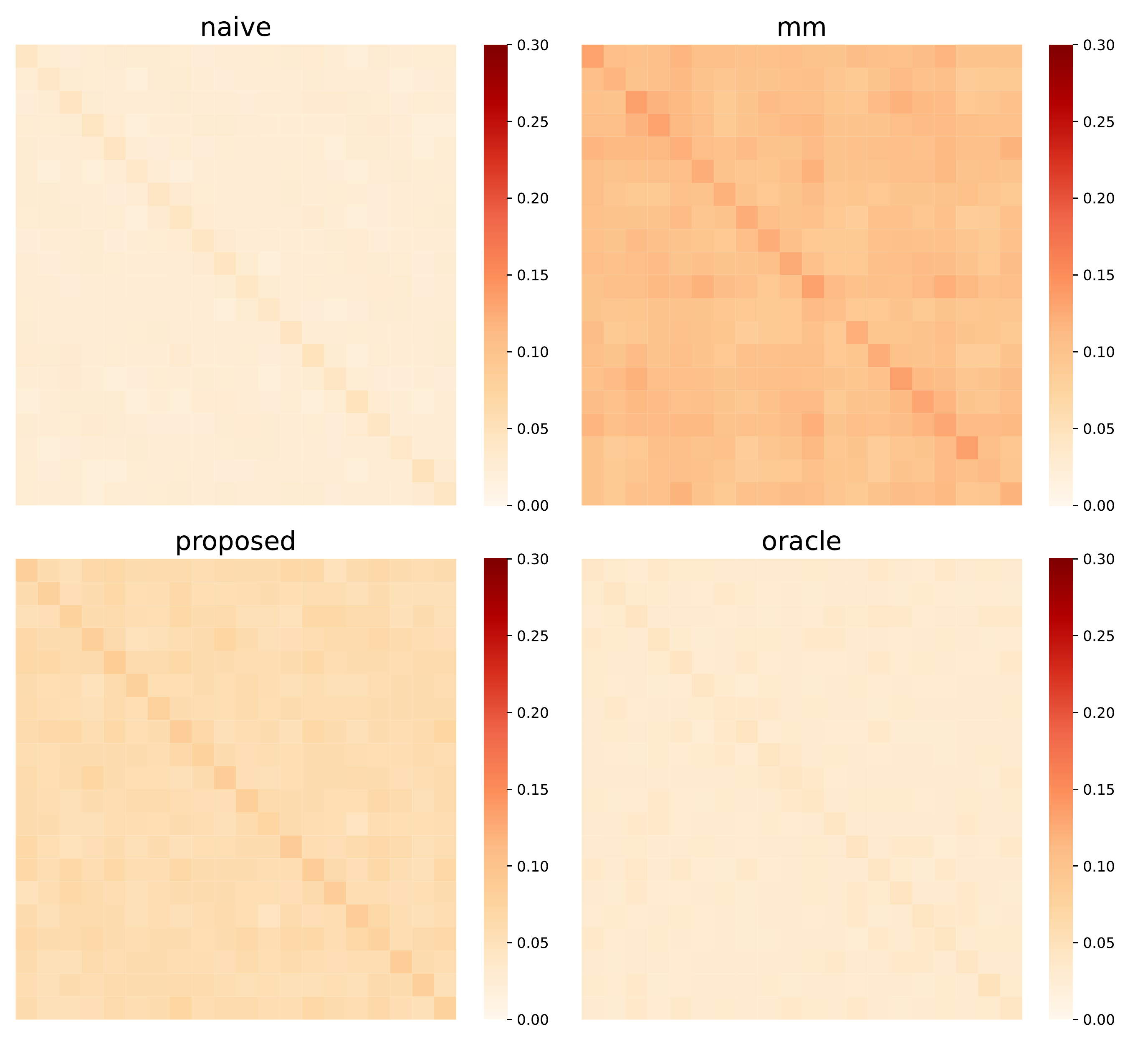} }}%
	\caption{Heatmaps for estimation of full rank $\bS$ with $n=1000$ over 100 repetitions.}%
	\label{Fig:1}%
\end{figure}

\section{Limitations and future work}\label{sec: discussion}
One limitation of the proposed estimator $\hbS$ is 
that its construction necessitates the knowledge of $\bSigma$.
Let us first mention that the estimator
$n^{-1}\fnorm{(\bI_T - \hbA/n)^{-1}\bF^\top}$ of $\trace(\bS) + \fnorm{\bSigma^{1/2}(\hbB-\bB^*)}^2$ in \Cref{thm:generalization error} does not require knowing $\bSigma$. Thus, this estimator can further be used as a proxy of the error $\fnorm{\bSigma^{1/2}(\hbB-\bB^*)}^2$,
say for parameter tuning, without the knowledge of $\bSigma$. 
The problem of estimating $\bS$ with known
$\bSigma$ was studied in \cite{celentano2021cad} for $T=2$: in this 
inaccurate covariate model
and for $p/n\le \gamma$, our results yield the convergence rate $n^{-1/2}$
for $\bS$ which improves upon the rate $n^{-c_0}$ for a non-explicit constant $c_0>0$ in \cite[Theorem 2.1]{celentano2021cad}.

In order to use $\hbS$ when $\bSigma$ is unknown, one may plug-in
an estimator $\hbSigma$ in \Cref{eq: hbS}, resulting in an extra term
of order $\opnorm{\hbSigma{}^{-1} - \bSigma^{-1}}\fnorm{\bF}$ for the Frobenius error. See \cite[\S4]{dicker2014variance} for related discussions in the $T=1$ (single-task) case. While, under the proportional regime $p/n\to \gamma$, 
no estimator is consistent for all covariance matrices $\bSigma$ in operator norm, consistent estimators do exist under additional structural assumptions \cite{bickel2008covariance, el2008operator, cai2010optimal}.
If available, additional unlabeled samples $(\bx_i)_{i\ge n+1}$ can also
be used to construct norm-consistent estimator of $\bSigma$.

Future directions include extending estimator $\hbS$ to utilize 
other estimators of the nuisance $\bB^*$
than the multi-task elastic-net \eqref{eq: hbB}; for instance 
\eqref{eq: multi-ScaledLasso} or the estimators studied in \cite{van2016chi,molstad2019new,Salmon2019handlings}. In the simpler case where
columns of $\bB^*$ are estimated independently on each task, e.g.,
if the $T$ columns of $\hbB$ are Lasso estimators 
$(\hbbeta^{(t)})_{t\in[T]}$ each computed from $\by{}^{(t)}$,
then minor modifications of our proof yield that the estimator \eqref{eq: hbS}
with $\hbA=\diag(\|\hbbeta^{(1)}\|_0,...,\|\hbbeta^{(T)}\|_0)$ enjoys similar
Frobenius norm bounds of order $n^{-1/2}$.


%


\bibliography{paper}
\bibliographystyle{plainnat}

\appendix
\section*{SUPPLEMENT}
This supplement is organized as follows:
\begin{itemize}
	\item In \Cref{sec:appendix-experiments} we provide details regarding the setting of our simulations, as well as additional experiment results.
	
	\item In \Cref{sec:out-of-sample} we establish an upper bound for the out-of-sample error, which we could not put in the full paper due to page length limit. 
	
	\item \Cref{sec:opnorm-bound} provides the upper bound for $\opnorm{(\bI_T - \hbA/n)^{-1}}$ mentioned after \Cref{thm:covariance} in the full paper, and 
	Appendices \ref{sec:Lipschitz-fix-E}, \ref{sec:Lipschitz-fix-X}, \ref{sec:proba-tools} contain preliminary theoretical statements, which will be useful for proving our main results in the full paper. 

	\item \Cref{sec: proof-thms} contains proofs of main results in \Cref{sec: main-results} of the full paper and \Cref{sec:out-of-sample}. 

	\item \Cref{sec:appendix-preliminary} contains proofs of preliminary results in Appendices \ref{sec:opnorm-bound} to \ref{sec:proba-tools}.
\end{itemize}

\textbf{Notation.} 
Here we introduce basic notations that will be used throughout this supplement. 
We use indexes $i$ and $l$ only to loop or sum over $[n] =  \{1, 2, \ldots, n\}$, use $j$ and $k$ only to loop or sum over $[p] =  \{1, 2, \ldots, p\}$, 
use $t$ and $t'$ only to loop or sum over $[T] =  \{1, 2, \ldots, T\}$, so that $\be_i$ (and $\be_l$) refer to the $i$-th (and $l$-th) canonical basis vector in $\R^n$, $\be_j$ (and $\be_k$) refer to the $j$-th (and $k$-th) canonical basis vector in $\R^p$, $\be_t$ (and $\be_{t'}$) refer to the $t$-th (and $t'$-th) canonical basis vector in $\R^T$. 
For any two real numbers $a$ and $b$, let $a\vee b = \max(a,b)$, and $a\wedge b = \min(a,b)$. 
Positive constants that depend on $\gamma, \tau'$ only are denoted by $C(\gamma, \tau')$, and positive constants that depend on $\gamma, c$ only are denoted by $C(\gamma, c)$. The values of these constants
may vary from place to place.

\section{Experiment details and additional simulation results}
\label{sec:appendix-experiments}

\subsection{Experimental details}
\label{sec:appendix-experiment-details}
This section provides more experimental detail for \Cref{sec: simulation} of the full paper. 

We consider two types of noise covariance matrix: 
(i) $\bS$ is full-rank with $(t,t')$-th entry $\bS_{t,t'} = \frac{\cos(t-t')}{1 + \sqrt{|t-t'|}}$; 
(ii) $\bS$ is low-rank with $\bS = \bu\bu^\top$, where $\bu\in \R^{T\times 10}$ has \iid entries from $\calN(0, 1/T)$. 

To build the coefficient matrix $\bB^*$, we first set its sparsity pattern, \ie, we define the support $\mathscr{S}$ of cardinality $|\mathscr{S}| = 0.1 p$,
then we generate an intermediate matrix $\bB\in \R^{p\times T}$.
The $j$-th row of $\bB$ is sampled from $\calN_T(\mathbold0, p^{-1}\bI_T)$ if $j\in \mathscr{S}$,
otherwise we set it to be the zero vector. 
Finally we let $\bB^* =\bB [\trace(\bS)/ \trace(\bB^\top\bSigma \bB)]^{\frac 12}$,
which forces a signal-to-noise ratio of exactly $1$.

The design matrix $\bX$ is constructed by independently sampling its rows from $\calN_p(\mathbold 0, \bSigma)$ with $\bSigma_{jk} = |j-k|^{0.5}$. 

The Python library Scikit-learn \cite{scikit-learn} is used to calculate $\hbB$ in \eqref{eq: hbB}. 
More precisely we invoke \texttt{MultiTaskElasticNetCV} to
obtain $\hbB$ by 5-fold cross-validation with parameters \texttt{l1-ratio=[0.5, 0.7, 0.9, 1]}, \texttt{n\_alpha=100}. 
To compute the interaction matrix $\hbA$ 
we used the efficient implementation described in \cite[Section 5]{bellec2021chi}. 

The full code needed to reproduce our experiments is part of the supplementary material.
A detailed Readme file is located in the corresponding folder.

The simulations results reported in the full paper and this supplementary material were run on a cluster of 50 CPU-cores (each is an Intel Xeon E5-2680 v4 @2.40GHz) equipped with a total of 150 GB of RAM. 
It takes approximately six hours to obtain all of our simulation results.

\subsection{Numerical results of Frobenius norm loss}

While \Cref{Fig:0} in the full paper provides boxplots of Frobenius norm loss for 100 repetitions,
we report in following \Cref{table:t1} the mean and standard deviation of the Frobenius norm loss over 100 repetitions. 
\begin{table}[H]
	\caption{Comparison of different methods for estimation of $\bS$}
	\label{table:t1}
	\centering
	\begin{tabular}{cc|cc|cc|cc}
		\toprule
		& & \multicolumn{2}{c|}{$n=1000$} & \multicolumn{2}{c|}{$n=1500$}  & \multicolumn{2}{c}{$n=2000$} \\
		\hline
		$\bS$ &method & mean & sd & mean & sd & mean & sd \\
		\midrule
		\multirow{4}*{full rank}
		&naive  			&2.593 &0.090   & 2.572&0.076& 2.562 &0.070 \\
		&mm     			& 2.030 &0.616   & 1.554 &0.421  & 1.413 &0.405\\
		&proposed	& 1.207 &0.119   & 0.984 &0.084  & 0.858 &0.072 \\
		&oracle 			& 0.652 &0.061  & 0.534 &0.052  & 0.469 &0.045\\
		\hline
		\multirow{4}*{low rank}
		&naive  			&2.942 &0.119   & 2.912&0.102& 2.908 &0.094 \\
		&mm     			& 2.027 &0.686   & 1.561 &0.435  & 1.423 &0.414\\
		&proposed	& 1.216 &0.172   & 0.989 &0.125  & 0.854 &0.118 \\
		&oracle 			& 0.654 &0.096  & 0.531 &0.081  & 0.464 &0.065\\
		\bottomrule
	\end{tabular}
\end{table}
The numerical results in  \Cref{table:t1} are consistent with the boxplots in \Cref{Fig:0}. 
It is clear from \Cref{table:t1} that our proposed estimator has significantly smaller loss than the  naive estimator and method-of-moments estimator. 
These comparisons again show the superior performance of our proposed estimator. 

\newpage 
\subsection{Additional heatmaps for estimating full rank \texorpdfstring{$\bS$}{S}}
When estimating the full rank $\bS$ with $(t,t')$-th entry $\bS_{t,t'} = \frac{\cos(t-t')}{1 + \sqrt{|t-t'|}}$, the heatmaps for different estimators from $n=1500$ and $n=2000$ are presented in \Cref{Fig:2,Fig:3}, respectively. 
The comparison patterns in \Cref{Fig:2,Fig:3} are similar to the case $n=1000$ in \Cref{Fig:1} of the full paper; our proposed estimator outperforms the naive estimator and method-of-moments estimator.
\begin{figure}[H]
	\centering
	\subfloat[\centering Heatmap of \emph{bias} for each entry ]{{\includegraphics[width=.48\textwidth]{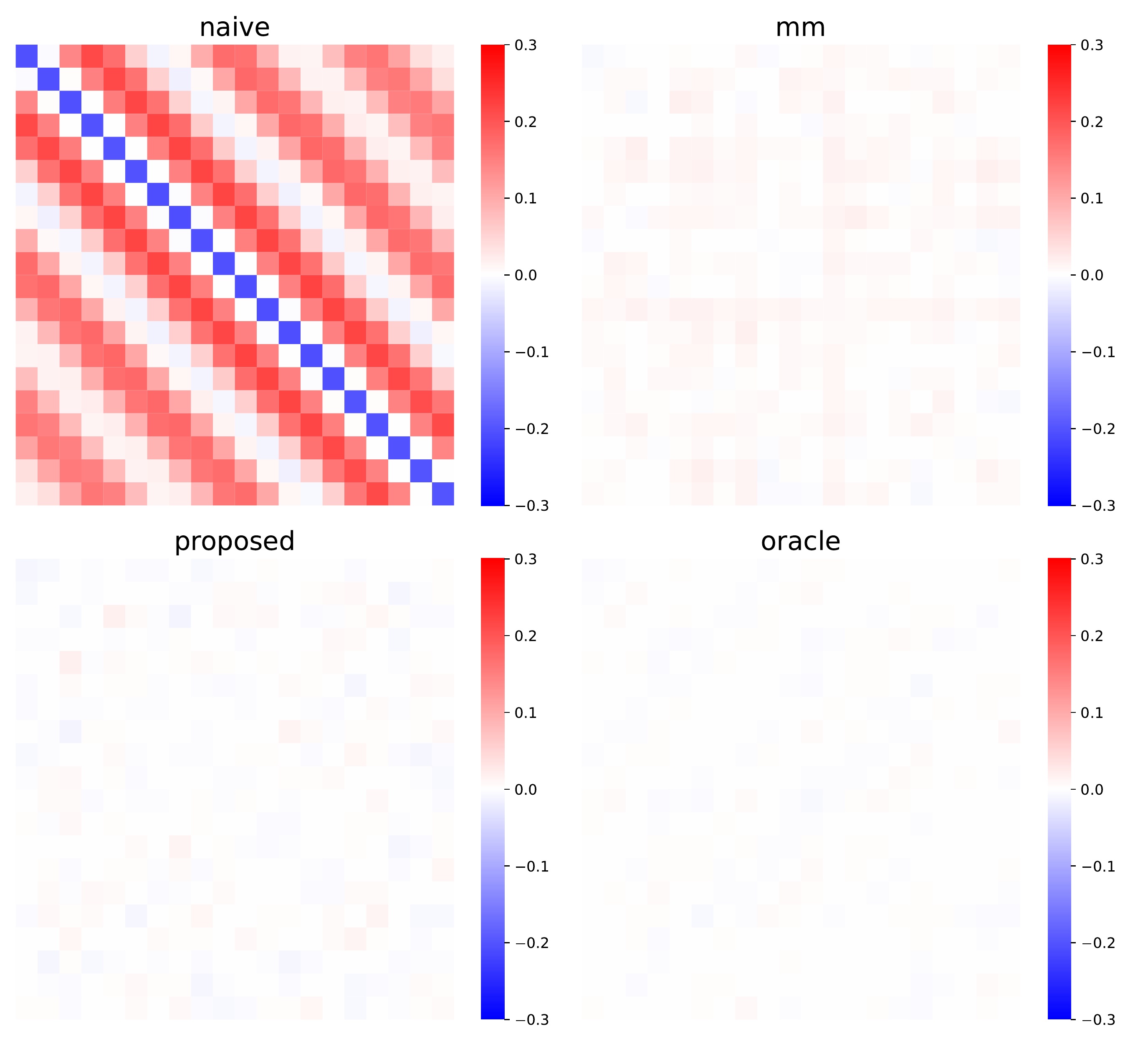} }}%
	~
	\subfloat[\centering Heatmap of \emph{standard deviation} in each entry ]{{\includegraphics[width=.48\textwidth]{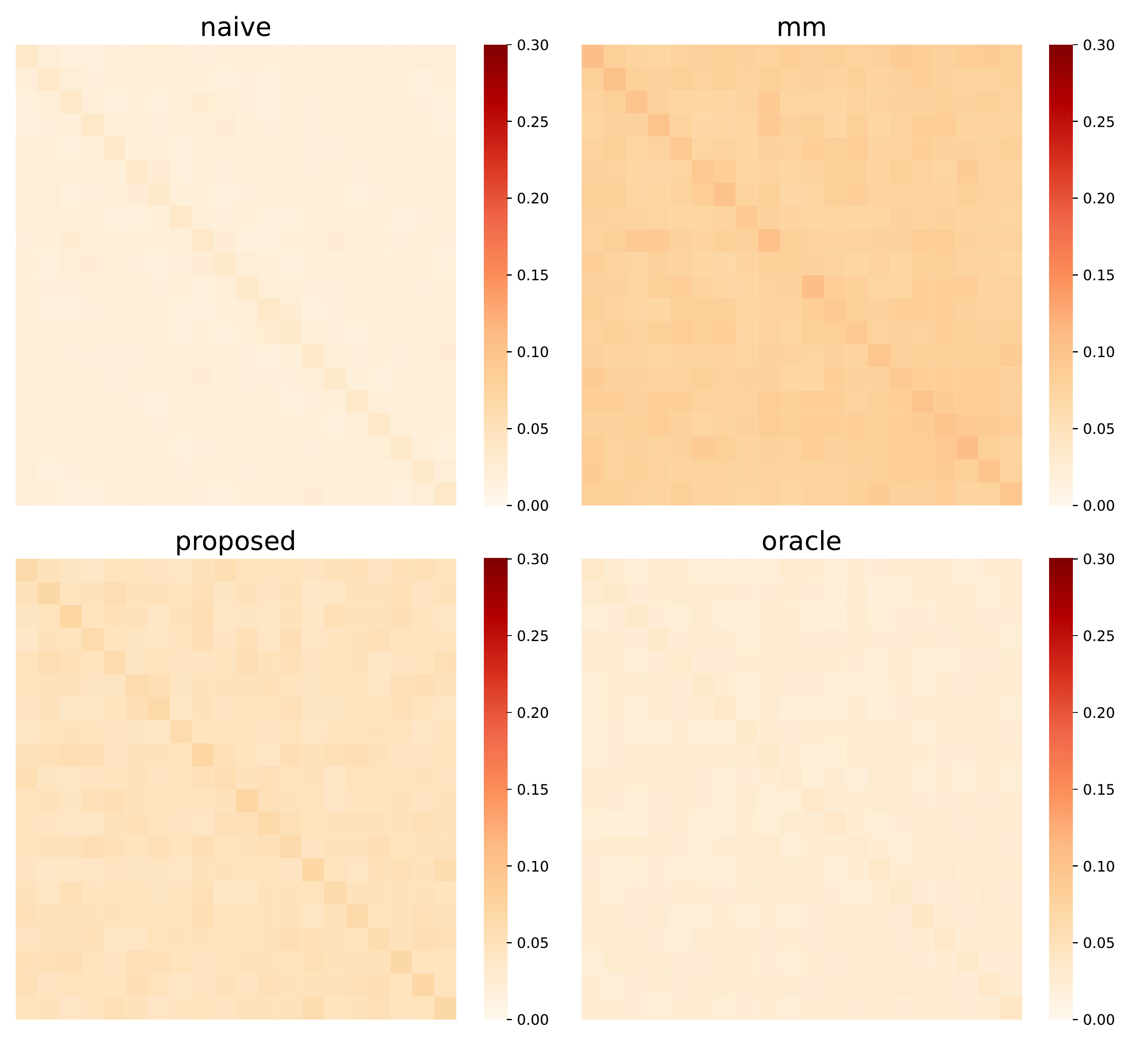} }}%
	\caption{Heatmaps for estimation of full rank $\bS$ with $n=1500$ over 100 repetitions.}%
	\label{Fig:2}%
\end{figure}

\begin{figure}[H]
	\centering
	\subfloat[\centering Heatmap of \emph{bias} for each entry ]{{\includegraphics[width=.48\textwidth]{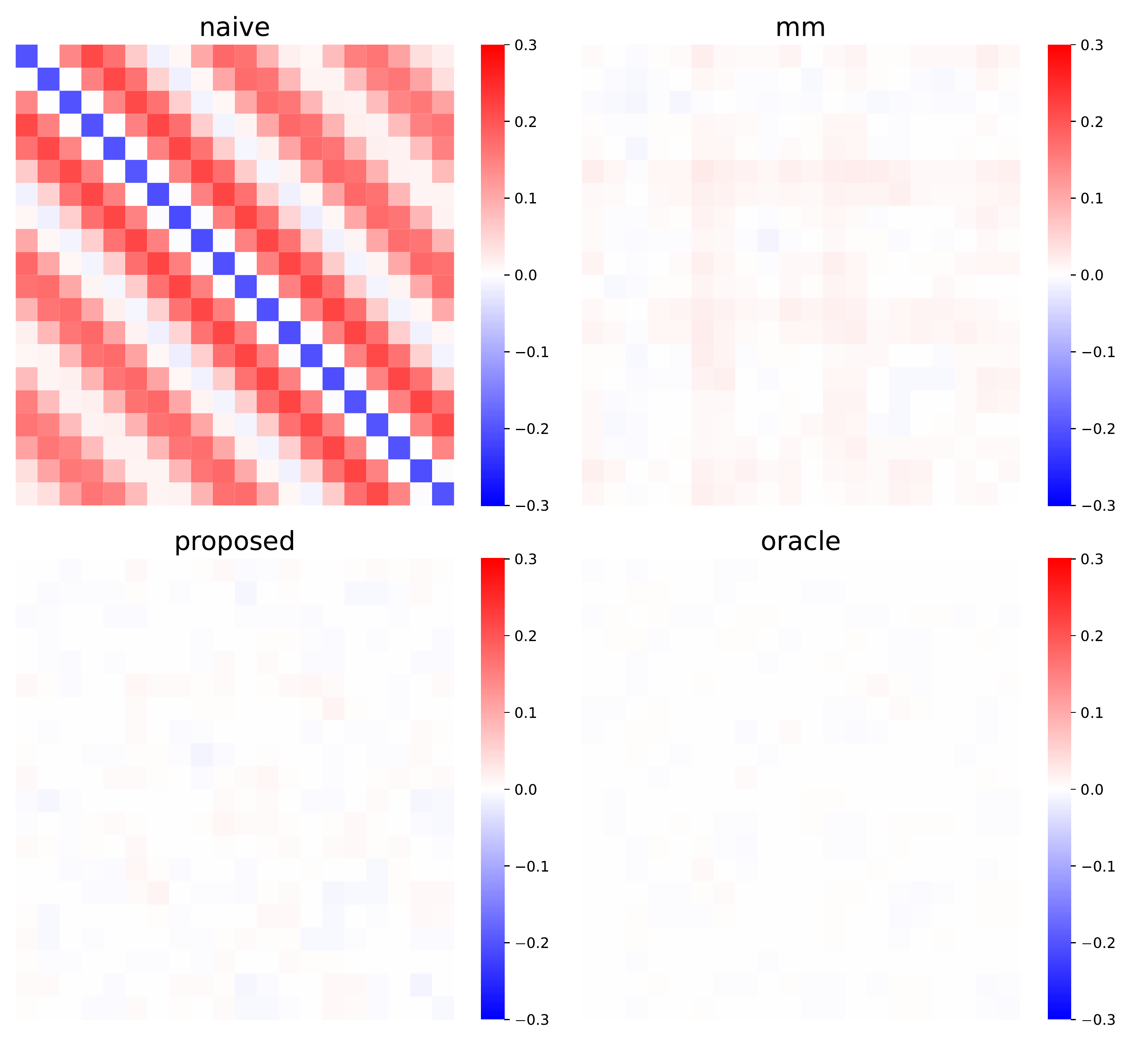} }}%
	~
	\subfloat[\centering Heatmap of \emph{standard deviation} in each entry ]{{\includegraphics[width=.48\textwidth]{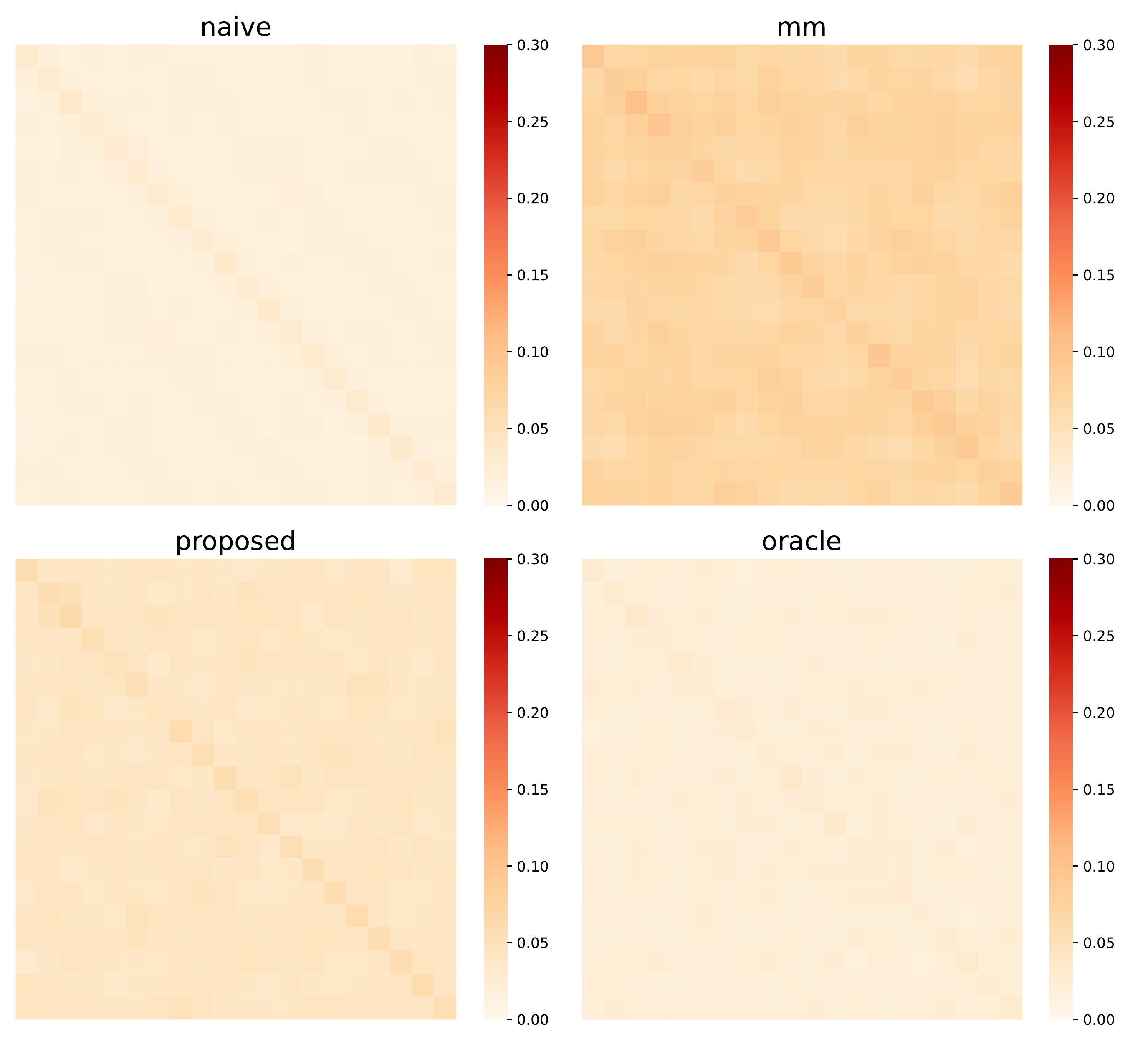} }}%
	\caption{Heatmaps for estimation of full rank $\bS$ with $n=2000$ over 100 repetitions.}%
	\label{Fig:3}%
\end{figure}
\newpage

\subsection{Heatmaps for estimating low rank \texorpdfstring{$\bS$}{S}}

When estimating the low rank with $\bS = \bu\bu^\top$, and $\bu\in \R^{T\times 10}$ with entries are \iid from $\calN(0, 1/T)$. 
We present the heatmaps for different estimators with $n=1000, 1500, 2000$ in \Cref{Fig:4,Fig:5,Fig:6} below. All of these figures convince us that besides the oracle estimator, the proposed estimator has the best performance. 
\begin{figure}[H]
	\centering
	\subfloat[\centering Heatmap of \emph{bias} for each entry ]{{\includegraphics[width=.48\textwidth]{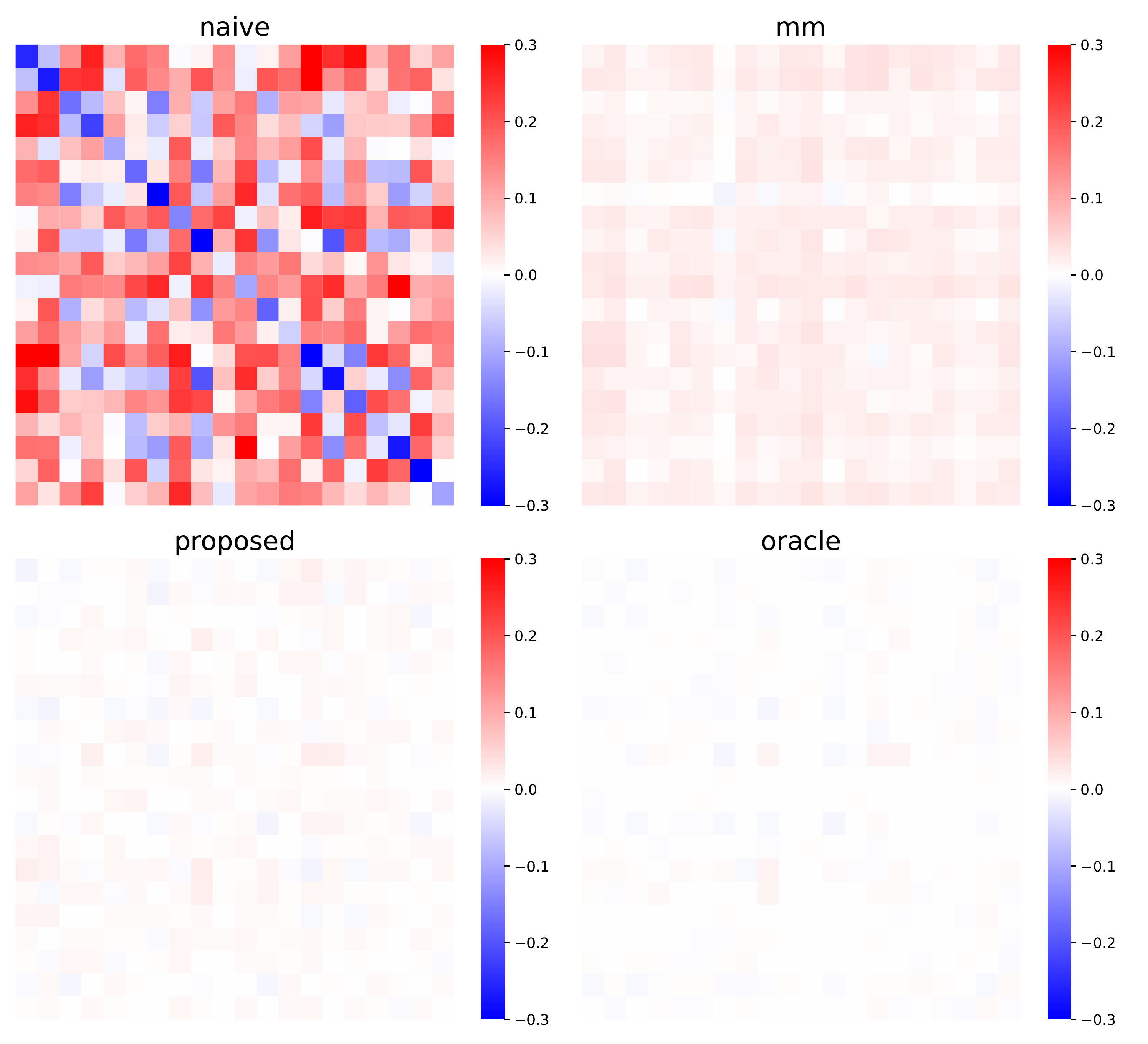} }}%
	~
	\subfloat[\centering Heatmap of \emph{standard deviation} in each entry ]{{\includegraphics[width=.48\textwidth]{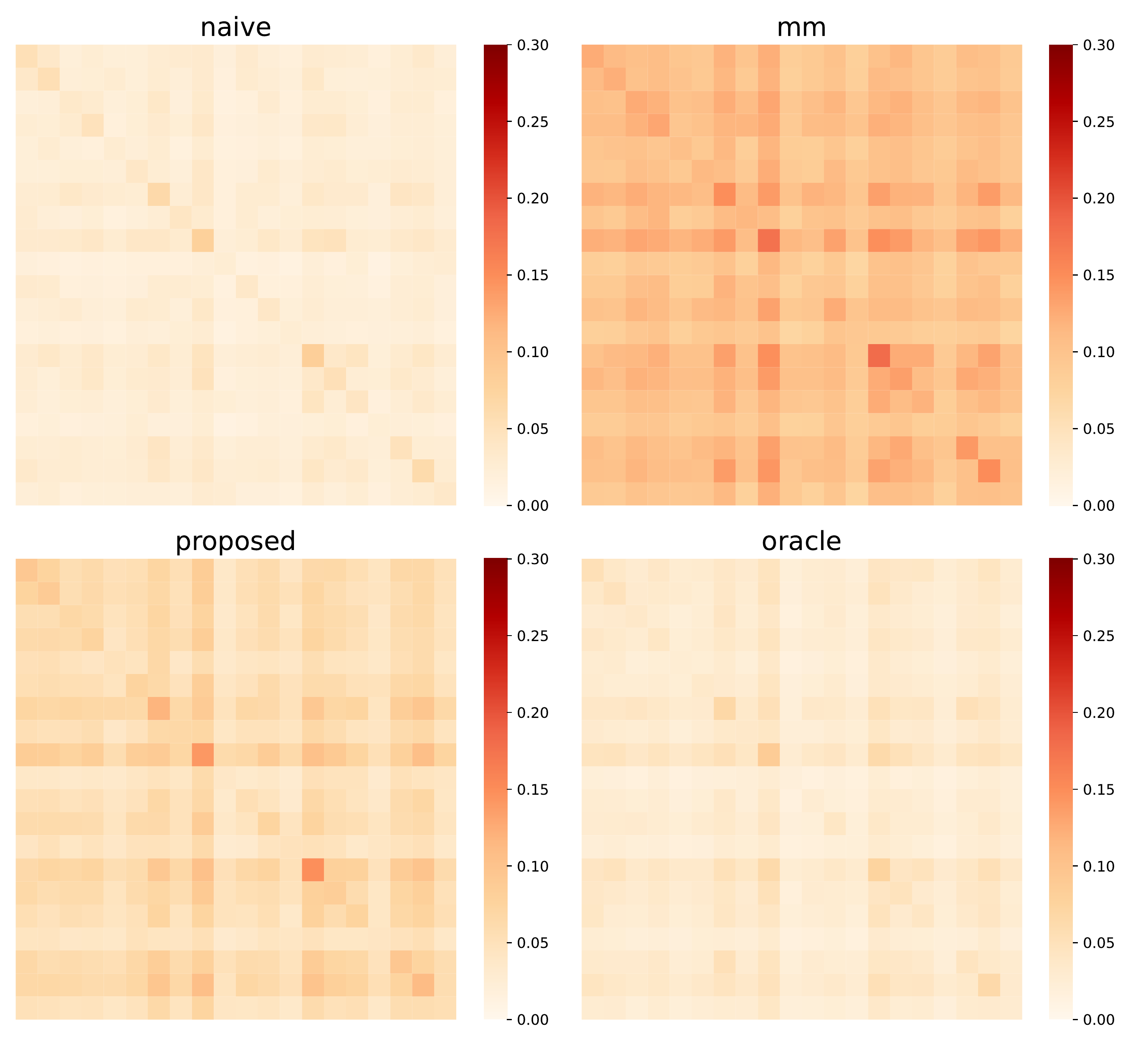} }}%
	\caption{Heatmaps for estimation of low rank $\bS$ with $n=1000$ over 100 repetitions.}%
	\label{Fig:4}%
\end{figure}

\begin{figure}[H]
	\centering
	\subfloat[\centering Heatmap of \emph{bias} for each entry ]{{\includegraphics[width=.48\textwidth]{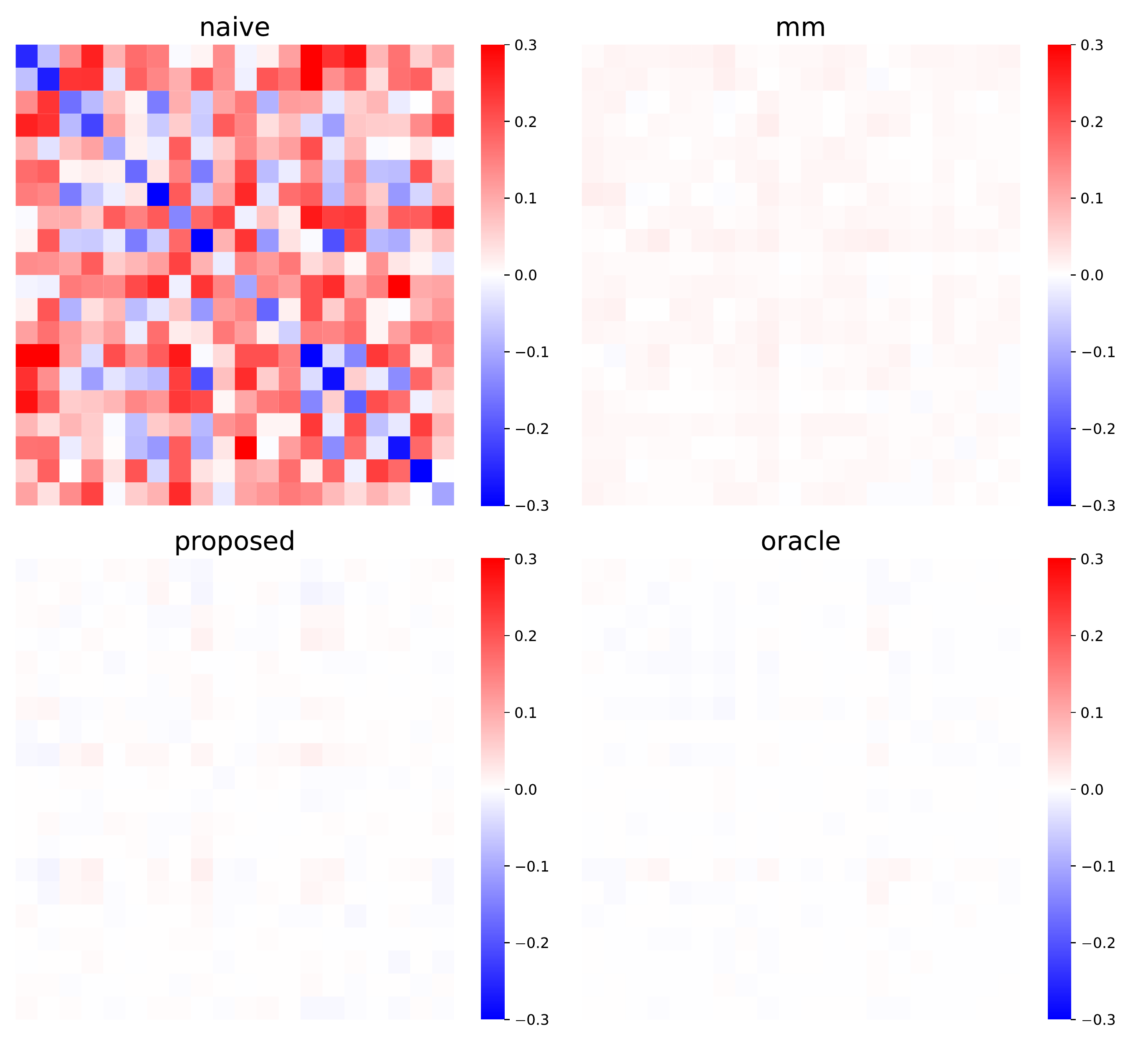} }}%
	~
	\subfloat[\centering Heatmap of \emph{standard deviation} in each entry ]{{\includegraphics[width=.48\textwidth]{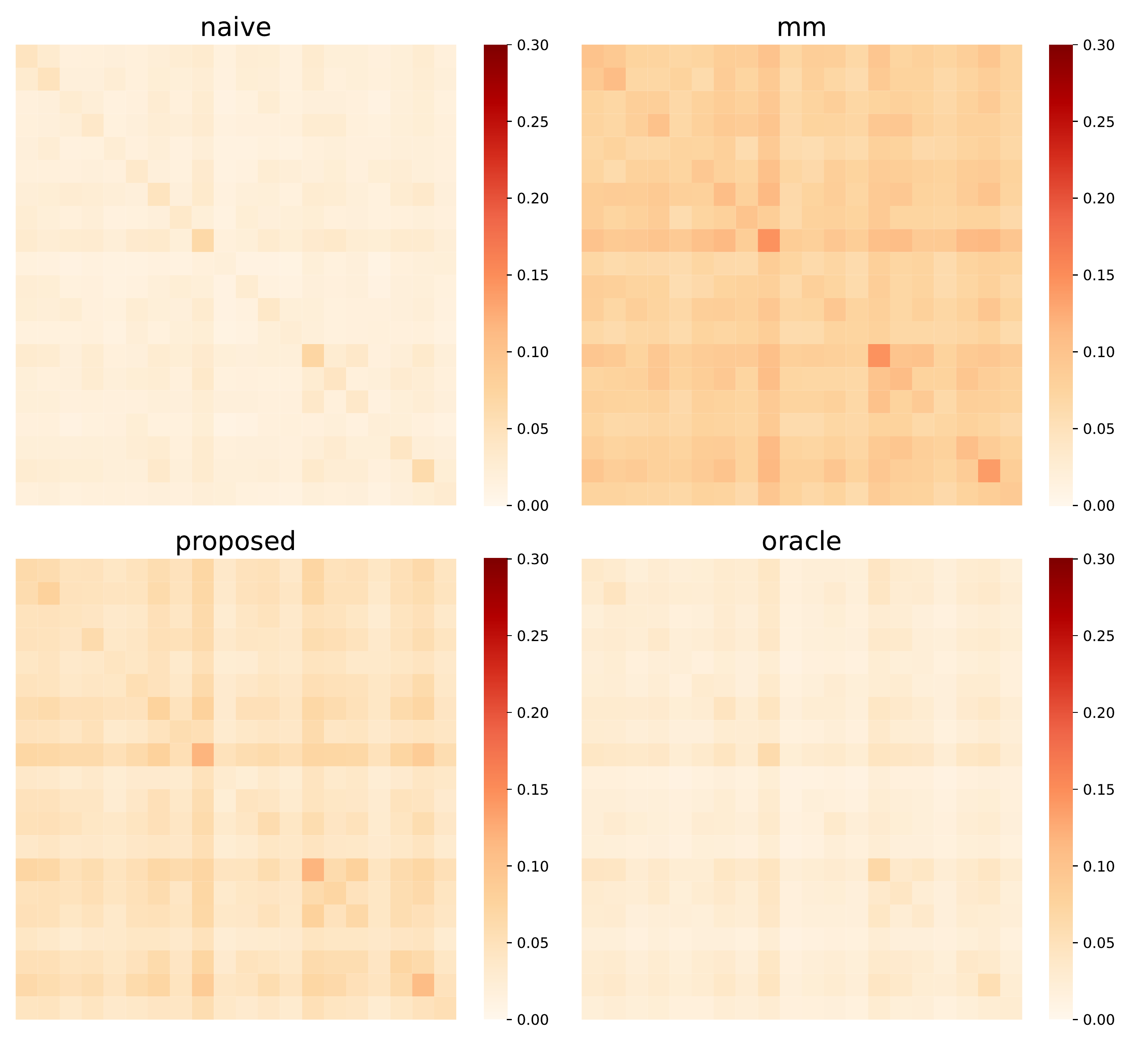} }}%
	\caption{Heatmaps for estimation of low rank $\bS$ with $n=1500$ over 100 repetitions.}%
	\label{Fig:5}%
\end{figure}

\begin{figure}[H]
	\centering
	\subfloat[\centering Heatmap of \emph{bias} for each entry ]{{\includegraphics[width=.48\textwidth]{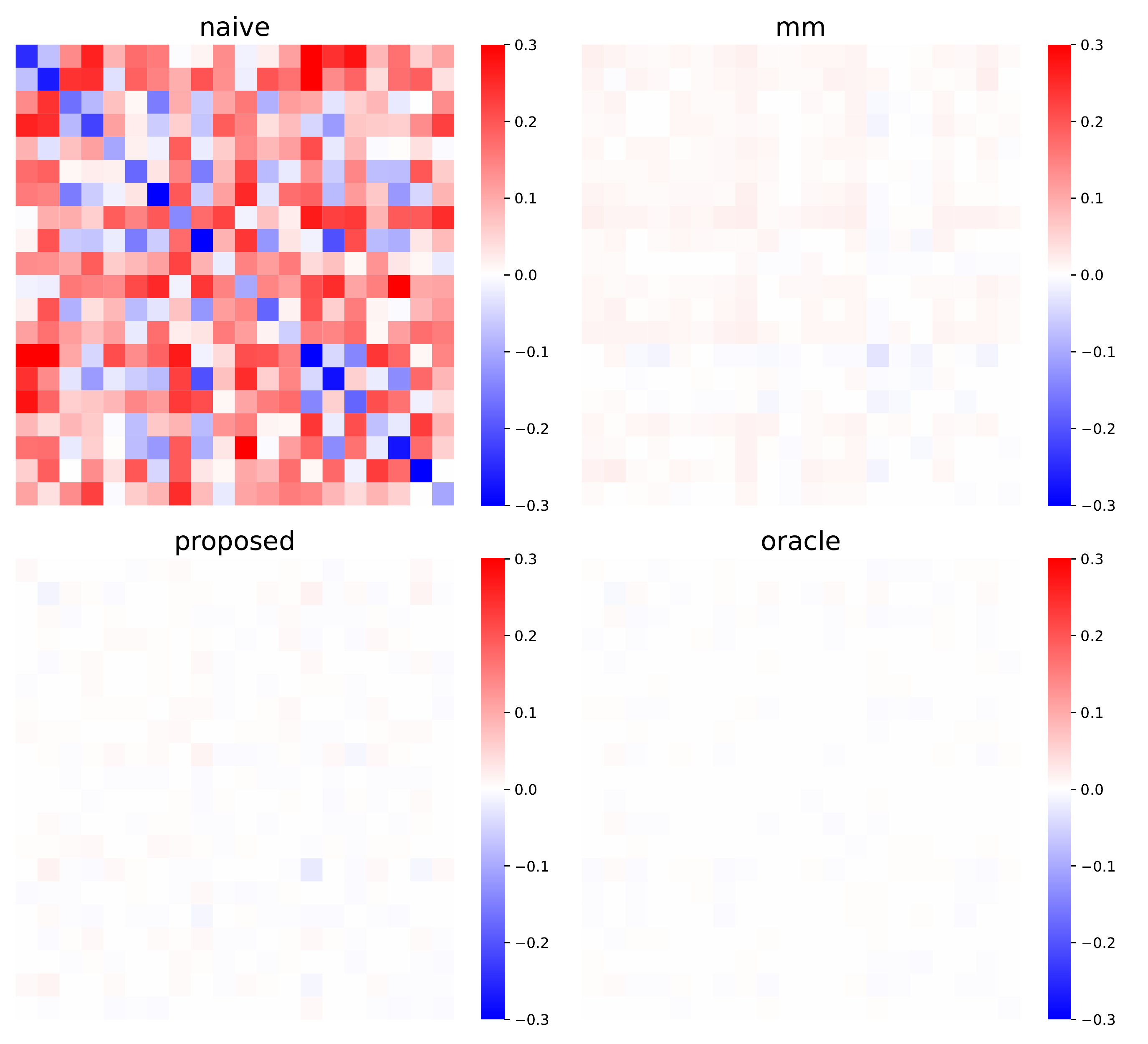} }}%
	~
	\subfloat[\centering Heatmap of \emph{standard deviation} in each entry ]{{\includegraphics[width=.48\textwidth]{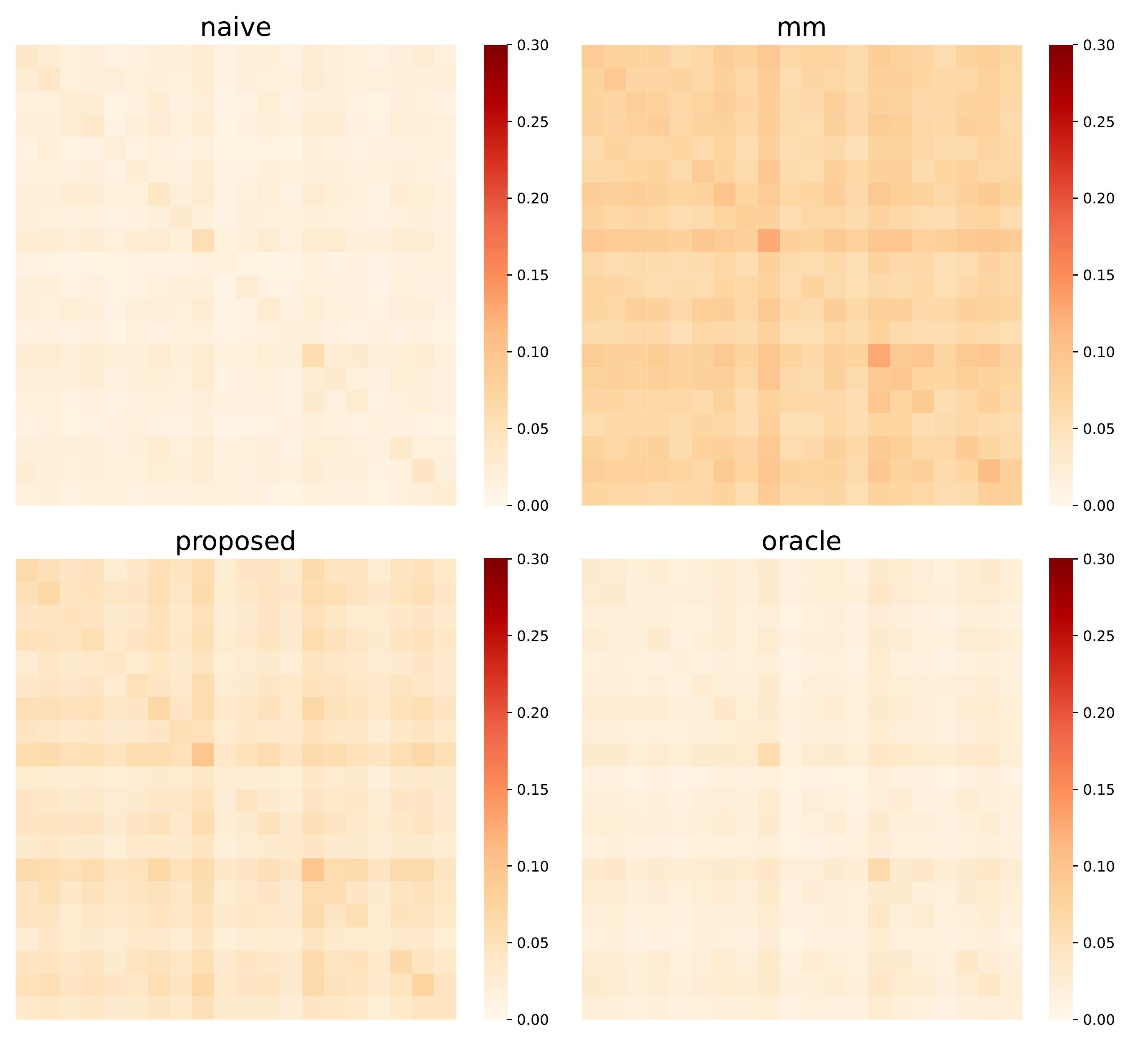} }}%
	\caption{Heatmaps for estimation of low rank $\bS$ with $n=2000$ over 100 repetitions.}%
	\label{Fig:6}%
\end{figure}

\section{Out-of-sample error estimation}
\label{sec:out-of-sample}

In this appendix, we present a by-product of our techniques for estimating the noise covariance.  
Suppose we wish to evaluate the performance of a regression method on a new data, we define the out-of-sample error for the multi-task linear model \eqref{eq: model} as
$$
\E \big[(\hbB - \bB^*)^\top\bx_{\text{new}} \bx_{\text{new}}^\top (\hbB - \bB^*)|(\bX, \bY)\big] = \bH^\top \bH,
$$
where $\bx_{\text{new}}$ is independent of the data $(\bX; \bY)$ with the same distribution as any row of $\bX$. 

 The following theorem on estimation of out-of-sample error is an by-product of our technique for constructing $\hbS$.  
\begin{theorem}[Out-of-sample error]
	\label{thm:out-of-sample}
	Under the same conditions of \Cref{thm:covariance}, with $\bZ = \bX \bSigma^{-\frac12}$, we have 
	{\small
		\begin{align*}
			&\fnorm{
				(\bI_T -\hbA/n)\bH^\top\bH (\bI_T -\hbA/n) -  \frac{1}{n^2} \big( \bF^\top \bZ \bZ^\top \bF + \hbA \bF^\top \bF + \bF^\top \bF \hbA - 
				p\bF^\top \bF\big)}\\
			\le &  \Theta_3 n^{-\frac 12} (\fnorm{\bH}^2 + \fnorm*{\bF}^2/n)
		\end{align*}
	}
	for some non-negative random variable $\Theta_3$ of constant order, in the sense that $\E [\Theta_3] \le C(\gamma,\tau')$ under \Cref{assu: tau}(i), 
	and with ${\E [I(\Omega)\Theta_3]} \le C(\gamma,c)$ under \Cref{assu: tau}(ii),  where $I(\Omega)$ is the indicator function of an event $\Omega$
	with $\P(\Omega)\to 1$.
\end{theorem}
\Cref{thm:out-of-sample} generalizes the result in \cite{bellec2020out} to multi-task setting. 
While the out-of-sample error $\bH^\top\bH$ is unknown, the quantities $\bZ$, $\bF$, $\hbA$ are observable. {Since typically the quantity $(\fnorm{\bH}^2 + \fnorm*{\bF}^2/n)$ is of a constant order}, \Cref{thm:out-of-sample} suggests the following estimate of $\bH^\top\bH$:
\[
\frac{1}{n^2} (\bI_T - \hbA/n)^{-1} \big( \bF^\top \bZ\bZ^\top \bF + \hbA \bF^\top \bF + \bF^\top \bF \hbA - 
p\bF^\top \bF\big) (\bI_T - \hbA/n)^{-1},
\]
which can further be used for parameter tuning in multi-task linear model. 

We present the proof of \Cref{thm:out-of-sample} in \Cref{sec: out-of-sample proof}
\section{Useful operator norm bounds}\label{sec:opnorm-bound}
Let us first introduce two events besides the event $U_1 = \big\{ \norm{\hbB}_0 \le n(1-c)/2 \big\}$ in \Cref{assu: tau}(ii), we define events $U_2$ and $U_3$ as below,
\begin{align*}
	U_2 &= \big\{\inf_{\bb\in \R^p: \| \bb\|_0 \le (1-c)n}  \|\bX \bb\|^2/(n \|\bSigma^{\frac 12} \bb\|^2)  > \eta\big\},\\
	U_3 &= \big\{ \opnorm{\bX\bSigma^{-\frac 12}} < 2 \sqrt n  + \sqrt p\big\}.
\end{align*}
Under \Cref{assu: design,assu: regime}, \cite[Lemma B.1]{bellec2021biasing} guarantees $\P(U_2) \ge 1 - C(\gamma, c)e^{-C(\gamma, c)n}$ for some constant $\eta$ that only depends on constants $\gamma, c$.  
Under \Cref{assu: design}, \cite[Theorem II.13]{DavidsonS01} guarantees $\P(U_3) > 1 - e^{-n/2}$ and there exists a random variable $z\sim \calN(0,1)$ s.t. $\opnorm{\bX\bSigma^{-\frac 12}} \le \sqrt{n} + \sqrt{p} + z$ almost surely. Therefore, under \Cref{assu: design,assu: regime}, we have
\begin{equation}\label{eq: opnorm-Z}
	\E[ \opnorm{n^{-\frac 12}\bX\bSigma^{-\frac 12}}^2] \le (1 + \sqrt{p/n})^2 + n^{-1} \le C(\gamma). 
\end{equation}
Furthermore, under \Cref{assu: regime,,assu: design,,assu: tau}(ii),  $\P(U_1\cap U_2\cap U_3)\to1$ by a union bound, and for large enough $n$, 
\begin{align}
	\P\big\{(U_1\cap U_2\cap U_3)^c\big\} \notag
	<& 1/T + C(\gamma, c)e^{-n/C(\gamma, c)} \notag\\
	=& \frac 1T (1 + TC(\gamma, c)e^{-n/C(\gamma, c)} )\notag\\
	<& \frac 1T  (1 + C(\gamma, c)e^{\sqrt{n} -n/C(\gamma, c)})\notag \\
	<&  \frac 1T C(\gamma, c). \label{eq: P-Omega}
\end{align}

Now we provide the operator norm bounds for $\bI_T - \hbA/n$ and $(\bI_T - \hbA/n)^{-1}$.  
\begin{lemma} \label{lem: aux-tau>0}
	Suppose that \Cref{assu: design} holds. If $\tau >0$ in \eqref{eq: hbB} with $\tau' = \tau /\opnorm{\bSigma}$, then 
	\begin{enumerate}[label = (\roman*)]
		\item $\opnorm{\bI_T - \hbA/n} \le 1$. 
		\item In the event $U_3$, we have $\opnorm{(\bI_T - \hbA/n)^{-1}} \le 1 + (\tau')^{-1} (2 + \sqrt{p/n})^2$. Furthermore, $\E [\opnorm{(\bI_T - \hbA/n)^{-1}} ]\le 1 + (\tau')^{-1} [(1 + \sqrt{p/n})^2 +n^{-1}].$
	\end{enumerate}
\end{lemma}

\begin{lemma}  \label{lem: aux-tau=0}
	Suppose that \Cref{assu: design} holds. 
	If $\tau=0$ in \eqref{eq: hbB}, then 
	\begin{enumerate}[label = (\roman*)]
		\item In the event $U_1$, we have $\opnorm{\bI_T - \hbA/n} \le 1$. 
		\item In the event $U_1$, $\opnorm{(\bI_T - \hbA/n)^{-1}} \le C(c)$. Hence, $\E [I(U_1)\opnorm{(\bI_T - \hbA/n)^{-1}} ]\le C(c).$
	\end{enumerate}
\end{lemma}

\begin{lemma}\label{lem: MN}
	With $\bN = (\bI_T \otimes \bX) \bM^\dagger (\bI_T \otimes \bX^\top)$, we have $\opnorm{\bN}\le 1$. 
\end{lemma}

\section{Lipschitz and differential properties for a given, fixed noise matrix \texorpdfstring{$\bE$}{E} }
\label{sec:Lipschitz-fix-E}

We need to study Lipschitz and differential properties of certain mappings when the noise matrix $\bE$ is fixed.
Let
$g:\R^{p\times T}\to \R$ defined by
$g(\bB)=\tau \fnorm{\bB}^2/2 + \lambda \norm{\bB}_{2,1}$ be the penalty in
\eqref{eq: hbB}.
For a fixed value of $\bE$,
define the mappings 
\begin{align}
\bZ
&\longmapsto
\bH(\bZ) = \argmin_{\bar\bH\in\R^{p\times T}} \frac{1}{2n}\fnorm{\bE - \bZ\bar\bH}^2 + g(\bSigma^{-1/2}\bar\bH)
&(
\R^{n\times p}
\to
\R^{p\times T}
)
\\
\bZ
&\longmapsto
\bF(\bZ) = \bE - \bZ\bH(\bZ)
&(
\R^{n\times p}
\to
\R^{n\times T}
)
\\
\bZ
&\longmapsto
D(\bZ) = (\fnorm{\bH(\bZ)}^2 + \fnorm{\bF(\bZ)}^2/n)^{1/2}
&(
\R^{n\times p}
\to
\R
).
\end{align}
Next, define the random variable $\bZ = \bX\bSigma^{-\frac12}\in\R^{n\times p}$, and let us use the convention that if arguments of $\bH,\bF$ or $D$ are omitted then these mappings are implicitly taken at the realized value of the random variable $\bZ = \bX\bSigma^{-\frac12}\in\R^{n\times p}$ where $\bX$ is the observed design matrix. With this convention and by definition of the above mappings, we then have
$\bH = \bH(\bZ) = \bSigma^{1/2}(\hbB - \bB^*)$ as well as 
$\bF = \bF(\bZ) = \bY - \bX\hbB$
and $D = [\fnorm{\bH}^2 + \fnorm{\bF}^2/n]^{1/2}$
so that the notation is consistent with the rest of the paper
(in particular, with \eqref{eq:def F H N}).

Finally, denote the $(i,j)$-th entry of $\bZ$ by $z_{ij}$
throughout this appendix, and the corresponding partial derivatives of the above mappings
by $\frac{\partial}{\partial z_{ij}}$.

\subsection{Lipschitz properties}

\begin{lemma}\label{lem: lipschitz elastic-net}
	For multi-task elastic-net (\ie, $\tau>0$ in \eqref{eq: hbB}), the mapping $\bZ \mapsto D^{-1}\bF/\sqrt{n}$ is $n^{-\frac12}L$-Lipschitz with $L = 8 \max(1, (2\tau')^{-1})$, where $\tau' = \tau/\opnorm{\bSigma}$.
\end{lemma}

\begin{lemma}\label{lem: lipschitz-Lasso}
	For multi-task group Lasso (\ie, $\tau=0$ in \eqref{eq: hbB}). 
we have 
\begin{itemize}
	\item[(1)] In $U_1\cap U_2$, the map $\bZ \mapsto D^{-1}\bF/\sqrt{n}$ is $n^{-\frac12}L$-Lipschitz with $L = 8 \max(1, (2\eta)^{-1})$. 
	\item[(2)] In $U_1\cap U_2 \cap U_3$,  the map $\bZ \mapsto D^{-1}\bZ^\top\bF/n$ is $n^{-1/2} (1 + (2 +\sqrt{p/n})L)$-Lipschitz, where $L = 8 \max(1, (2\eta)^{-1})$ as in (1). 
\end{itemize}
\end{lemma}

\begin{corollary}\label{cor: partialD}
	Suppose that \Cref{assu: tau} holds, then
	\begin{itemize}
		\item[(1)] Under \Cref{assu: tau}(i) that $\tau>0$ and $\tau'=\tau/\opnorm{\bSigma}$, we have 
		\begin{align*}
			\sum_{ij}\Big(\frac{\partial D}{\partial z_{ij}}\Big)^2 \le n^{-1}D^2 [4 \max(1, (2\tau')^{-1})]^2.
		\end{align*}
	This implies that 
	$ nD^{-2}\sum_{ij}(\frac{\partial D}{\partial z_{ij}})^2 \le C(\tau'). 
	$
		\item[(2)] Under \Cref{assu: tau}(ii) that $\tau=0$ and $\P(U_1)\to 1$, in the event $U_1\cap U_2$, we have 
		\begin{align*}
			\sum_{ij}\Big(\frac{\partial D}{\partial z_{ij}}\Big)^2 \le n^{-1}D^2 [4 \max(1, (2\eta)^{-1})]^2.
		\end{align*}
	This implies that 
	$ nD^{-2}\sum_{ij}(\frac{\partial D}{\partial z_{ij}})^2 \le C(\eta) = C(\gamma, c)$ since $\eta$ is a constant that only depends on $\gamma, c$. 
	\end{itemize}

\end{corollary}

\subsection{Derivative formulae} 
\label{sec: gradient-identity}
Note that with a fixed noise $\bE$, {\Cref{lem: lipschitz elastic-net,lem: lipschitz-Lasso} guarantee that the map $\bZ \mapsto \bF$ is Lipschitz}, hence the derivative exists almost everywhere by Rademacher’s theorem. 
We present the formula for derivative of this map in \Cref{lem: Dijlt}. 
\begin{lemma}\label{lem: Dijlt}
	Recall $\bF = \bY - \bX\hbB$ with $\hbB$ defined in \eqref{eq: hbB}. Under \Cref{assu: tau}(i) $\tau>0$, or under \Cref{assu: tau}(ii) $\tau =0$ and in the event $U_1\cap U_2$, for each $i,l\in [n], j\in [p], t\in [T]$, the following derivative exists almost everywhere and has the expression
	\begin{equation*}
		\frac{\partial F_{lt}}{\partial z_{ij}} = D_{ij}^{lt} + \Delta_{ij}^{lt},
	\end{equation*}
	where 
	$D_{ij}^{lt} 
	= -(\be_j^\top\bH \otimes \be_i^\top) (\bI_{nT} - \bN) (\be_t\otimes \be_l),$ and 
	$\Delta_{ij}^{lt} 
	= -(\be_t^\top \otimes \be_l^\top)(\bI_T\otimes \bX)
	\bM^\dagger (\bI_T\otimes \bSigma^{\frac12})
	\bigl(\bF^\top \otimes \bI_{p}\bigr)(\be_i \otimes\be_j).
	$
	Furthermore, a straightforward calculation yields 
	$$\sum_{i=1}^n D_{ij}^{it} = -\be_j^\top \bH (n\bI_T - \hbA)\be_t.$$
\end{lemma}

\begin{lemma}\label{lem: partialF}
	Suppose that \Cref{assu: tau} holds. 
	\begin{itemize}
		\item[(1)] Under \Cref{assu: tau}(i) that $\tau>0$ and $\tau' = \tau/\opnorm{\bSigma}$, we have 
		\begin{align*}
			\frac 1n \sum_{ij}\norm*{\frac{\partial (\bF/D)}{\partial z_{ij}}}^2_{\rm F} \le \underbrace{4 \max(1, (\tau')^{-1} (T\wedge \frac{p}{n})) + 2  n^{-1}  [4 \max(1, (2\tau')^{-1})]^2}_{f(\tau', T, n, p)}. 
		\end{align*}
		\item[(2)] Under \Cref{assu: tau}(ii) that $\tau=0$ and $\P(U_1)\to 1$, in the event $U_1\cap U_2$, we have 
		\begin{align*}
			\frac 1n \sum_{ij}\norm*{\frac{\partial (\bF/D)}{\partial z_{ij}}}^2_{\rm F} \le \underbrace{4 \max(1, (\eta)^{-1} (T\wedge \frac{p}{n})) + 2  n^{-1}  [4 \max(1, (2\eta)^{-1})]^2}_{f(\eta, T, n, p)}. 
		\end{align*}
	\end{itemize}
	Furthermore, the right-hand side in (1) can be bounded from above by $C(\tau') (T\wedge \frac pn)$, and the  right-hand side in (2) can be bounded from above by $C(\gamma, c)$ in the regime $p/n\le \gamma$. 
\end{lemma}

\section{Lipschitz and differential properties for a given, fixed design matrix}
\label{sec:Lipschitz-fix-X}
We also need to study Lipschitz and derivative properties of functions
of the noise $\bE$ when the design $\bX$ is fixed. Formally, for a given and fixed design matrix $\bX$, define the function
$\bE\mapsto \bF(\bE)$ by
the value $\bY-\bX\hbB$ of the residual matrix 
when the observed data $(\bX,\bY)$ is $(\bX,\bX\bB^* + \bE)$
and with $\hbB$ the estimator \eqref{eq: hbB}.
Note that this map is 1-Lipschitz by \cite[proposition 3]{bellec2016bounds}. Rademacher’s theorem thus guarantees this map is differentiable almost everywhere. We denote its partial derivative by $\frac{\partial}{\partial E_{it}}$
for each entry $(E_{it})_{i\in[n],t\in[T]}$ of the noise matrix $\bE$.
We present its derivative formula in \Cref{lem:JacobianE} below. 
\begin{lemma}\label{lem:JacobianE}
	For each $i,l\in [n], t,t'\in [T]$, the following derivative exists almost everywhere and has the expression
	\begin{align*}
		\frac{\partial F_{lt}}{\partial E_{it'}} = \be_l^\top\be_i\be_t^\top\be_{t'}
		- \be_l^\top (\be_t^\top \otimes \bX)\bM^\dagger (\be_{t'} \otimes \bX^\top) \be_i.
	\end{align*}
As a consequence, we further have
	\begin{equation*}
		\sum_{i=1}^n \frac{\partial F_{it}}{\partial E_{it'}} = \be_t^\top (n\bI_T -\hbA)\be_{t'},\quad \sum_{i=1}^n \frac{\partial \be_i^\top\bZ \bH\be_{t}}{\partial E_{it'}} = \be_t^\top \hbA \be_{t'}.
	\end{equation*}
\end{lemma}

\section{Probabilistic tools}
\label{sec:proba-tools}

We first list some useful variants of Stein's formulae and Gaussian-Poincar\'e inequalities. Let $f'$ denote the derivative of a differentiable univariate function. For a differentiable vector-valued function $\bff(\bz): \R^n \to \R^n$, denote its Jacobian (derivative) and divergence respectively by $\nabla \bff(\bz)$ and $\div \bff(\bz)$, \ie, $[\nabla \bff(\bz)]_{i,l} = \frac{\partial f_i(\bz)}{\partial z_l}$ for $i,l\in [n]$, and $\div \bff(\bz) =  \trace(\nabla \bff(\bz))$. 

\begin{lemma}[Second-order Stein's formula \cite{bellec2021second}]\label{lem: 2nd-Stein}
	The following identities hold provided the involved derivatives exist a.e. and the expectations are finite.
	\begin{enumerate}[label = \roman*)]
		\item $z\sim \calN(0, 1)$, $f: \R \to \R$, then 
		$$\E [(z f(z) - f'(z))^2] = \E [f(z)^2] + \E[(f'(z))^2].$$
		\item $\bz \sim \calN_n(\mathbold 0, \bI_n)$, $f: \R^n \to \R^n$, then 
		$$
		\E [(\bz^\top \bff(\bz) - \div\bff(\bz))^2] = \E \big[\norm{\bff(\bz)}^2 + \trace[( \nabla\bff(\bz))^2]\big] \le  \E \big[\norm{\bff(\bz)}^2 + \fnorm{ \nabla\bff(\bz)}^2\big],
		$$
		where the inequality uses Cauchy-Schwarz inequality. 
		\item More generally, for $\bz \sim \calN_n(\mathbold 0, \bSigma)$, $\bff: \R^n \to \R^n$, then 
		\begin{align*}
			\E [(\bz^\top \bff(\bz) - \trace(\bSigma \nabla\bff(\bz))^2] 
			&= \E \big[\norm{\bSigma^{\frac12}\bff(\bz)}^2 + \trace[(\bSigma \nabla\bff(\bz))^2]\big]\\
			&\le \E \big[\norm{\bSigma^{\frac12}\bff(\bz)}^2 + \fnorm{(\bSigma \nabla\bff(\bz)}^2]\big],
		\end{align*}
		where the inequality uses Cauchy-Schwarz inequality. 
	\end{enumerate}
	
\end{lemma}

\begin{lemma}[Gaussian-Poincar\'e inequality \cite{boucheron2013concentration}]\label{lem: Gaussian-Poincare}
	The following inequalities hold provided the right-hand side derivatives exist a.e. and the expectations are finite.   
	\begin{enumerate}[label = \roman*)]
		\item $z \sim \calN(0, 1)$, $f: \R \to \R$, then 
		$$\text{Var} [f(z)] \le \E [(f'(z))^2].$$
		\item $\bz \sim \calN_n(\mathbold 0, \bI_n)$, $f: \R^n \to \R$, then 
		$$\text{Var} [f(\bz)] \le \E [\norm{\nabla f(\bz)}^2].$$
		\item $\bz \sim \calN_n(\mathbold 0, \bI_n)$, $\bff: \R^n \to \R^m$, then 
		$$\E[\norm{\bff(\bz) - \E[\bff(\bz)]}^2] \le \E [\fnorm{\nabla \bff(\bz)}^2].$$
		\item More generally, for $\bz \sim \calN_n(\mathbold 0, \bSigma)$, $\bff: \R^n \to \R^m$, then 
		$$\E[\norm{\bff(\bz) - \E[\bff(\bz)]}^2] \le \E [\fnorm{\bSigma^{\frac12}\nabla \bff(\bz)}^2].$$
	\end{enumerate}
\end{lemma}

Now we present a few important lemmas, whose proofs are based on \Cref{lem: 2nd-Stein} and \Cref{lem: Gaussian-Poincare}. 
\begin{lemma}\label{lem: stein-EF}
	Assume that Assumption \ref{assu: noise} holds. For fixed $\bX$, we have
	\begin{equation*}
		\E\Bigl[
		\fnorm{\bE^\top \bF/\tD - \bS (n\bI_T - \hbA )/\tD}^2\Bigr] 
		\le 4\trace(\bS), 
	\end{equation*}
where $\tD = \big(\fnorm*{\bF}^2 + n \trace(\bS)\big)^{\frac 12}$. 
\end{lemma}

\begin{lemma}\label{lem:steinX}
	Let $\bU, \bV:\R^{n\times p} \to \R^{n\times T}$ be two locally Lipschitz functions of $\bZ$ with \iid $\calN(0,1)$ entries, then 
	\begin{align*}
		&\E\Big[\norm[\Big]{\bU^\top \bZ \bV - 
			\sum_{j=1}^p\sum_{i=1}^n
			\frac{\partial}{\partial z_{ij} }\Bigl(\bU^\top \be_i \be_j^\top \bV \Bigr)
		}_{\rm F}^2\Big]\\
		\le~& \E \fnorm*{\bU}^2 \fnorm*{\bV}^2+ \E \sum_{ij}\Big[
		2\fnorm*{\bV}^2\fnorm*{ \frac{\partial \bU}{\partial z_{ij}} }^2
		+ 2\fnorm*{\bU}^2\fnorm*{ \frac{\partial \bV}{\partial z_{ij}} }^2\Big].
	\end{align*}
\end{lemma}

\begin{corollary}\label{cor:steinX}
	Assume the same setting as \Cref{lem:steinX}. 
	If on some open set $\Omega\subset \R^{n\times p}$ with $\P(\Omega^c)\le C/T$ for some constant $C$, we have (i) $\bU$ is $n^{-1/2} L_1$ -Lipschitz and $\fnorm{\bU}\le 1$, (ii) $\bV$ is $n^{-1/2}L_2$ -Lipschitz and $\fnorm{\bV}\le K$. Then 
	\begin{align*}
		&\E\Big[I(\Omega) \Big\|\bU^\top \bZ \bV - 
			\sum_{j=1}^p\sum_{i=1}^n
			\frac{\partial}{\partial z_{ij} }\Bigl(\bU^\top \be_i \be_j^\top \bV \Bigr)
		\Big\|_{\rm F}^2\Big]\\
		\le~& K^2 + 2C( K^2 L_1^2  + L_2^2 ) + 2\E \Big[I(\Omega)
		\sum_{ij}\Big( K^2\fnorm*{\frac{\partial \bU}{\partial z_{ij}} }^2 + \fnorm*{\frac{\partial \bV}{\partial z_{ij}} }^2\Big) \Big].
	\end{align*}
\end{corollary}
	
\begin{lemma} \label{lem: Chi2type}
	Let $\bU,\bV: \R^{n\times p} \to \R^{n\times T}$
	be two locally Lipschitz 
	functions of $\bZ$ with \iid $\calN(0,1)$ entries.
	Assume also that $\fnorm{\bU} \vee \fnorm{\bV} \le 1$ almost surely.
	Then
	\begin{align*}
		&\E\Bigl[
		\fnorm[\Big]{
			p \bU^\top \bV - \sum_{j=1}^p 
			\Bigl(\sum_{i=1}^n \partial_{ij} \bU^\top \be_i - \bU^\top \bZ \be_j\Bigr)
			\Bigl(\sum_{i=1}^n \partial_{ij} \be_i^T \bV  - \be_j^T \bZ^T \bV\Bigr)
		}
		\Bigr] 
		\\
		\le~&
		2 \|\bU\|_\partial
		\|\bV\|_\partial
		+
		\sqrt p
		\bigl(
		\sqrt 2 + (3+\sqrt{2})(\|\bU\|_{\partial} + \|\bV\|_{\partial})
		\bigr),
	\end{align*}
	where $ \partial_{ij} \bU \defas \partial \bU /\partial z_{ij}$,  
	and $\|\bU\|_\partial \defas \E[\sum_{i=1}^n\sum_{j=1}^p
	\fnorm{ \partial_{ij} \bU}^2]^{\frac 12}$. 
\end{lemma}

\begin{proposition}\label{prop: V1}
	Suppose that \Cref{assu: noise} holds. \newline
	Let $\bQ_1 = \frac{\frac{1}{n}\big(\bF^\top\bF + \bH^\top\bZ^\top\bF - \bS(n\bI_T - \hbA) \big)}{\fnorm{\bS^{\frac 12}}(\fnorm*{\bF}^2/n + \trace(\bS))^{\frac 12} n^{-\frac 12} }$,  
	then  $\E [\fnorm*{\bQ_1}^2] \le 4$.
\end{proposition}

\begin{proposition} \label{prop: V6}
	Suppose that \Cref{assu: design,,assu: regime,assu: tau} hold. 
	Let \[\bQ_2 = \frac{\frac{1}{n^2} \big( \bF^\top \bZ\bZ^\top \bF - \bF^\top\bF (p\bI_T - \hbA)+ (n\bI_T -\hbA) \bH^\top\bZ^\top\bF \big)}{(\fnorm{\bH}^2 + \fnorm*{\bF}^2/n)n^{-\frac 12}}, 
	\]
	then $\E [\fnorm*{\bQ_2}^2] \le C(\tau') (T\wedge (1+\frac{p}{n})) (1 + \frac pn)$ under \Cref{assu: tau}(i), and 
	 $\E [I(\Omega)\fnorm*{\bQ_2}^2] \le C(\gamma, c)$ under \Cref{assu: tau}(ii) for some set $\Omega$ with $\P(\Omega)\to 1$. 
\end{proposition}

\begin{proposition} \label{prop: V8}
	Suppose that \Cref{assu: design,,assu: regime,assu: tau} hold. 
	Let $\Xi= (n\bI_T -\hbA)\bH^\top\bZ^\top\bF$, and 
	 \[
	\bQ_3 = \frac{\frac{1}{n^2} \big( p\bF^\top \bF -\bF^\top \bZ\bZ^\top \bF - (n\bI_T -\hbA)\bH^\top\bH (n\bI_T -\hbA) -  \Xi - \Xi^\top\big)}{(\fnorm{\bH}^2 + \fnorm*{\bF}^2/n)n^{-\frac 12}},
	\]
	then $\E [\fnorm*{\bQ_3}] \le C(\gamma,\tau')$ under \Cref{assu: tau}(i), and $\E [I(\Omega)\fnorm*{\bQ_3}] \le C(\gamma, c)$ under \Cref{assu: tau}(ii) for some set $\Omega$ with $\P(\Omega)\to 1$. 
\end{proposition}

\section{Proofs of main results}\label{sec: proof-thms}
In this appendix, we provide proofs of the theoretical results in \Cref{sec: main-results} of the pull paper and  \Cref{sec:out-of-sample} of this supplement. 
\subsection{Proof of \saferef{Proposition}{prop: oracle}}
\begin{proof}[Proof of \Cref{prop: oracle}]
	With $\bS=\sum_{t=1}^T \sigma_t^2 \bu_t \bu_t^T$
	the spectral decomposition of $\bS$,
	we have $\fnorm{\bE^\top\bE-n\bS}^2
	=\sum_{t\in[T]}\sum_{t'\in[T]}
	[\bu_{t'}^T((\bE^\top\bE-n\bS) \bu_t]^2$.
	We now compute the expectation of one term
	indexed by $(t,t')$.
	The random variable $\bu_{t'}^T(\bE^\top\bE-n\bS) \bu_t$
	is the sum of $n$ \iid mean zero random variables
	with the same distribution as $z_{t'} z_{t}-\bu_{t'}^T\bS\bu_t$
	where $(z_t,z_{t'})\sim \calN_2(\mathbf{0},\diag(\sigma_t^2,\sigma_{t'}^2))$. Thus
	$$
	\E[(\bu_{t'}^T(\bE^\top\bE-n\bS) \bu_t)^2]
	=
	n\text{Var}[z_{t'} z_{t}-\bu_{t'}^T\bS\bu_t]
	=
	n2\sigma_t^4I_{t=t'}
	+
	n\sigma_t^2\sigma_{t'}^2 I_{t\ne t'}
	$$
	due to $\text{Var}[\chi^{2}_1]=2$
	if $t=t'$ and independence if $t\ne t'$.
	Summing over all $(t,t')\in[T] \times [T]$
	yields 
	$2n\sum_{t=1}^T \sigma_t^4
	+n \sum_{t\ne t'}\sigma_t^2\sigma_{t'}^2
	=n\sum_{t=1}^T \sigma_t^4
	+n(\sum_{t=1}^T\sigma_t^2)^2
	= n \fnorm{\bS}^2 
	+ n [\trace(\bS)]^2
	$
	as desired.
	
	The inequality simply follows from
	$\fnorm{\bS}^2\le [\trace(\bS)]^2$ since $\bS$ is positive semi-definite.
\end{proof}

\subsection{Proof of \saferef{Proposition}{prop: mom}}
\begin{proof}[Proof of Proposition~\ref{prop: mom}]
	Without of loss of generality, we assume $\bSigma = \bI_p$. For general positive definite $\bSigma$, the proof follows by replacing $(\bX, \bB^*)$ with $(\bX \bSigma^{-\frac12}, \bSigma^{\frac12}\bB^*)$. 
	
	We first derive the method-of-moments estimator $\hbS_{\rm{(mm)}}$. Under \Cref{assu: noise,,assu: design} with $\bSigma = \bI_p$, $\bX$ has \iid rows from $\calN_p(\mathbf 0, \bI_p)$, $\bE$ has \iid rows from $\calN_T(\bf0, \bS)$, and $\bX$ and $\bE$ are independent. Then, the expectations of $\bY^\top \bY$ and  $\bY^\top\bX \bX^\top\bY$ are given by 
	\begin{align}\label{eq:yy}
		\E (\bY^\top \bY) &= \E \big[(\bX\bB^* + \bE)^\top (\bX\bB^* + \bE)\big] = n (\bB^{*\top}\bB^* + \bS), 
	\end{align}
	and 
	\begin{align}\label{eq:yxxy}
		\E (\bY^\top\bX\bX^\top\bY ) &= \E \big[(\bX\bB^* + \bE)^\top \bX \bX^\top (\bX\bB^* + \bE)\big]\notag\\
		&= \E (\bB^{*\top}\bX^\top\bX \bX^\top\bX\bB^*) + \E (\bE^\top\bX \bX^\top\bE)\notag\\
		&= \bB^{*\top}\E(\bX^\top\bX \bX^\top\bX)\bB^* + \E (\bE^\top\bX \bX^\top\bE)\notag\\
		&= n(n+p+1)  \bB^{*\top}\bB^* + np\bS,
	\end{align}
	where the last line uses
	\begin{align*}
		&\E (\bX^\top\bX \bX^\top\bX) \\
		=~& \E \Big[\sum_{i=1}^n (\bx_i\bx_i^\top) \sum_{l=1}^n(\bx_l\bx_l^\top)\Big]\\
		=~& \sum_{i\ne l}\E(\bx_i\bx_i^\top \bx_l\bx_l^\top) + \sum_{i= l} \E (\bx_i\bx_i^\top \bx_l\bx_l^\top)\\
		=~& n(n-1) \bI_p^2 + n\E (\bx_1\bx_1^\top \bx_1\bx_1^\top)\\
		=~& n(n-1) \bI_p + n [2\bI_p^2 + \trace(\bI_p) \bI_p]\\
		=~& n(n+p+1) \bI_p,
	\end{align*}
	and 
	\begin{align*}
		\E (\bE^\top\bX \bX^\top\bE) 
		= \E\big[ \E (\bE^\top\bX \bX^\top\bE |\bE)\big] 
		= \E\big[ \bE^\top \E(\bX \bX^\top)\bE\big] 
		= np \bS.
	\end{align*}
	
	Solving for $\bS$ from the system of equations \eqref{eq:yy} and \eqref{eq:yxxy}, we obtain the method-of-moments estimator
	\begin{align*}
		\hbS_{\rm{(mm)}} = \frac{(n+p+1)}{n(n+1)} \bY^\top \bY - \frac{1}{n(n+1)}  \bY^\top\bX \bX^\top\bY,
	\end{align*}
	and $\E [\hbS_{\rm{(mm)}} ]= \bS$. 
	
	Now we derive the variance lower bound for $\hbS_{\rm{(mm)}}$. Since $\E[\hbS_{\rm{(mm)}} ] = \bS$, $\E \big[\fnorm{\hbS_{\rm{(mm)}} - \bS}^2\big] = \sum_{t, t'} \text{Var}\big\{[\hbS_{\rm{(mm)}} ]_{t,t'}\big\}.$ By definition of $\hbS_{\rm{(mm)}} $, 
	\begin{align*}
		[\hbS_{\rm{(mm)}} ]_{t,t'} = \frac{n+p+1}{n(n+1)} [\by^{(t)}]^\top \by^{(t')}  - \frac{1}{n(n+1)}[\by^{(t)}]^\top \bX\bSigma^{-1}\bX^\top\by^{(t')}.
	\end{align*}
	
	Since $\by^{(t)} = \bX \bbeta^{(t)} + \bep^{(t)},\quad \by^{(t')} = \bX \bbeta^{(t')} + \bep^{(t')},$ for $t\ne t'$,  
	without loss of generality, we assume $\bbeta^{(t)} =  a_0\be_1$ and $\bbeta^{(t')} =  a_1\be_1 + a_2\be_2$ for some constants $a_0, a_1, a_2$. 
	If necessary, we could let $\bu_1 = \bbeta^{(t)}/\norm{\bbeta^{(t)}}$, and $\bu_2 = \tbu_2/\norm{\tbu_2}$ where $\tbu_2 =\bbeta^{(t')} - \bP_{\bu_1}\bbeta^{(t')}$, and completing the basis to obtain an orthonormal basis $\{\bu_1, \bu_2, \ldots, \bu_p\}$ for $\R^p$. Let $\bU =[\bu_1, \bu_2, \ldots, \bu_p]$, then $\bU$ is an orthogonal matrix, hence $\bX\bU$ and $\bX$ have the same distribution, only the first coordinate of $\bU^\top\bbeta^{(t)}$ is nonzero, and only the first two coordinates of $\bU^\top\bbeta^{(t')}$ are be nonzero. That is, we could perform change of variables by replacing $(\bX, \bbeta^{(t)}, \bbeta^{(t')})$ with $(\bX\bU, \bU^\top\bbeta^{(t)}, \bU^\top\bbeta^{(t')})$. 
	
	Therefore,  $\by^{(t)}$ and $\by^{(t')}$ are independent of $\{\bX\be_j: 3\le j\le p\}$. 
	Let $\calF = \sigma(\by^{(t)}, \by^{(t')}, \bX\be_1, \bX\be_2)$ be the $\sigma-$field generated by $(\by^{(t)}, \by^{(t')}, \bX\be_1, \bX\be_2)$, then 
	\begin{align*}
		\text{Var}\big\{[\hbS_{\rm{(mm)}} ]_{t,t'}\big\} 
		&\ge \E\big[ \text{Var}\big\{[\hbS_{\rm{(mm)}} ]_{t,t'}|\calF\big\}\big]= \frac{1}{n^2(n+1)^2} \E\big[ \text{Var}\big\{ [\by^{(t)}]^\top \bX\bX^\top\by^{(t')}|\calF\big\}\big]. 
	\end{align*}
	Note that in the above display, 
	\begin{align*}
		&[\by^{(t)}]^\top \bX\bX^\top\by^{(t')} 
		= \sum_{j=1}^2[\by^{(t)}]^\top \bX\be_j\be_j^\top\bX^\top\by^{(t')} 
		+ \sum_{j=3}^p[\by^{(t)}]^\top \bX\be_j\be_j^\top\bX^\top\by^{(t')}, 
	\end{align*}
	where the first term is measurable with respect to $\calF$, and the second term is a quadratic form
	\begin{align*}
		&\sum_{j=3}^p[\by^{(t)}]^\top \bX\be_j\be_j^\top\bX^\top\by^{(t')} = \sum_{j=3}^p \be_j^\top \bX^\top\by^{(t')} [\by^{(t)}]^\top \bX \be_j =  \bxi^\top \bLambda \bxi,
	\end{align*}
	here $\bxi = [\be_3^\top\bX^\top, \ldots, \be_p^\top\bX^\top]^\top\sim \calN(\mathbold{0}, \bI_{n(p-2)})$, and $\bLambda = \bI_{p-2} \otimes \by^{(t')} [\by^{(t)}]^\top$. 
	Thus, for $t\ne t'$,
	\begin{align*}
		\text{Var}\big\{[\hbS_{\rm{(mm)}} ]_{t,t'}\big\} 
		\ge~& \frac{1}{n^2(n+1)^2} \E \Big\{\text{Var}\big\{ \bxi^\top \bLambda \bxi|\calF\big\}\Big\}\\
		=~& \frac{1}{n^2(n+1)^2} \E \Big\{ \fnorm{\bLambda}^2 + \trace(\bLambda^2)\Big\}\\
		\ge~& \frac{1}{n^2(n+1)^2} \E [ \fnorm{\bLambda}^2 ]\\
		=~& \frac{p-2}{n^2(n+1)^2} \E[\norm{\by^{(t)}}^2\norm{\by^{(t')}}^2].
	\end{align*}
	For $t=t'$, using a similar argument we obtain 
	\begin{align*}
		\text{Var}\big\{[\hbS_{\rm{(mm)}} ]_{t,t'}\big\} 
		\ge~  \frac{p-1}{n^2(n+1)^2} \E[\norm{\by^{(t)}}^2\norm{\by^{(t')}}^2].
	\end{align*}
	Summing over all $(t,t')\in [T]\times [T]$ yields
	\begin{align*}
		\E \big[\fnorm{\hbS_{\rm{(mm)}} - \bS}^2\big] 
		&\ge \frac{p-2}{n^2(n+1)^2}  \sum_{t,t'}\E[\norm{\by^{(t)}}^2\norm{\by^{(t')}}^2]\\
		&= \frac{p-2}{n^2(n+1)^2} \E [\fnorm{\bY}^4]\\
		&\ge \frac{p-2}{n^2(n+1)^2} (\E [\fnorm{\bY}^2])^2\\
		&= \frac{p-2}{(n+1)^2} [\trace(\bS) + \norm{\bB^*}^2]^2.
	\end{align*}
\end{proof}


\subsection{Proof of \saferef{Theorem}{thm:covariance}}
\begin{proof}[Proof of \Cref{thm:covariance}]\label{proof:thm:covariance}
	Recall definition of $\hbS$ in \Cref{def: hbS}, and let $\bQ_1$, $\bQ_2$ be defined as in \Cref{prop: V1,prop: V6}. With $\bZ = \bX\bSigma^{-1/2}$, we obtain
	\begin{align*}
		&n^2\Big[\bQ_2 (\fnorm{\bH}^2 + \fnorm*{\bF}^2/n)n^{-\frac 12} - n^{-1} (n\bI_T - \hbA) \bQ_1 \big(\fnorm{\bS^{\frac 12}}\big(\fnorm*{\bF}^2/n + \trace(\bS)\big)^{\frac 12} n^{-\frac 12}\big) \Big]\\
		=~& \big( \bF^\top \bZ\bZ^\top \bF - \bF^\top\bF(p\bI_T - \hbA) + (n\bI_T -\hbA)\bH^\top\bZ^\top\bF\big) \\
		& - \big[(n\bI_T - \hbA) (\bF^\top\bF + \bH^\top\bZ^\top\bF- \bS (n\bI_T - \hbA))\big]\\
		=~& \big( \bF^\top \bZ\bZ^\top \bF - \bF^\top\bF (p\bI_T - \hbA)\big) - (n\bI_T - \hbA) (\bF^\top\bF  - \bS (n\bI_T - \hbA))\\
		=~& \bF^\top \bZ\bZ^\top \bF + \bF^\top\bF\hbA + \hbA\bF^\top\bF -(n+p)\bF^\top\bF + (n\bI_T - \hbA) \bS (n\bI_T - \hbA)
		\\
		=~& (n\bI_T - \hbA) \bS (n\bI_T - \hbA) + \bF^\top\bF\hbA + \hbA\bF^\top\bF - \bF^\top((n+p) \bI_T - \bZ\bZ^\top )\bF\\
		=~& (n\bI_T - \hbA) \bS (n\bI_T - \hbA) - (n\bI_T - \hbA) \hbS (n\bI_T - \hbA)\\
		=~& (n\bI_T - \hbA) (\bS - \hbS) (n\bI_T - \hbA).
	\end{align*}
	Therefore, by triangle inequality and $\opnorm*{\bI_T - \hbA/n}\le 1$ in \Cref{lem: aux-tau>0,lem: aux-tau=0},
	\begin{align*}
		&\fnorm[\big]{(\bI_T - \hbA/n) (\bS - \hbS) (\bI_T - \hbA/n)}\\
		\le~& \fnorm*{\bQ_2}n^{-\frac 12} (\fnorm{\bH}^2 + \fnorm*{\bF}^2/n) +\opnorm*{\bI_T - \hbA/n}\fnorm*{\bQ_1}n^{-\frac 12} \fnorm{\bS^{\frac 12}}\big(\fnorm*{\bF}^2/n + \trace(\bS)\big)^{\frac 12} \\
		\le~& \fnorm*{\bQ_2}n^{-\frac 12} (\fnorm{\bH}^2 + \fnorm*{\bF}^2/n) +\fnorm*{\bQ_1}n^{-\frac 12} \fnorm{\bS^{\frac 12}}\big(\fnorm*{\bF}^2/n + \trace(\bS)\big)^{\frac 12} \\
		\le~& \fnorm*{\bQ_2}n^{-\frac 12} (\fnorm{\bH}^2 + \fnorm*{\bF}^2/n) +\fnorm*{\bQ_1}n^{-\frac 12} \frac12 \big[\trace(\bS)+ \big(\fnorm*{\bF}^2/n + \trace(\bS)\big)\big] \\
		\le~ & (\fnorm*{\bQ_2} + \fnorm*{\bQ_1}) n^{-\frac 12} (\fnorm{\bH}^2 + \fnorm*{\bF}^2/n + \trace(\bS)).
	\end{align*}
Therefore, 
$$
\fnorm[\big]{(\bI_T - \hbA/n) (\bS -\hbS)(\bI_T - \hbA/n)} \le \Theta_1 n^{-\frac 12} (\fnorm{\bH}^2 + \fnorm*{\bF}^2/n + \trace(\bS)),
$$
where $\Theta_1 = \fnorm*{\bQ_1} + \fnorm*{\bQ_2}$. 
Note that we have $\E[ \fnorm*{\bQ_1}^2] \le 4$ from \Cref{prop: V1}. 
By \Cref{prop: V6}, we have 
		\begin{itemize}
			\item[(1)] under \Cref{assu: tau}(i), $\E [\fnorm*{\bQ_2}^2] \le C(\tau')(T \wedge (1 + \frac pn))(1 + \frac pn)$. 
			Hence 
			\begin{align*}
				\E [\Theta_1^2] \le 2\E[\fnorm*{\bQ_1}^2 +\fnorm*{\bQ_2}^2]
				&\le 2 [4 + C(\tau')(T \wedge (1 + \frac pn))(1 + \frac pn)]\\ 
				&\le C(\tau')(T \wedge (1 + \frac pn))(1 + \frac pn).
			\end{align*}
			
			\item[(2)] under \Cref{assu: tau}(ii), $\E [I(\Omega)\fnorm*{\bQ_2}^2] \le C(\gamma, c)$ with $\P(\Omega)
			\to 1$. Thus, 
			 $\E [I(\Omega)\Theta_1^2] \le 2\E[\fnorm*{\bQ_1}^2 +I(\Omega)\fnorm*{\bQ_2}^2]\le 
			 C(\gamma, c)
			 $. 
		\end{itemize}
\end{proof}

\subsection{Proof of \saferef{Theorem}{thm:generalization error}}
\begin{proof}[Proof of \Cref{thm:generalization error}]\label{proof:thm:generalization error}
	From the definitions of $\bQ_1, \bQ_2, \bQ_3$ in \Cref{prop: V1,,prop: V6,prop: V8}, we have
	\begin{align*}
		&(\bQ_2^\top + \bQ_3) n^{-\frac 12} (\fnorm{\bH}^2 + \fnorm*{\bF}^2/n) + (\bI_T - \hbA/n) \bQ_1 n^{-\frac 12}\fnorm{\bS^{\frac 12}}\big(\fnorm*{\bF}^2/n + \trace(\bS)\big)^{\frac 12}\\
		=~& \frac{1}{n} \bF^\top \bF - (\bI_T - \hbA/n) (\bH^\top\bH + \bS) (\bI_T - \hbA/n)\\
		=~& (\bI_T - \hbA/n)\big[n^{-1}(\bI_T - \hbA/n)^{-1} \bF^\top \bF (\bI_T - \hbA/n)^{-1} - (\bH^\top\bH + \bS)\big](\bI_T - \hbA/n)\\
		=~& (\bI_T - \hbA/n) (\hbR - \bR) (\bI_T - \hbA/n),
	\end{align*}
	where $\hbR \defas n^{-1}(\bI_T - \hbA/n)^{-1}\bF^\top \bF (\bI_T - \hbA/n)^{-1}$, and $\bR \defas \bH^\top\bH + \bS$.
	
	Therefore, by triangle inequality and $\opnorm*{\bI_T - \hbA/n}\le 1$ from \Cref{lem: aux-tau>0,lem: aux-tau=0},
	\begin{align*}
		&\fnorm[\big]{(\bI_T - \hbA/n) (\hbR - \bR) (\bI_T - \hbA/n)} \\
		\le~& (\fnorm*{\bQ_2} + \fnorm*{\bQ_3}) n^{-\frac 12} (\fnorm{\bH}^2 + \fnorm*{\bF}^2/n) + \fnorm*{\bQ_1}n^{-\frac 12} \fnorm{\bS^{\frac 12}}\big(\fnorm*{\bF}^2/n + \trace(\bS)\big)^{\frac 12} \\
		\le~& (\fnorm*{\bQ_2} + \fnorm*{\bQ_3} + \fnorm*{\bQ_1}) n^{-\frac 12} (\fnorm*{\bF}^2/n +\fnorm{\bH}^2 + \trace(\bS))\\
		=~& \Theta_2 n^{-\frac 12} (\fnorm*{\bF}^2/n +\fnorm{\bH}^2 + \trace(\bS)), 
	\end{align*}
	where $\Theta_2 = \fnorm*{\bQ_1} + \fnorm*{\bQ_2} + \fnorm*{\bQ_3}$. 
	By \Cref{prop: V1,,prop: V6,prop: V8}, we obtain 
$\E [\Theta_2] \le C(\gamma, \tau')$ under \Cref{assu: tau}(i)  
and 
$\E [I(\Omega)\Theta_2] \le C(\gamma, c)$ with $P(\Omega)\to1$ under \Cref{assu: tau}(ii). 

	Furthermore, since $\Theta_2 = O_P(1)$, and $\opnorm*{(\bI_T - \hbA/n)^{-1}} = O_P(1)$ from \Cref{lem: aux-tau>0,lem: aux-tau=0},
	\begin{align*}
		\fnorm{\hbR - \bR} &\le \opnorm*{(\bI_T - \hbA/n)^{-1}}^2 \Theta_2 n^{-\frac 12} (\fnorm*{\bF}^2/n +\fnorm{\bH}^2 + \trace(\bS))\\
		&= O_P(n^{-\frac 12})  (\fnorm*{\bF}^2/n +\fnorm{\bH}^2 + \trace(\bS)).
	\end{align*}
	Since $\frac{1}{n} \bF^\top \bF = (\bI_T - \hbA/n) \hbR (\bI_T - \hbA/n)$, taking trace of both sides gives $\frac 1n \fnorm*{\bF}^2 \le \norm{\hbR}_*$ thanks to $\opnorm{(\bI_T - \hbA/n)} \le1$. Note that $\norm{\bR}_* = \fnorm{\bH}^2 + \trace(\bS)$ by definition of $\bR$, we obtain 
	\begin{align}\label{eq: R-R}
		\fnorm{\hbR - \bR} &\le O_P(n^{-\frac 12})  (\norm{\hbR}_* + \norm{\bR}_*).
	\end{align}
	
	Since $\hbR$ and $\bR$ are both $T\times T$ positive semi-definite matrices, whose ranks are at most $T$, 
	\begin{align*}
		&\big|\|\hbR\|_* - \|\bR\|_*\big| \le  \|\hbR - \bR\|_*
		\le \sqrt{2T}\fnorm{\hbR - \bR}\\
		\le~& O_P( (T/n)^{\frac 12}) (\norm{\hbR}_* + \norm{\bR}_*) = o_P(1) (\norm{\hbR}_* + \norm{\bR}_*),
	\end{align*}
	thanks to $T = o(n)$. That is, 
	\begin{align*}
		\frac{\big|\|\hbR\|_* - \|\bR\|_*\big|}{\|\hbR\|_*+ \|\bR\|_*} \le O_P( (T/n)^{\frac 12}),
	\end{align*}
	which implies 
	$\frac{\|\bR\|_*}{\|\hbR\|_*} -1 = O_P( (T/n)^{\frac 12})$, \ie,
	\[
	\frac{\trace(\bS)+ \fnorm{\bH}^2}{\fnorm{(\bI_T - \hbA/n)^{-1}\bF^\top}^2/n} -1 = O_P( (T/n)^{\frac 12}) = o_P(1).
	\]
\end{proof}

\subsection{Proof of \saferef{Theorem}{thm:main}}
\begin{proof}[Proof of \Cref{thm:main}]
	This proof is based on results of \Cref{thm:covariance,thm:generalization error}. 
	We begin with the result of  \Cref{thm:generalization error}, 
	\begin{equation*}
		\frac{\trace(\bS)+ \fnorm{\bH}^2}{\fnorm{(\bI_T - \hbA/n)^{-1}\bF^\top}^2/n} \overset{p}\to 1. 
	\end{equation*}
	In other words, 
	\[
	\trace(\bS)+ \fnorm{\bH}^2 = (1 + o_P(1)) \fnorm{(\bI_T - \hbA/n)^{-1}\bF^\top}^2/n.
	\]
	Thus, the upper bound in \Cref{thm:covariance} can be bounded from above as follows
	\begin{align*}
		&\fnorm{(\bI_T - \hbA/n) (\hbS -\bS)(\bI_T - \hbA/n)} \\
		\le~&\Theta_1 n^{-\frac 12} ( \fnorm*{\bF}^2/n + \fnorm{\bH}^2 + \fnorm{\bS^{\frac 12}}^2) \\
		\le~& \Theta_1 n^{-\frac 12} (\fnorm*{\bF}^2/n + (1 + o_P(1)) \fnorm{(\bI_T - \hbA/n)^{-1}\bF^\top}^2/n)\\
		\le~& \Theta_1 n^{-\frac 12} \big(1 + (1 + o_P(1)) \opnorm{(\bI_T -\hbA)^{-1}}^2\big)\fnorm*{\bF}^2/n\\
		=~& O_P(n^{-\frac12}) \fnorm*{\bF}^2/n,
	\end{align*}
	Using $\opnorm{(\bI_T -\hbA)^{-1}}= O_P(1)$ again, it follows
	\begin{align}\label{eq: pf1}
		\fnorm{\hbS -\bS} \le O_P(n^{-\frac12})  \fnorm{\bF}^2/n.
	\end{align}
	A similar argument leads to 
	\begin{align}\label{eq: pf2}
		\fnorm{\hbS -\bS} \le O_P(n^{-\frac12}) (\trace(\bS) + \fnorm{\bH}^2).
	\end{align}
	\end{proof}

\subsection{Proof of \saferef{Corollary}{cor36}}
	\begin{proof}[Proof of \Cref{cor36}]
		Under \Cref{assu: tau}(i)  and \ref{assu: SNR}, we proceed to bound $\fnorm{\bF}^2/n$ in terms of $\trace(\bS)$.  
		Let $L(\bB)  = \frac{1}{2n}\fnorm*{\bY - \bX\bB }^2 + \lambda \norm{\bB}_{2,1} + \frac{\tau}{2} \fnorm{\bB}^2$ be the objective function in \eqref{eq: hbB}, then $L(\hbB) \le L(\bf0)$ by definition of $\hbB$. Thus, 
	\begin{align*}
		&\frac{1}{2n}\fnorm{\bF }^2 \le \frac{1}{2n}\fnorm{\bF }^2 + \lambda \norm{\hbB}_{2,1}
		+ \frac{\tau}{2} \fnorm{\hbB}^2 \le\frac{1}{2n}\fnorm{\bY }^2.
	\end{align*}
	Now we bound $\frac1n\fnorm{\bY}^2$ by Hanson-Wright inequality. 
	Since $\bY = \bX\bB^* + \bE$, the rows of $\bY$ are \iid $\calN_T(\bf0, \bSigma_{\by})$ with $\bSigma_{\by} = (\bB^*)^\top \bSigma \bB^*+ \bS$, then $\vec(\bY^\top) \sim \calN(\bf0, \bI_n \otimes \bSigma_{\by})$, and $\bxi \defas [\bI_n \otimes \bSigma_{\by}]^{-\frac12} \vec(\bY^\top)\sim \calN(\mathbold{0}, \bI_{nT})$. 
	Since 
	$\fnorm{\bY}^2 = [\vec(\bY^\top)]^\top \vec(\bY^\top) = \bxi^\top (\bI_n \otimes \bSigma_{\by}) \bxi$,
	 we apply the following variant of Hanson-Wright inequality.
	\begin{lemma}[Lemma 1 in \cite{laurent2000adaptive}]
		For $\bxi\sim \calN(\mathbold{0}, \bI_N)$, then 
		\begin{align*}
			\P(\bxi^\top \bA\bxi - \trace(\bA) \le 2 \sqrt{x}\fnorm{\bA} + 2x\opnorm{\bA}) \ge1 - \exp(-x).
		\end{align*}
	\end{lemma}
	In our case, take $\bA = (\bI_n \otimes \bSigma_{\by})$, then $\trace(\bA) = n\trace(\bSigma_{\by})$, $\fnorm{\bA}  =  \sqrt{n}\fnorm{\bSigma_{\by}}\le \sqrt{n} \trace(\bSigma_{\by})$, $\opnorm{\bA} = \opnorm{\bSigma_{\by}}\le \trace(\bSigma_{\by})$, 
	thus with probability at least $1 - \exp(-x)$, 
	\begin{align*}
		\fnorm{\bY}^2 - n\trace(\bSigma_{\by})
		\le 2 \sqrt{nx}\trace(\bSigma_{\by}) + 2x \trace(\bSigma_{\by}).
	\end{align*}
	
	Take $x=n$, then with probability at least $1 - \exp(-n)$,
	\begin{align*}
		\fnorm{\bF}^2/n \le  \fnorm{\bY}^2/n \le  5\trace(\bSigma_{\by}).
	\end{align*}
	Thus, $\fnorm{\bF}^2/n = O_P(1)\trace(\bSigma_{\by}).$ Together with \eqref{eq: pf1}, we obtain 
	\begin{equation*}
		\fnorm{\hbS -\bS} \le O_P(n^{-\frac12}) \trace(\bSigma_{\by}).
	\end{equation*}
	Note that by \Cref{assu: SNR}, $\trace(\bSigma_{\by}) = \fnorm{\bSigma^{\frac12} \bB^*}^2 + \trace(\bS) \le (1 + \mathfrak{snr})\trace(\bS).$
	Therefore, we obtain 
	\begin{equation*}
		\fnorm{\hbS -\bS} \le O_P(n^{-\frac12}) \trace(\bS).
	\end{equation*}
	Furthermore, since $\trace(\bS)\le \sqrt{T}\fnorm{\bS}$ and $T = o(n)$, we have
	$$
	\fnorm{\hbS -\bS} \le  O_P(\sqrt{ T/n}) \fnorm{\bS} = o_P(1) \fnorm{\bS}.
	$$
	Finally, since $\norm{\bS}_* = \trace(\bS)$, by triangular inequality
	\begin{equation*}
		\big|\norm{\hbS}_* - \trace(\bS)\big| 
		\le \norm{\hbS - \bS}_*
		\le \sqrt{T}\fnorm{\hbS - \bS}
		\le O_P(\sqrt{T/n}) \trace(\bS) 
		= o_P(1) \trace(\bS). 
	\end{equation*}
\end{proof}

\subsection{Proof of \saferef{Corollary}{cor37}}
\begin{proof}[Proof of \Cref{cor37}]
	For $\tau=0$, by the optimality of $\hbB$ in \eqref{eq: hbB},
	$$
	\frac{1}{2n}\fnorm{\bF}^2 + \lambda\|\hbB\|_{2,1}
	\le 
	\frac{1}{2n}\fnorm{\bE}^2  + \lambda \|\bB^*\|_{2,1}.
	$$
	Note that $\bF = \bE - \bX(\hbB - \bB^*)= \bE - \bZ\bH$, expanding the squares and rearranging terms yields
	\begin{equation}\label{eq:boundH1}
		\fnorm{\bZ\bH}^2 \le 2\langle \bE, \bZ\bH\rangle + 2n\lambda (\|\bB^*\|_{2,1} - \|\hbB\|_{2,1})  \le 2\langle \bE, \bZ\bH\rangle + 2n\lambda \|\hbB - \bB^*\|_{2,1}. 
	\end{equation}
	From assumptions in this corollary, $\hbB - \bB^*$ has at most $(1-c)n$ rows. Thus, in the event $U_2$, we have
	$$
	n\eta \fnorm{\bH}^2 = n\eta \fnorm{\bSigma^{1/2}(\hbB - \bB^*)}^2\le \fnorm{\bX(\hbB - \bB^*)}^2 = \fnorm{\bZ\bH}^2. 
	$$
	We bound the right-hand side two terms in \eqref{eq:boundH1} by Cauchy-Schwarz inequality,
	$$
	\quad \|\hbB - \bB^*\|_{2,1}\le \sqrt{(1-c)n} \fnorm{\hbB - \bB^*}\le  \frac{\sqrt{(1-c)n}}{\sqrt{\phi_{\min}(\bSigma)}} \fnorm{\bH} \le \frac{\sqrt{1-c}}{\sqrt{\eta \phi_{\min}(\bSigma)}} \fnorm{\bZ\bH},
	$$
	and $	\langle \bE, \bZ\bH\rangle \le \fnorm{\bE}  \fnorm{\bZ\bH} \le \fnorm{\bS^{\frac 12}} \opnorm{\bE\bS^{-\frac 12}} \fnorm{\bZ\bH}.$
	
	Therefore, by canceling a factor $\fnorm{\bZ\bH}$ from both sides of \eqref{eq:boundH1}, we have
	\begin{align*}
		\sqrt{n\eta} \fnorm{\bH} \le \fnorm{\bZ\bH} \le 2 \fnorm{\bS^{\frac 12}} \opnorm{\bE\bS^{-\frac 12}}  + \frac{2\sqrt{(1-c)}n\lambda}{\sqrt{\eta \phi_{\min}(\bSigma)}}.
	\end{align*}
	Using $(a + b)^2 \le 2a^2 + 2b^2$,  
	\begin{align*}
		\fnorm{\bH}^2 \le \frac{4}{n\eta} \trace(\bS) \opnorm{\bE\bS^{-\frac 12}}^2 + \frac{4(1-c)n\lambda^2}{\eta^2 \phi_{\min}(\bSigma)}.
	\end{align*}
	Hence, using $\lambda$ is of the form $\mu\sqrt{\trace(\bS)/n}$, we have 
	\begin{align}
		&\trace(\bS) + \fnorm{\bH}^2 \\
		\le~& (1 +4 \eta^{-1} n^{-1}\opnorm{\bE\bS^{-\frac 12}}^2) \trace(\bS) + \frac{4(1-c)\mu^2}{\eta^2 \phi_{\min}(\bSigma)}\trace(\bS)\\
		\le~& O_P(1) (1 + \mu^2)\trace(\bS),
	\end{align}
where we used that $n^{-1}\opnorm{\bE\bS^{-\frac 12}} = O_P(1)$ by \cite[Theorem II.13]{DavidsonS01} and $T = o(n)$. 
	Now, by \Cref{thm:main}, 
	\begin{align*}
		\fnorm{\hbS - \bS} \le O_P(n^{-\frac12}) [\trace(\bS) + \fnorm{\bH}^2]\le O_P(n^{-\frac12}) (1 + \mu^2)\trace(\bS), 
	\end{align*}
where the $O_P(\cdot) $ hides constants depending on $\gamma, c, \phi_{\min}(\bSigma)$ since $\eta$ is a constant that only depends on $\gamma, c$. 
\end{proof}

\subsection{Proof of \saferef{Theorem}{thm:out-of-sample}}\label{sec: out-of-sample proof}

\begin{proof}[Proof of \Cref{thm:out-of-sample}] \label{proof:thm:out-out-sample}
	From the definitions of $\bQ_2, \bQ_3$ in \Cref{prop: V6,prop: V8}, we have
	\begin{align*}
		&\bQ_2 + \bQ_2^\top + \bQ_3\\
		=~& \frac{n^{-2}\big( \bF^\top \bZ\bZ^\top \bF + 
			\hbA \bF^\top \bF + \bF^\top \bF \hbA - 
			p\bF^\top \bF -  (n\bI_T -\hbA)\bH^\top\bH (n\bI_T -\hbA)}{(\fnorm{\bH}^2 + \fnorm*{\bF}^2/n)n^{-\frac 12}}.
	\end{align*}
	Therefore, 
	\begin{align*}
		&\fnorm[\big]{
			(\bI_T -\hbA/n)\bH^\top\bH (\bI_T -\hbA/n) -  n^{-2} \big( \bF^\top \bZ\bZ^\top \bF + \hbA \bF^\top \bF + \bF^\top \bF \hbA - 
			p\bF^\top \bF\big)}\\
		=~& \fnorm{\bQ_2 +\bQ_2^\top + \bQ_3} (\fnorm{\bH}^2 + \fnorm*{\bF}^2/n)n^{-\frac 12}\\
		\le ~&\Theta_3(\fnorm{\bH}^2 + \fnorm*{\bF}^2/n)n^{-\frac 12},
	\end{align*}
	where $\Theta_3= 2\fnorm*{\bQ_2} + \fnorm*{\bQ_3}$. The conclusion thus follows by \Cref{prop: V6,prop: V8}.
\end{proof}

\section{Proofs of preliminary results}
\label{sec:appendix-preliminary}

\subsection{Proofs of results in \saferef{Appendix}{sec:opnorm-bound}}

\begin{proof}[Proof of \Cref{lem: aux-tau>0}]
	
	\begin{itemize}
	\item[(i)] For any $\bu\in \R^T$, by defintion \eqref{eq: hbA-matrix},
	\begin{align*}
		\bu^\top\hbA\bu &= \trace\big[ (\bu^{\top} \otimes \bX_{\hat{\mathscr{S}}})
		\bM^{\dagger}
		(\bu \otimes \bX^\top_{\hat{\mathscr{S}}})\big]\\
		&\le \trace\big[ (\bu^{\top} \otimes \bX_{\hat{\mathscr{S}}})
		[\bI_T \otimes (\bX_{\hat{\mathscr{S}}}^\top\bX_{\hat{\mathscr{S}}} + n\tau\bP_{\hat{\mathscr{S}}})^{\dagger}]
		(\bu \otimes \bX_{\hat{\mathscr{S}}}^\top)\big]\\
		&= \trace\big[ 
		(\bu^{\top}\bI_T \bu) \otimes [\bX_{\hat{\mathscr{S}}}(\bX_{\hat{\mathscr{S}}}^\top\bX_{\hat{\mathscr{S}}} + n\tau\bP_{\hat{\mathscr{S}}})^{\dagger}\bX_{\hat{\mathscr{S}}}^\top]\big]\\
		&= \norm*{\bu}^2\trace[\bX_{\hat{\mathscr{S}}}(\bX_{\hat{\mathscr{S}}}^\top\bX_{\hat{\mathscr{S}}} + n\tau\bP_{\hat{\mathscr{S}}})^{\dagger}\bX_{\hat{\mathscr{S}}}^\top]\\
		&= \norm*{\bu}^2\trace[\bX_{\hat{\mathscr{S}}}^\top\bX_{\hat{\mathscr{S}}}(\bX_{\hat{\mathscr{S}}}^\top\bX_{\hat{\mathscr{S}}} + n\tau\bP_{\hat{\mathscr{S}}})^{\dagger}]. 
	\end{align*}
	Let $r = \rank(\bX_{\hat{\mathscr{S}}})\le \min(n, |\hat{\mathscr{S}}|)$ be the rank of $\bX_{\hat{\mathscr{S}}}$, and $\hphi_1\ge\cdots\ge \hphi_{r}>0$ be the nonzero eigenvalues of $\frac 1n \bX_{\hat{\mathscr{S}}}^\top\bX_{\hat{\mathscr{S}}}$. We have
	\begin{align*}
		\opnorm{\hbA/n}&\le \frac 1n \trace[\bX_{\hat{\mathscr{S}}}^\top\bX_{\hat{\mathscr{S}}}(\bX_{\hat{\mathscr{S}}}^\top\bX_{\hat{\mathscr{S}}} + n\tau\bP_{\hat{\mathscr{S}}})^{\dagger}]\\
		&= \frac 1n \trace[\frac 1n \bX_{\hat{\mathscr{S}}}^\top\bX_{\hat{\mathscr{S}}}(\frac 1n \bX_{\hat{\mathscr{S}}}^\top\bX_{\hat{\mathscr{S}}} + \tau\bP_{\hat{\mathscr{S}}})^{\dagger}]\\
		&\le \frac rn \opnorm[\Big]{\frac 1n \bX_{\hat{\mathscr{S}}}^\top\bX_{\hat{\mathscr{S}}}(\frac 1n \bX_{\hat{\mathscr{S}}}^\top\bX_{\hat{\mathscr{S}}} + \tau\bP_{\hat{\mathscr{S}}})^{\dagger}}\\
		&\le \frac{\hphi_1}{\hphi_1 + \tau}\le 1.
	\end{align*}
	Thus, $\opnorm{\bI - \hbA/n}\le1$ as $\hbA$ is positive semi-definite. 
	
	\item[(ii)] Note that 
	$
	\opnorm{(\bI_T - \hbA/n)^{-1}} = (1- \opnorm*{\hbA/n})^{-1} \le 1 +  \frac{\hphi_1}{\tau},
	$
	where
	$
	\hphi_1 
	= \opnorm{\frac{1}{n} \bX_{\hat{\mathscr{S}}}^\top \bX_{\hat{\mathscr{S}}}} 
	\le  \opnorm{\frac{1}{n} \bX^\top \bX}\le \frac{1}{n}\opnorm{ \bX^\top\bSigma^{-\frac 12}}^2\opnorm{\bSigma}. 
	$
	Therefore, 
	\begin{itemize}
		\item[(1)] in the event $\{\opnorm{\bX\bSigma^{-\frac12}} <  2\sqrt{n}+\sqrt{p}\}$,  we have
		$$
		\opnorm{(\bI_T - \hbA/n)^{-1}} \le 1 + \tau^{-1} (2 + \sqrt{p/n})^2\opnorm{\bSigma} = 1 + (\tau')^{-1} (2 + \sqrt{p/n})^2. 
		$$
		\item[(2)]
		$\E[\hphi_1] \le  \E[n^{-1}\opnorm{ \bX^\top\bSigma^{-\frac 12}}^2\opnorm{\bSigma}] \le [(1 + \sqrt{p/n})^2 +n^{-1}]\opnorm{\bSigma}$  by \eqref{eq: opnorm-Z}.  
		Hence, 
		$$
		\E \opnorm{(\bI_T - \hbA/n)^{-1}}\le 1 + \tau^{-1} \E[\hphi_1]  \le 1 + (\tau')^{-1} [(1 + \sqrt{p/n})^2 +n^{-1}].
		$$
	\end{itemize}
	\end{itemize}

\end{proof}

\begin{proof}[Proof of \Cref{lem: aux-tau=0}]
	
	\begin{itemize}
		\item[(i)] For $\tau=0$, using the same arguement as proof of \Cref{lem: aux-tau>0}, we obtain 
		\begin{align*}
			\bu^\top\hbA\bu 
			\le 
			\norm{\bu}^2\trace[\bX_{\hat{\mathscr{S}}}^\top\bX_{\hat{\mathscr{S}}}(\bX_{\hat{\mathscr{S}}}^\top\bX_{\hat{\mathscr{S}}})^{\dagger}] 
			\le \norm{\bu}^2 |\hat{\mathscr{S}}|. 
		\end{align*}
		Thus, in the event $U_1$, we have $\opnorm{\hbA}/n \le |\hat{\mathscr{S}}|/n\le (1-c)/2<1$, hence 
		$$
		\opnorm{\bI_T - \hbA/n} \le 1. 
		$$
		\item[(ii)] In the event $U_1$, we have
		$\opnorm{(\bI_T - \hbA/n)^{-1}} = (1- \opnorm*{\hbA/n})^{-1}\le (1- (1-c)/2)^{-1}$. Furthermore, $\E [ I(U_1) \opnorm{(\bI_T - \hbA/n)^{-1}}]  \le (1- (1-c)/2)^{-1}.$
	\end{itemize}

\end{proof}
\begin{proof}[Proof of \Cref{lem: MN}]
	Since $\bM^\dagger\preceq \bM_1^\dagger = \bI_T \otimes (\bX_{\hat{\mathscr{S}}}^\top\bX_{\hat{\mathscr{S}}} + \tau n \bP_{\hat{\mathscr{S}}})^\dagger$, 
	\begin{align*}
		\opnorm{\bN} &= \opnorm{(\bI_T \otimes \bX)\bM^\dagger (\bI_T \otimes \bX^\top)}\\
		&\le \opnorm{(\bI_T \otimes \bX)(\bI_T \otimes (\bX_{\hat{\mathscr{S}}}^\top\bX_{\hat{\mathscr{S}}} + \tau n \bP_{\hat{\mathscr{S}}})^\dagger) (\bI_T \otimes \bX^\top)}\\
		&=\opnorm{\bX_{\hat{\mathscr{S}}}(\bX_{\hat{\mathscr{S}}}^\top\bX_{\hat{\mathscr{S}}} + \tau n \bP_{\hat{\mathscr{S}}})^\dagger \bX_{\hat{\mathscr{S}}}^\top}\\
		&\le 1,
	\end{align*}
	where the first inequality uses $\opnorm{\bA\bB\bA^\top} \le \opnorm{\bA\bC\bA^\top}$ for $0\preceq \bB \preceq \bC$.
\end{proof}

\subsection{Proofs of results in \saferef{Appendix}{sec:Lipschitz-fix-E}}

\begin{proof}[Proof of \Cref{lem: lipschitz elastic-net}]
	Fixing $\bE$, if $\bX,\bar\bX$ are two design matrices, and $\hbB, \bar\bB$ are the two corresponding multi-task elastic net estimates. 
	Let $\bZ = \bX \bSigma^{-\frac 12}$,  
	$\bar\bZ = \bar\bX \bSigma^{-\frac 12}$,
	$\bar\bH = \bSigma^{\frac12}(\bar\bB - \bB^*)$, 
	$\bar\bF = \bY - \bar\bX\bar\bB$, 
	and $\bar D = [\fnorm{\bar\bH}^2 + \fnorm{\bar\bF}^2/n]^{\frac12}$. 
	Without loss of generality, we assume $\bar D \le D$. 
	Recall the multi-task elastic net estimate $\hbB = \argmin_{\bB\in\R^{p\times T}}
	\big( \frac{1}{2n}\fnorm{\bY - \bX\bB }^2 + g(\bB) \big)$, where  $g(\bB) = \lambda \norm{\bB}_{2,1}
	+ \frac{\tau}{2} \fnorm{\bB}^2$.
	Define $\varphi:\bB\mapsto \frac 1{2n}
	\fnorm{\bE+\bX(\bB^*-\bB)}^2 + g(\bB)$, 
	$\psi:\bB\mapsto \frac 1{2n} \fnorm{\bX(\hbB-\bB)}^2$ 
	and $\zeta:\bB\mapsto \varphi(\bB) - \psi(\bB)$. 
	When expanding the squares, it is clear that $\zeta$ is the sum of a linear function and a $\tau$-strong convex penalty, thus $\zeta$ is $\tau$-strongly convex of $\bB$. Additivity of subdifferentials yields $\partial \varphi (\hbB) = \partial \zeta(\hbB) + \partial \psi(\hbB) = \partial \zeta(\hbB)$. By optimality of $\hbB$ we have ${\mathbf 0}_{p\times T}\in \partial \varphi(\hbB)$, thus ${\mathbf 0}_{p\times T}\in \partial \zeta(\hbB)$. By strong convexity of $\zeta$,
	$
	\zeta(\bar\bB) - \zeta(\hbB) 
	\ge \langle \partial \zeta(\hbB), \bar\bB - \hbB \rangle + \frac{\tau}{2} \fnorm{\bar\bB - \hbB}^2
	=  \frac{\tau}{2} \fnorm{\bar\bB - \hbB}^2, 
	$
	which can further be rewritten as 
	$$
	\fnorm{\bX (\hbB - \bar\bB)}^2 + n\tau \fnorm{\hbB -\bar\bB}^2
	\le \fnorm{\bE - \bX (\bar\bB - \bB^*)}^2 -  \fnorm{\bE - \bX (\hbB - \bB^*)}^2 + 2n(g(\bar\bB) - g(\hbB)), 
	$$
	\ie, 
	$$\fnorm{\bZ (\bH - \bar\bH)}^2 + n\tau \fnorm{\bSigma^{-\frac12}(\bH - \bar\bH)}^2 
	\le \fnorm{\bE - \bZ \bar\bH}^2 -  \fnorm{\bE - \bZ \bH}^2 + 2n(g(\bar\bB) - g(\hbB)).$$
	Summing the above inequality with its counterpart obtained by replacing $(\bX, \hbB, \bH)$ with $(\bar\bX, \bar\bB, \bar\bH)$, we have 
	\begin{align*}
		&(LHS) \\
		\defas~&\fnorm{\bZ (\bH - \bar\bH)}^2 + \fnorm{\bar\bZ (\bH - \bar\bH)}^2  + 2n\tau' \fnorm{\bH - \bar\bH}^2\\
		\le~& \fnorm{\bZ (\bH - \bar\bH)}^2 + \fnorm{\bar\bZ (\bH - \bar\bH)}^2  + 2n\tau \fnorm{\bSigma^{-\frac12}(\bH - \bar\bH)}^2\\
		\le~& \fnorm{\bE - \bZ \bar\bH}^2 -  \fnorm{\bE - \bZ \bH}^2  + \fnorm{\bE - \bar\bZ \bH}^2 -  \fnorm{\bE - \bar\bZ \bar\bH}^2\\
		=~& \langle \bZ(\bH - \bar\bH),\bF+\bar\bF +(\bar\bZ- \bZ)\bar\bH\rangle
		+ \langle -\bar\bZ(\bH - \bar\bH),\bF+\bar\bF +(\bZ-\bar\bZ)\bH\rangle\\
		=~& \langle (\bZ-\bar\bZ)(\bH - \bar\bH),\bF+\bar\bF\rangle + \langle \bZ(\bH - \bar\bH),(\bar\bZ- \bZ)\bar\bH\rangle + \langle \bar\bZ(\bar\bH - \bH), (\bZ-\bar\bZ)\bH\rangle\\
		\le~& \opnorm{\bZ-\bar\bZ}\fnorm{\bH - \bar\bH} (\fnorm{\bF} + \fnorm{\bar\bF}) 
		+ \opnorm{\bZ-\bar\bZ}\fnorm{\bZ (\bH - \bar\bH)}   \fnorm{\bar\bH}\\
		&+ \opnorm{\bZ-\bar\bZ}\fnorm{\bar\bZ(\bH - \bar\bH)} \fnorm{\bH}\\
		\le~&  \opnorm{\bZ-\bar\bZ}
		\Big[
		\sqrt{\frac{(LHS)}{2n\tau'}} (\fnorm{\bF} + \fnorm{\bar\bF}) + \sqrt{(LHS)} (\fnorm{\bar\bH} + \fnorm{\bH}) 
		\Big]\\
		\le~& \opnorm{\bZ-\bar\bZ}  \sqrt{(LHS)} (D + \bar D)\max(1, (2\tau')^{-\frac12})
	\end{align*}
	where $\tau' = \tau \phi_{\min}(\bSigma^{-1}) = \tau/\opnorm{\bSigma}$. 
	That is, 
	\begin{align*}
		\sqrt{(LHS)} 
		\le 
		\opnorm{\bZ-\bar\bZ} 2D \max(1, (2\tau')^{-\frac12}).
	\end{align*}
	Therefore,
	\begin{align*}
		&n^{-\frac12}\fnorm{\bF - \bar\bF} = n^{-\frac12}\fnorm{\bZ\bH - \bar\bZ\bar\bH}\\
		\le~& n^{-\frac12}[\fnorm{\bZ(\bH-\bar\bH)} + \fnorm{(\bZ- \bar\bZ)\bar\bH}]\\
		\le~&n^{-\frac12}[\fnorm{\bZ(\bH-\bar\bH)} + \opnorm{\bZ- \bar\bZ}\fnorm{\bH}]\\
		\le~&n^{-\frac12}[\sqrt{(LHS)}+ \opnorm{\bZ- \bar\bZ}D]\\
		\le~&n^{-\frac12}  \opnorm{\bZ- \bar\bZ} D[2 \max(1, (2\tau')^{-\frac12}) + 1].
	\end{align*}
	So far we obtained
	\begin{align*}
		\fnorm{\bH - \bar\bH} 
		&\le 
		\sqrt{\frac{(LHS)}{2n\tau'}}\le 
		n^{-\frac12} \opnorm{\bZ-\bar\bZ}D (2\tau')^{-\frac12} 2\max(1, (2\tau')^{-\frac12}),\\	
		n^{-\frac12}\fnorm{\bF - \bar\bF} 
		&\le n^{-\frac12}  \opnorm{\bZ- \bar\bZ} D[2\max(1, (2\tau')^{-\frac12}) + 1].
	\end{align*}
	Let $\bQ = [\bH^\top, \bF^\top/\sqrt{n}]^\top$ and $\bar \bQ = [\bar\bH^\top, \bar\bF^\top/\sqrt{n}]^\top$, then $D = \fnorm{\bQ}$, $\bar D = \fnorm{\bar \bQ}$. By triangular inequality,
	\begin{align*}
		|D - \bar D|\le \fnorm{\bQ-\bar \bQ} \le~& \fnorm{\bH-\bar\bH} + \fnorm{\bF-\bar\bF}/\sqrt{n} \\
		\le~ & n^{-\frac12}\opnorm{\bZ- \bar\bZ} D [4 \max(1, (2\tau')^{-1})],
	\end{align*} 
	where the last inequality uses the elementary inequality $\max(a,b)(a+b)\le 2 [\max(a,b)]^2$ for $a,b>0$ with $a = 1, b = (2\tau')^{-\frac12}$. 
	Let $\frac{\partial D}{\partial \bZ}\defas \frac{\partial D}{\partial \vec(\bZ)} \in \R^{1\times np}$, then $\norm*{\frac{\partial D}{\partial \bZ}} \le n^{-\frac12}D L_1$ with $L_1 = [4 \max(1, (2\tau')^{-1})]$. Hence, 
	\begin{align*}
		\sum_{ij} \Big(\frac{\partial D}{\partial z_{ij}}\Big)^2 
		=  \norm*{\frac{\partial D}{\partial \bZ}}^2
		\le n^{-1}D^2 L_1^2.
	\end{align*}
	Furthermore, by triangle inequality
	\begin{align*}
		\fnorm[\Big]{\frac{\bQ}{D} - \frac{\bar \bQ}{\bar D}}
		& \le \frac1D \fnorm{\bQ-\bar \bQ} + \Big|\frac1D - \frac{1}{\bar D} \Big| \fnorm{\bar \bQ}\\
		& = \frac1D \fnorm{\bQ-\bar \bQ} + \frac{|D-\bar D|}{D\bar D} \fnorm{\bar \bQ}\\
		&\le \frac1D \fnorm{\bQ-\bar \bQ} + \frac{1}{D} \fnorm{\bQ - \bar \bQ}\\
		&\le n^{-\frac12}\opnorm{\bZ- \bar\bZ} L,
	\end{align*}
	where $L = 8 \max(1, (2\tau')^{-1})$. 
	Therefore, when $\tau >0$, we obtain the two mappings $\bZ \mapsto D^{-1}\bF/\sqrt{n}$, and $\bZ \mapsto D^{-1}\bH$ are both $n^{-\frac12} L$-lipschitz with $L = 8 \max(1, (2\tau')^{-1})$, where $\tau' = \tau/\opnorm{\bSigma}$.
\end{proof}

The proof of \Cref{lem: lipschitz-Lasso} uses a similar argument as proof of \Cref{lem: lipschitz elastic-net}, we present it here for completeness. 
\begin{proof}[Proof of \Cref{lem: lipschitz-Lasso}]
	For multi-task group Lasso ($\tau=0$), we restrict our analysis in the event $U_1\cap U_2$, where $U_1 = \big\{ \norm{\hbB}_0 \le n(1-c)/2 \big\}$, $U_2 = \big\{\inf_{\bb\in \R^p: \| \bb\|_0 \le (1-c)n}  \|\bX \bb\|^2/(n \|\bSigma^{\frac 12} \bb\|^2)  > \eta\big\}.$ 
	
	Since the only randomness of the problem comes from $\bX$ and $\bE$, there exists a measurable set $\calU$ such that $U_1\cap U_2 =\{ (\bX, \bE)\in \calU\}$. 
	For some noise matrix $\bE$, consider  $\bX,\bar\bX$ two design matrices such that $(\bX, \bE)\in \calU$ and $(\bar\bX, \bE)\in \calU$. 
	We slightly abuse the notation and let $\hbB, \bar\bB$ denote the two corresponding multi-task group-Lasso estimates. 
	Thus, the row sparsity of $ \hbB-\bar\bB$ is at most $n(1-c)$. 
	Let  
	$\bar\bH = \bar\bB - \bB^*$, $\bar\bF = \bY - \bar\bX\bar\bB$, and $\bar D = [\fnorm{\bar\bH}^2 + \fnorm{\bar\bF}^2/n]^{\frac12}$. 
	Without loss of generality, we assume $\bar D \le D$. 
	Since when $\tau=0$, the multi-task group Lasso estimate is $\hbB = \argmin_{\bB\in\R^{p\times T}}
	\big( \frac{1}{2n}\fnorm{\bY - \bX\bB }^2 + g(\bB) \big)$, where  $g(\bB) = \lambda \norm{\bB}_{2,1}$.
	Define $\varphi:\bB\mapsto \frac 1{2n}
	\fnorm{\bE+\bX(\bB^*-\bB)}^2 + g(\bB)$, 
	$\psi:\bB\mapsto \frac 1{2n} \fnorm{\bX(\hbB-\bB)}^2$ 
	and $\zeta:\bB\mapsto \varphi(\bB) - \psi(\bB)$. 
	Under $\tau=0$, by the same arguments in proof of \ref{lem: lipschitz elastic-net} with the same functions $\varphi(\cdot), \psi(\cdot), \zeta(\cdot)$, we obtain
	$$
	\fnorm{\bX (\hbB - \bar\bB)}^2 
	\le \fnorm{\bE - \bX (\bar\bB - \bB^*)}^2 -  \fnorm{\bE - \bX (\hbB - \bB^*)}^2 + 2n(g(\bar\bB) - g(\hbB)).
	$$
	Summing the above inequality with its counterpart obtained by replacing $(\bX, \hbB, \bH)$ with $(\bar\bX, \bar\bB, \bar\bH)$, we have 
	\begin{align*}
		&\fnorm{\bX (\hbB - \bar\bB)}^2 + \fnorm{\bar\bX (\hbB - \bar\bB)}^2\\
		\le~& \fnorm{\bE - \bZ \bar\bH}^2 - \fnorm{\bE - \bZ \bH}^2
		+ \fnorm{\bE - \bar\bZ \bH}^2 - \fnorm{\bE - \bar\bZ \bar\bH}^2.
	\end{align*}
	Note that in event $U_1\cap U_2$, we have  
	\begin{align*}
		\eta n\fnorm{\bSigma^{\frac12}(\hbB - \bar\bB)}^2 
		\le \fnorm{\bX (\hbB - \bar\bB)}^2, \quad \eta n\fnorm{\bSigma^{\frac12}(\hbB - \bar\bB)}^2 \le \fnorm{\bar\bX (\hbB - \bar\bB)}^2. 
	\end{align*}
	Thus,  
	$
	2\eta n\fnorm{(\hbH - \bar\bH)}^2 \le \fnorm{\bZ(\bH - \bar\bH)}^2 + \fnorm{\bar\bZ(\bH - \bar\bH)}^2
	$, and 
	\begin{align*}
		(LHS) 
		\defas~& \max (2\eta n\fnorm{\bH - \bar\bH}^2,\fnorm{\bZ(\bH - \bar\bH)}^2 + \fnorm{\bar\bZ(\bH - \bar\bH)}^2)\\
		=~&\fnorm{\bZ(\bH - \bar\bH)}^2 + \fnorm{\bar\bZ(\bH - \bar\bH)}^2\\
		\le~& \fnorm{\bE - \bZ \bar\bH}^2 - \fnorm{\bE - \bZ \bH}^2
		+ \fnorm{\bE - \bar\bZ \bH}^2 - \fnorm{\bE - \bar\bZ \bar\bH}^2.
	\end{align*}
	Now, in $U_1\cap U_2$, the Lipschitz property of the map $\bZ \mapsto D^{-1}\bF/\sqrt{n}$ follows from the same arguments in proof of \Cref{lem: lipschitz elastic-net}, with $\tau'$ in \ref{lem: lipschitz elastic-net} replaced by $\eta$ in this proof. 
	
	Furthermore, in the event $U_1\cap U_2\cap U_3$, the Lipschitz property of $\bZ \mapsto D^{-1}\bZ^\top\bF/n$ follows by triangle inequality. To see this, let $\bU = D^{-1}\bF/\sqrt{n}$, and $\bV = D^{-1}\bZ^\top\bF/n = n^{-1/2} \bZ^\top \bU$, thus by  triangle inequality
	\begin{align*}
		\opnorm{\bV - \bar\bV} &= n^{-1/2} \opnorm{\bZ^\top \bU - \bar\bZ^\top \bar\bU}\\
		&= n^{-1/2}[ \opnorm{(\bZ - \bar\bZ)^\top \bU} + \opnorm{\bar\bZ^\top (\bU - \bar\bU)}]\\
		&\le n^{-1/2}[ \opnorm{\bZ - \bar\bZ} + \opnorm{\bar\bZ}\opnorm{\bU - \bar\bU}]\\
		&\le n^{-1/2}( 1 + n^{-1/2}\opnorm{\bar\bZ}L) 
		\opnorm{\bZ - \bar\bZ}\\
		&\le n^{-1/2} (1 + (2 +\sqrt{p/n})L).
	\end{align*}
	where the last line uses $\opnorm{\bar\bZ}\le 2\sqrt{n} +\sqrt{p} $ in the event $U_3$. 
\end{proof}

\begin{proof}[Proof of \Cref{cor: partialD}]
	\Cref{cor: partialD} (1) is a direct consequence of the intermediate result $|D - \bar D| \le  n^{-\frac12}\opnorm{\bZ- \bar\bZ} D [4 \max(1, (2\tau')^{-1})]$ in proof of \Cref{lem: lipschitz elastic-net}, while \Cref{cor: partialD} (2) is a direct consequence of the intermediate result $|D - \bar D| \le  n^{-\frac12}\opnorm{\bZ- \bar\bZ} D [4 \max(1, (2\eta)^{-1})]$ in proof of \Cref{lem: lipschitz-Lasso}. 
\end{proof}


Before proving the derivative formula, we restate $\hbB$ (defined in \eqref{eq: hbB} of the full paper) below, 
\begin{equation}\label{eq: hbB-1}
	\hbB=\argmin_{\bB\in\R^{p\times T}}
	\Big(
	\frac{1}{2n}\fnorm*{\bY - \bX\bB }^2 + \lambda \norm{\bB}_{2,1}
	+ \frac{\tau}{2} \fnorm{\bB}^2
	\Big),
\end{equation}
where $\|\bB\|_{2,1} = \sum_{j=1}^p \|{\bB^{\top} \be_j}\|_2$.

For the reader's convenience, we recall some useful notations. $\bP_{\hat{\mathscr{S}}} = \sum_{k\in\hat{\mathscr{S}}} \be_k\be_k^\top$. For each $k\in \hat{\mathscr{S}}$, $\bH^{(k)}=\lambda\|\hbB{}^\top \be_k\|_2^{-1}\left(\bI_T - \hbB{}^\top\be_k \be_k^\top\hbB ~  \|\hbB{}^\top\be_k\|_2^{-2} \right)$. $\tbH = \sum_{k\in\hat{\mathscr{S}}} (\bH^{(k)} \otimes \be_k\be_k^\top).$ $\bM_1 = \bI_T \otimes (\bX_{\hat{\mathscr{S}}}^\top\bX_{\hat{\mathscr{S}}} + \tau n \bP_{\hat{\mathscr{S}}})$, $\bM = \bM_1 + n\tbH\in \R^{pT\times pT}$,  and $\bN = (\bI_T \otimes \bX)\bM^\dagger (\bI_T \otimes \bX^\top)$.

\begin{proof}[Proof of \Cref{lem: Dijlt}]
	We first derive $\frac{\partial F_{lt}}{\partial x_{ij}}$. 
	Since $\bF = \bY - \bX\hbB = \bE - \bX(\hbB -\bB^*)$, by product rule, 
	\begin{align*}
		\frac{\partial F_{lt}}{\partial x_{ij} } = \be_l^\top \frac{\partial \bE - \bX(\hbB -\bB^*) }{\partial x_{ij} } \be_t
		= - \be_l^\top (\dot\bX (\hbB -\bB^*) + \bX\dot\bB) \be_t,
	\end{align*}
	where 
	$\dot\bX \defas \frac{\partial \bX}{\partial x_{ij}} = \be_i\be_j^\top$,
	and $\dot\bB \defas \frac{\partial \hbB}{\partial x_{ij}}$. 
	
	
	Now we derive $\vec(\dot\bB)$ from KKT conditions for $\hbB$ defined in \eqref{eq: hbB-1}:
	\begin{itemize}
		\item[1)] For $k\in \hat{\mathscr{S}}$, \ie, $\hbB{}^\top \be_k \ne\mathbf{0}$,
		$$\be_k^\top\bX^\top\big[\bE-\bX(\hbB-\bB^*)\big] -n\tau\be_k^\top \hbB= \frac{n \lambda}{\|\hbB{}^\top\be_k\|_2} \be_k^\top \hbB
		\quad \in\R^{1\times T}.$$
		
		\item[2)] For $k\notin \hat{\mathscr{S}}$, \ie, $\hbB{}^\top \be_k = \mathbf 0$,
		$$\norm*{\be_k^\top\bX^\top\big[\bE-\bX(\hbB-\bB^*)\big] -n\tau\be_k^\top \hbB}< n\lambda.$$
		Here the strict inequality is guaranteed by Proposition 2.3  of \cite{bellec2020out}. 
	\end{itemize}
	Keeping $\bE$ fixed, differentiation of the above display for $k\in\hat{\mathscr{S}}$ w.r.t. $x_{ij}$ yields
	$$\be_k^\top\Big[\dot\bX{}^\top\bF - \bX^\top[ \dot\bX(\hbB-\bB^*) +\bX\dot\bB]-n\tau\dot \bB\Big]= n\be_k^\top \dot\bB \bH^{(k)},$$
	with $\bH^{(k)}
	=
	\lambda
	\|\hbB{}^\top \be_k\|_2^{-1}\left(\bI_T - \hbB{}^\top\be_k \be_k^\top\hbB ~  \|\hbB{}^\top\be_k\|_2^{-2} \right)\in\R^{T\times T}$. 
	Rearranging 
	and using $\dot\bX = \be_i\be_j^\top$, 
	\[
	\be_k^\top\Big[\be_j\be_i^\top\bF - \bX^\top \be_i \be_j^\top(\hbB-\bB^*)\Big]= \be_k^\top [(\bX^\top\bX + n\tau\bI_{p}) \dot\bB + n\dot\bB \bH^{(k)}].
	\]
	Recall $\bP_{\hat{\mathscr{S}}} = \sum_{k\in\hat{\mathscr{S}}} \be_k\be_k^\top$. Multiplying by $\be_k$ to the left and summing over $k\in \hat{\mathscr{S}}$, we obtain
	\[
	\bP_{\hat{\mathscr{S}}}\Big[\be_j\be_i^\top\bF - \bX^\top \be_i \be_j^\top(\hbB-\bB^*)\Big]= \bP_{\hat{\mathscr{S}}} (\bX^\top\bX + n\tau\bI_{p}) \dot\bB + n \sum_{k\in\hat{\mathscr{S}}} \be_k\be_k^\top \dot\bB \bH^{(k)}.
	\]
	Since $\hat{\mathscr{S}}$ is locally constant in a small neighborhood of $\bX$, $\hbB_{\hat{\mathscr{S}}{}^c}=0$, $\supp(\dot\bB)\subseteq \hat{\mathscr{S}}$. Hence, $\bP_{\hat{\mathscr{S}}}\dot\bB = \dot\bB$, and $\bX\dot\bB = \bX_{\hat{\mathscr{S}}}\dot\bB$. The above display can be rewritten as 
	\[
	\bP_{\hat{\mathscr{S}}} \be_j\be_i^\top\bF - \bX_{\hat{\mathscr{S}}}^\top \be_i \be_j^\top(\hbB-\bB^*)
	=
	(\bX_{\hat{\mathscr{S}}}^\top\bX_{\hat{\mathscr{S}}} + n\tau\bP_{\hat{\mathscr{S}}}) \dot\bB + n \sum_{k\in\hat{\mathscr{S}}} \be_k\be_k^\top \dot\bB \bH^{(k)}.
	\]
	Vectorizing the above display using property $\vec(\bA\bB\bC) = (\bC^\top \otimes \bA)\vec(\bA)$ yields
	\begin{align*}
		&(\bF^\top \otimes \bP_{\hat{\mathscr{S}}} \be_j) \vec(\be_i^\top) -
		((\hbB-\bB^*)^\top\be_j\otimes \bX_{\hat{\mathscr{S}}}^\top)\vec(\be_i)
		\\=~& 
		[\bI_T \otimes (\bX_{\hat{\mathscr{S}}}^\top\bX_{\hat{\mathscr{S}}}+ n\tau\bP_{\hat{\mathscr{S}}}) +  n\sum_{k\in\hat{\mathscr{S}}} (\bH^{(k)} \otimes \be_k\be_k^\top)] \vec(\dot\bB)\\
		=~& (\bM_1 + n \tbH)\vec(\dot\bB)\\
		=~& \bM \vec(\dot\bB),
	\end{align*}
	where $\bM_1 = \bI_T \otimes (\bX_{\hat{\mathscr{S}}}^\top\bX_{\hat{\mathscr{S}}}+ n\tau\bP_{\hat{\mathscr{S}}})$, and $\tbH = \sum_{k\in\hat{\mathscr{S}}} (\bH^{(k)} \otimes \be_k\be_k^\top)$. 
	
	Under \Cref{assu: tau}(i) that $\tau>0$, it's obviously that $\rank(\bI_T \otimes (\bX_{\hat{\mathscr{S}}}^\top\bX_{\hat{\mathscr{S}}}+ n\tau\bP_{\hat{\mathscr{S}}})) = T |\hat{\mathscr{S}}|$. 
	Under \Cref{assu: tau}(ii) that $\tau=0$ with $\P(U_1)\to 1$. In the event $U_1\cap U_2$, we know $\rank(\bX_{\hat{\mathscr{S}}}) = |\hat{\mathscr{S}}|$ from \cite[Lemma C.4]{bellec2021chi}, hence $\rank(\bI_T \otimes (\bX_{\hat{\mathscr{S}}}^\top\bX_{\hat{\mathscr{S}}}+ n\tau\bP_{\hat{\mathscr{S}}})) = T |\hat{\mathscr{S}}|$. 
	In either of the above two scenarios, we thus have $\dim(\ker(\bI_T \otimes (\bX_{\hat{\mathscr{S}}}^\top\bX_{\hat{\mathscr{S}}}+ n\tau\bP_{\hat{\mathscr{S}}}))) = T (p - |\hat{\mathscr{S}}|)$ by rank-nullity theorem. Since $[\bI_T \otimes (\bX_{\hat{\mathscr{S}}}^\top\bX_{\hat{\mathscr{S}}}+ n\tau\bP_{\hat{\mathscr{S}}})] (\be_t\otimes\be_k) = \bf0$ for $t\in[T], k\in \hat{\mathscr{S}}^c$. 
	Let $V = \{(\be_t\otimes \be_k): t\in[T], k\in\hat{\mathscr{S}}^c\}$ be a vector space, then the elements of $V$ are linear independent, and $\dim(V) = T (p -|\hat{\mathscr{S}}|)$. Thus, $V$ forms a basis for $\ker(\bI_T \otimes (\bX_{\hat{\mathscr{S}}}^\top\bX_{\hat{\mathscr{S}}}+ n\tau\bP_{\hat{\mathscr{S}}})$. Since for any $(\be_t\otimes \be_k) \in V$,  we also have $\tbH (\be_t\otimes \be_k) = \bf0$, $\ker(\bI_T \otimes (\bX_{\hat{\mathscr{S}}}^\top\bX_{\hat{\mathscr{S}}}+ n\tau\bP_{\hat{\mathscr{S}}})) \subseteq \ker(\bI_T \otimes (\bX_{\hat{\mathscr{S}}}^\top\bX_{\hat{\mathscr{S}}}+ n\tau\bP_{\hat{\mathscr{S}}}) + n\tbH)$. On the other hand, if any vector $\bv$ s.t. $[\bI_T \otimes (\bX_{\hat{\mathscr{S}}}^\top\bX_{\hat{\mathscr{S}}}+ n\tau\bP_{\hat{\mathscr{S}}}) + n\tbH]\bv = \bf0$, since these matrices are all positive semi-definite, we have $\bI_T \otimes (\bX_{\hat{\mathscr{S}}}^\top\bX_{\hat{\mathscr{S}}}+ n\tau\bP_{\hat{\mathscr{S}}})\bv = \bf0$, which implies that $\ker(\bI_T \otimes (\bX_{\hat{\mathscr{S}}}^\top\bX_{\hat{\mathscr{S}}}+ n\tau\bP_{\hat{\mathscr{S}}})+ n\tbH) \subseteq \ker(\bI_T \otimes (\bX_{\hat{\mathscr{S}}}^\top\bX_{\hat{\mathscr{S}}}+ n\tau\bP_{\hat{\mathscr{S}}}))$. Therefore, 
	\begin{align*}
		\ker(\bI_T \otimes (\bX_{\hat{\mathscr{S}}}^\top\bX_{\hat{\mathscr{S}}}+ n\tau\bP_{\hat{\mathscr{S}}})+ n\tbH) &= \ker(\bI_T \otimes (\bX_{\hat{\mathscr{S}}}^\top\bX_{\hat{\mathscr{S}}}+ n\tau\bP_{\hat{\mathscr{S}}}))\\
		&= \mathrm{span}\{(\be_t\otimes \be_k): t\in[T], k\in\hat{\mathscr{S}}^c\},
	\end{align*}
	and 
	\begin{align*}
		\range(\bI_T \otimes (\bX_{\hat{\mathscr{S}}}^\top\bX_{\hat{\mathscr{S}}}+ n\tau\bP_{\hat{\mathscr{S}}})+ n\tbH)
		&= \mathrm{span}\{(\be_t\otimes \be_k): t\in[T], k\in\hat{\mathscr{S}}\}.
	\end{align*}
	
	Since $\dot\bB = \bP_{\hat{\mathscr{S}}}\dot\bB$, $\vec(\dot\bB) = (\bI_T \otimes \bP_{\hat{\mathscr{S}}})\vec(\dot\bB)$, then $\vec(\dot\bB) \in \mathrm{col}(\bI_T\otimes \bP_{\hat{\mathscr{S}}}) = \range (\bM)$. Since $\bM$ is symmetric, $\bM^{\dagger} \bM$ is the orthogonal projection on the range of $\bM$.
	Therefore, 
	\begin{align}\label{eq: vecB}
		\vec(\dot\bB) = \bM^{\dagger} \bM \vec(\dot\bB) = \bM^{\dagger} [(\bF^\top\otimes \be_j) - ((\hbB-\bB^*)^\top\be_j\otimes\bX^\top)] \be_i.
	\end{align}
	
	Since $\supp(\dot\bB)\subseteq \hat{\mathscr{S}}$, $\bX\dot\bB = \bX_{\hat{\mathscr{S}}}\dot\bB$, we have 
	\begin{align*}
		\frac{\partial F_{lt}}{\partial x_{ij} } 
		&= - \be_l^\top (\dot\bX(\hbB-\bB^*) + \bX\dot\bB) \be_t\\
		&= - (\be_l^\top \be_i\be_j^\top(\hbB-\bB^*)\be_t +
		\be_l^\top\bX_{\hat{\mathscr{S}}}\dot\bB\be_t)\\
		&= -(\be_l^\top \be_i\be_j^\top(\hbB-\bB^*)\be_t + (\be_t^\top \otimes \be_l^\top\bX_{\hat{\mathscr{S}}})\vec(\dot\bB))\\
		&= - \be_l^\top \be_i\be_j^\top(\hbB-\bB^*)\be_t - (\be_t^\top \otimes \be_l^\top\bX_{\hat{\mathscr{S}}})\bM^{\dagger} [(\bF^\top\otimes \be_j) - ((\hbB-\bB^*)^\top\be_j\otimes\bX^\top)] \be_i\\
		&= - (e_j^\top(\hbB-\bB^*) \otimes \be_i^\top) (\be_t\otimes \be_l) +
		(\be_t^\top \otimes \be_l^\top\bX_{\hat{\mathscr{S}}})\bM^{\dagger}((\hbB-\bB^*)^\top\be_j\otimes\bX^\top\be_i) \\
		&\quad -
		(\be_t^\top \otimes \be_l^\top\bX_{\hat{\mathscr{S}}})\bM^{\dagger} (\bF^\top\otimes \be_j)\be_i\\
		&= - (e_j^\top(\hbB-\bB^*) \otimes \be_i^\top) (\be_t\otimes \be_l) + 
		(e_j^\top(\hbB-\bB^*) \otimes \be_i^\top)\bN (\be_t\otimes \be_l)\\
		&\quad -
		(\be_t^\top \otimes \be_l^\top\bX_{\hat{\mathscr{S}}})\bM^{\dagger} (\bF^\top\otimes \bI_p) (\be_i\otimes \be_j)\\
		&= - (e_j^\top(\hbB-\bB^*) \otimes \be_i^\top) (\bI_{nT} -\bN)(\be_t\otimes \be_l)
		 -
		(\be_t^\top \otimes \be_l^\top\bX)\bM^{\dagger} (\bF^\top\otimes \bI_p) (\be_i\otimes \be_j)
	\end{align*}

Now we calculate $\frac{\partial F_{lt}}{\partial z_{ij}}$. 
Since $\bX = \bZ \bSigma^{\frac 12}$, 
$x_{ik} = \sum_{j=1}^p z_{ij} (\bSigma^{\frac 12})_{jk}$, $ \frac{\partial x_{ik}}{\partial z_{ij}} = (\bSigma^{\frac 12})_{jk}$, 
\begin{align*}
	\frac{\partial F_{lt}}{\partial z_{ij}}  
	= \sum_{k=1}^p \frac{\partial F_{lt}}{\partial x_{ik}}  \frac{\partial x_{ik}}{\partial z_{ij}} 
	= \sum_{k=1}^p \frac{\partial F_{lt}}{\partial x_{ik}} (\bSigma^{\frac 12})_{jk}
	= D_{ij}^{lt} + \Delta_{ij}^{lt},
\end{align*}
where 
\begin{align*}
	D_{ij}^{lt} 
	&= -\sum_{k=1}^p (\be_k^\top(\hbB-\bB^*) \otimes \be_i^\top) (\bI_{nT} - \bN) (\be_t\otimes \be_l) (\bSigma^{\frac 12})_{jk}\\
	&= -(\be_j^\top \bSigma^{\frac 12} (\hbB-\bB^*) \otimes \be_i^\top) (\bI_{nT} - \bN) (\be_t\otimes \be_l)\\
	&= -(\be_j^\top\bH \otimes \be_i^\top) (\bI_{nT} - \bN) (\be_t\otimes \be_l),
\end{align*}
and 
\begin{align*}
	\Delta_{ij}^{lt} 
	&=-\sum_{k=1}^p  (\be_t^\top \otimes \be_l^\top)(\bI_T\otimes \bX )
	\bM^\dagger\bigl(\bF^\top \otimes \bI_{p}\bigr)(\be_i \otimes\be_k) (\bSigma^{\frac 12})_{jk}\\
	&=- (\be_t^\top \otimes \be_l^\top)(\bI_T\otimes \bX)
	\bM^\dagger\bigl(\bF^\top \otimes \bI_{p}\bigr)(\be_i \otimes \bSigma^{\frac 12}\be_j)\\
	&=- (\be_t^\top \otimes \be_l^\top)(\bI_T\otimes \bX)
	\bM^\dagger (\bI_T\otimes \bSigma^{\frac 12}) \bigl(\bF^\top \otimes \bI_{p}\bigr)(\be_i \otimes \be_j)\\
\end{align*}
It follows that 
\begin{align*}
	\sum_{i=1}^n D_{ij}^{it}
	&=-\sum_{i=1}^n  (\be_j^\top\bH \otimes \be_i^\top) (\bI_{nT} - \bN) (\be_t\otimes \be_i)\\
	&=-  \be_j^\top\bH \big[ \sum_{i=1}^n  (\bI_T \otimes \be_i^\top) (\bI_{nT} - \bN) (\bI_T \otimes \be_i)\big] \be_t\\
	&= -\be_j^\top \bH (n\bI_T - \hbA)\be_t, 
\end{align*}
where the last line follows from definition of $\hbA$ in \eqref{eq: hbA-matrix}. 
\end{proof}

\begin{proof}[Proof of \Cref{lem: partialF}]
	\begin{itemize}
		\item[(1)] For $\tau>0$, by formula of $\frac{\partial F_{lt}}{\partial z_{ij} }$ in \Cref{lem: Dijlt}, we have 
		\begin{align*}
			&\sum_{ij}\norm*{\frac{\partial \bF}{\partial z_{ij}}}^2_{\rm F} = \sum_{ij}\sum_{lt} \Big(\frac{\partial F_{lt}}{\partial z_{ij} } \Big)^2=\sum_{ij}\sum_{lt} \Big( D_{ij}^{lt} + \Delta_{ij}^{lt} \Big)^2 \\
			\le~ & 2\sum_{ij,lt} ( D_{ij}^{lt})^2 + 2\sum_{ij,lt} (\Delta_{ij}^{lt})^2 \\
			=~ & 2 \fnorm*{(\bH\otimes \bI_n)(\bI_{nT} - \bN)}^2 + 2 \fnorm{(\bI_T\otimes \bX)\bM^\dagger (\bI_T\otimes \bSigma^{\frac12}) (\bF^\top \otimes \bI_{p})}^2\\
			\le~ & 2n\fnorm{\bH}^2 + 2 \fnorm{(\bI_T\otimes \bX_{\hat{\mathscr{S}}})\bM^\dagger (\bI_T\otimes \bSigma^{\frac12}) (\bF^\top \otimes \bI_{p})}^2. 
		\end{align*}
		Since $0\preceq \bM^\dagger \preceq  \bI_T \otimes (\bX_{\hat{\mathscr{S}}}^\top\bX_{\hat{\mathscr{S}}} + \tau n \bP_{\hat{\mathscr{S}}})^\dagger$,
		\begin{align*}
			&\fnorm{(\bI_T\otimes \bX_{\hat{\mathscr{S}}})\bM^\dagger (\bI_T\otimes \bSigma^{\frac12}) (\bF^\top \otimes \bI_{p})}^2\\
			\le~& \opnorm{(\bI_T\otimes \bX_{\hat{\mathscr{S}}})\bM^\dagger (\bI_T\otimes \bSigma^{\frac12})}^2
			\fnorm{(\bF^\top \otimes \bI_{p})}^2\\
			\le~& p\opnorm{\bSigma} \fnorm{\bF}^2
			\opnorm{(\bI_T\otimes \bX_{\hat{\mathscr{S}}})\bM^\dagger}^2 \\
			\le~ &p\opnorm{\bSigma}\fnorm{\bF}^2\opnorm{(\bX_{\hat{\mathscr{S}}}^\top\bX_{\hat{\mathscr{S}}} + \tau n \bP_{\hat{\mathscr{S}}})^\dagger\bX_{\hat{\mathscr{S}}}^\top}^2\\
			\le~& \frac{p}{n\tau}\opnorm{\bSigma} \fnorm*{\bF}^2\\
			=~& \frac{p}{n\tau'} \fnorm*{\bF}^2,
		\end{align*}
		where the last inequality uses $\opnorm{(\bX_{\hat{\mathscr{S}}}^\top\bX_{\hat{\mathscr{S}}} + \tau n \bP_{\hat{\mathscr{S}}})^\dagger\bX_{\hat{\mathscr{S}}}^\top}\le (n\tau)^{-1}$.
		
		On the other hand, we also have
		\begin{align*}
			&\fnorm{(\bI_T\otimes \bX_{\hat{\mathscr{S}}})\bM^\dagger (\bI_T\otimes \bSigma^{\frac12}) (\bF^\top \otimes \bI_{p})}^2\\
			\le~& \fnorm{(\bI_T\otimes \bX_{\hat{\mathscr{S}}})\bM^\dagger }^2
			\opnorm{(\bI_T\otimes \bSigma^{\frac12})(\bF^\top \otimes \bI_{p})}^2\\
			\le~& \fnorm{(\bI_T\otimes \bX_{\hat{\mathscr{S}}}) (\bI_T \otimes (\bX_{\hat{\mathscr{S}}}^\top\bX_{\hat{\mathscr{S}}} + \tau n \bP_{\hat{\mathscr{S}}})^\dagger )}^2
			\fnorm{\bF}^2 \opnorm{\bSigma}\\
			\le~& T\fnorm{\bX_{\hat{\mathscr{S}}} (\bX_{\hat{\mathscr{S}}}^\top\bX_{\hat{\mathscr{S}}} + \tau n \bP_{\hat{\mathscr{S}}})^\dagger}^2
			\fnorm{\bF}^2 \opnorm{\bSigma}\\
			\le~& T\trace\big[
			(\bX_{\hat{\mathscr{S}}}^\top\bX_{\hat{\mathscr{S}}} + \tau n \bP_{\hat{\mathscr{S}}})^\dagger
			\bX_{\hat{\mathscr{S}}}^\top\bX_{\hat{\mathscr{S}}}
			(\bX_{\hat{\mathscr{S}}}^\top\bX_{\hat{\mathscr{S}}} + \tau n \bP_{\hat{\mathscr{S}}})^\dagger\big]
			\fnorm{\bF}^2 \opnorm{\bSigma}\\
			\le~& T\trace\big[
			(\bX_{\hat{\mathscr{S}}}^\top\bX_{\hat{\mathscr{S}}} + \tau n \bP_{\hat{\mathscr{S}}})^\dagger\big]
			\fnorm{\bF}^2 \opnorm{\bSigma}\\
			\le~& T (\tau)^{-1}\fnorm{\bF}^2 \opnorm{\bSigma}\\
			\le~& T\trace\big[
			(\tau n \bP_{\hat{\mathscr{S}}})^\dagger\big]
			\fnorm{\bF}^2 \opnorm{\bSigma}\\
			\le~& T (\tau)^{-1}\fnorm{\bF}^2 \opnorm{\bSigma}\\
			=~& \frac{T}{\tau'} \fnorm*{\bF}^2,
		\end{align*}
		
		Therefore, 
		\begin{align*}
			\frac{1}{n}\sum_{ij}\norm*{\frac{\partial \bF}{\partial z_{ij}}}^2_{\rm F} 
			&\le 2\fnorm{\bH}^2 + 2
			(\tau')^{-1} (T\wedge \frac{p}{n})\fnorm*{\bF}^2/n \\
			&\le 2\max(1, (\tau')^{-1} (T\wedge \frac{p}{n})) (\fnorm*{\bF}^2/n + \fnorm{\bH}^2)\\
			&= 2\max(1, (\tau')^{-1} (T\wedge \frac{p}{n})) D^2.
		\end{align*}
		Now by product rule and triangle inequality 
		\begin{align*}
			&\frac{1}{n}\sum_{ij}\fnorm*{\frac{\partial \bF/D}{\partial z_{ij}}}^2 \\
			\le~&
			2 D^{-2}\frac{1}{n}\sum_{ij}\fnorm*{\frac{\partial \bF}{\partial z_{ij}}}^2
			+ 
			2 \frac{1}{n}\sum_{ij}\fnorm*{\bF\frac{\partial D^{-1}}{\partial z_{ij}}}^2\\
			=~& 2 D^{-2}\frac{1}{n}\sum_{ij}\fnorm*{\frac{\partial \bF}{\partial z_{ij}}}^2
			+ 
			2 D^{-4}\frac{1}{n}\fnorm{\bF}^2 \sum_{ij}\Big(\frac{\partial D}{\partial z_{ij}}\Big)^2\\
			\le~& 2 D^{-2}\frac{1}{n}\sum_{ij}\fnorm*{\frac{\partial \bF}{\partial z_{ij}}}^2
			+ 
			2 D^{-4}\frac{1}{n}\fnorm{\bF}^2 n^{-1} D^2 [4 \max(1, (2\tau')^{-1})]^2\\
			\le~& 2 D^{-2}\frac{1}{n}\sum_{ij}\fnorm*{\frac{\partial \bF}{\partial z_{ij}}}^2
			+ 
			2  n^{-1}  [4 \max(1, (2\tau')^{-1})]^2\\
			\le~& 4 \max(1, (\tau')^{-1} (T\wedge \frac{p}{n})) + 2  n^{-1}  [4 \max(1, (2\tau')^{-1})]^2\\
			:=~& f(\tau', T, n, p), 
		\end{align*}
		where the second inequality is by \Cref{cor: partialD}. 
		
		\item[(2)] For $\tau=0$, by \Cref{lem: lipschitz-Lasso}, in the event $U_1\cap U_2$, we obtain the same upper bounds as in the first case (1) with $\tau'$ replaced by $\eta$. To see this, 
		\begin{align*}
			&\fnorm{(\bI_T\otimes \bX_{\hat{\mathscr{S}}})\bM^\dagger (\bI_T\otimes \bSigma^{\frac12}) (\bF^\top \otimes \bI_{p})}^2\\
			\le~& \opnorm{(\bI_T\otimes \bX_{\hat{\mathscr{S}}})\bM^\dagger (\bI_T\otimes \bSigma^{\frac12})}^2
			\fnorm{(\bF^\top \otimes \bI_{p})}^2\\
			=~&  \opnorm{(\bI_T\otimes \bSigma^{\frac12})\bM^\dagger(\bI_T\otimes \bX_{\hat{\mathscr{S}}}^\top\bX_{\hat{\mathscr{S}}})\bM^\dagger (\bI_T\otimes \bSigma^{\frac12})}
			p\fnorm{\bF}^2\\
			\le~&  \opnorm{(\bI_T\otimes \bSigma^{\frac12})\bM^\dagger (\bI_T\otimes \bSigma^{\frac12})}
			p\fnorm{\bF}^2\\
			\le~& p\ \fnorm{\bF}^2 \frac{1}{n\eta}\\
			=~& \frac{p}{n\eta} \fnorm*{\bF}^2,
		\end{align*} 
		where the third inequality is by \Cref{lem: aux-tau=0}. 
		Also, we have 
		\begin{align*}
			&\fnorm{(\bI_T\otimes \bX_{\hat{\mathscr{S}}})\bM^\dagger (\bI_T\otimes \bSigma^{\frac12}) (\bF^\top \otimes \bI_{p})}^2\\
			\le~& \fnorm{(\bI_T\otimes \bX_{\hat{\mathscr{S}}})\bM^\dagger (\bI_T\otimes \bSigma^{\frac12})}^2 \opnorm{(\bF^\top \otimes \bI_{p})}^2\\
			\le~& \trace\big[(\bI_T\otimes \bX_{\hat{\mathscr{S}}}^\top\bX_{\hat{\mathscr{S}}})\bM^\dagger (\bI_T\otimes \bSigma_{\hat{\mathscr{S}},\hat{\mathscr{S}}})\bM^\dagger\big]
			\fnorm{\bF}^2 \\
			\le~& \trace\big[ (\bI_T\otimes \bSigma_{\hat{\mathscr{S}},\hat{\mathscr{S}}})\bM^\dagger\big]
			\fnorm{\bF}^2 \\
			\le~& \trace\big[ (\bI_T\otimes \bSigma_{\hat{\mathscr{S}},\hat{\mathscr{S}}}) (\bI_T\otimes (\bX_{\hat{\mathscr{S}}}^\top\bX_{\hat{\mathscr{S}}})^\dagger)\big]
			\fnorm{\bF}^2 \\
			=~& T \trace\big[  \bSigma_{\hat{\mathscr{S}},\hat{\mathscr{S}}} (\bX_{\hat{\mathscr{S}}}^\top\bX_{\hat{\mathscr{S}}})^\dagger\big]
			\fnorm{\bF}^2 \\
			\le~& T \trace\big[ (n\eta)^{-1} \bX_{\hat{\mathscr{S}}}^\top\bX_{\hat{\mathscr{S}}} (\bX_{\hat{\mathscr{S}}}^\top\bX_{\hat{\mathscr{S}}})^\dagger \big]
			\fnorm{\bF}^2 \\
			\le~& \frac{T}{\eta} \fnorm*{\bF}^2,
		\end{align*}
		where the penultimate inequality uses $\bSigma_{\hat{\mathscr{S}},\hat{\mathscr{S}}} \preceq (n\eta)^{-1} \bX_{\hat{\mathscr{S}}}^\top\bX_{\hat{\mathscr{S}}}$ in the event $U_1\cap U_2$. 
		Therefore, on $U_1\cap U_2$, we have
		\begin{align*}
			\frac{1}{n}\sum_{ij}\fnorm*{\frac{\partial \bF/D}{\partial z_{ij}}}^2 
			&\le 4 \max(1, (\eta)^{-1} (T\wedge \frac{p}{n})) + 2  n^{-1}  [4 \max(1, (2\tau')^{-1})]^2 \\
			&:= f(\eta, T, n, p),
		\end{align*}
		where the function $f$ is the same as in case (1). The only difference is that $\tau'$ in the upper bound for case (1) is replaced by $\eta$ in case (2). 
	\end{itemize}
	
\end{proof}

\subsection{Proofs of results in \saferef{Appendix}{sec:Lipschitz-fix-X}}

The following proof of \Cref{lem:JacobianE} relies on a similar argument as proof of \Cref{lem: Dijlt}, we present the proof here for completeness.
\begin{proof}[Proof of \Cref{lem:JacobianE}]
	Recall the KKT condtions for $\hbB$ defined in \eqref{eq: hbB}:
	\begin{itemize}
		\item[1)] For $k\in \hat{\mathscr{S}}$, \ie, $\hbB{}^\top \be_k \ne\mathbf{0}$,
		$$
		\be_k^\top\bX^\top\big[\bE-\bX(\hbB-\bB^*)\big] -n\tau\be_k^\top \hbB= \frac{n \lambda}{\|\hbB{}^\top\be_k\|_2} \be_k^\top \hbB
		\quad \in\R^{1\times T}.
		$$
		\item[2)] For $k\notin \hat{\mathscr{S}}$, \ie, $\hbB{}^\top \be_k = \mathbf 0$,
		$$
		\norm*{\be_k^\top\bX^\top\big[\bE-\bX(\hbB-\bB^*)\big] -n\tau\be_k^\top \hbB}< n\lambda.
		$$
		Here the strict inequality is guaranteed by Proposition 2.3  of \cite{bellec2020out}. 
	\end{itemize}
	Let $\ddot{\bB} = \frac{\partial \hbB}{\partial E_{it'}}$, $\dot\bE = \frac{\partial \bE}{\partial E_{it'}}$. Differentiation of the above display for $k\in\hat{\mathscr{S}}$ w.r.t. $E_{it'}$ yields
	$$\be_k^\top\bX^\top(\dot\bE-\bX \ddot\bB) - n\tau \be_k^\top \ddot\bB
	= n\be_k^\top \ddot\bB \bH^{(k)}$$
	with $\bH^{(k)}
	=
	\lambda
	\|\hbB{}^\top \be_k\|_2^{-1}\left(\bI_T - \hbB{}^\top\be_k \be_k^\top\hbB ~  \|\hbB{}^\top\be_k\|_2^{-2} \right)\in\R^{T\times T}$. Rearranging
	and using $\dot\bE = \be_i \be_{t'}^\top$, 
	$$\be_k^\top \bX^\top \be_i \be_{t'}^\top = \be_k^\top[n\ddot \bB \bH^{(k)} + (\bX^\top\bX+n\tau\bI_{p\times p})\ddot\bB].$$
	Recall $\bP_{\hat{\mathscr{S}}} = \sum_{k\in\hat{\mathscr{S}}} \be_k\be_k^\top\in\R^{p\times p}$. Multiplying by $\be_k$ to the left and summing over $k\in \hat{\mathscr{S}}$, we obtain
	$$\bP_{\hat{\mathscr{S}}} \bX^\top \be_i \be_{t'}^\top =  n \sum_{k\in\hat{\mathscr{S}}} \be_k\be_k^\top \ddot \bB \bH^{(k)} + \bP_{\hat{\mathscr{S}}} (\bX^\top\bX+n\tau\bI_{p\times p})\ddot\bB,$$
	which reduces to the following by $\supp(\ddot \bB)\subseteq \hat{\mathscr{S}}$ and $\bX\ddot\bB = \bX_{\hat{\mathscr{S}}}\ddot\bB$, 
	\begin{align*}
		\bX_{\hat{\mathscr{S}}}^\top \be_i \be_{t'}^\top &= n \sum_{k\in\hat{\mathscr{S}}} \be_k\be_k^\top\ddot \bB \bH^{(k)} + \bX_{\hat{\mathscr{S}}}^\top\bX_{\hat{\mathscr{S}}}\ddot\bB \bI_T + n\tau\bP_{\hat{\mathscr{S}}} \ddot\bB \bI_T\\
		&= n \sum_{k\in\hat{\mathscr{S}}} \be_k\be_k^\top\ddot \bB \bH^{(k)} + (\bX_{\hat{\mathscr{S}}}^\top\bX_{\hat{\mathscr{S}}} + n\tau\bP_{\hat{\mathscr{S}}})
		\ddot\bB \bI_T.
	\end{align*}
	Vectorizing the above yields
	\begin{align*}
		(\be_{t'} \otimes \bX_{\hat{\mathscr{S}}}^\top) \vec(\be_i) &= [n\sum_{k\in\hat{\mathscr{S}}} (\bH^{(k)}\otimes \be_k\be_k^\top) +\bI_T \otimes (\bX_{\hat{\mathscr{S}}}^\top\bX_{\hat{\mathscr{S}}} + n\tau\bP_{\hat{\mathscr{S}}})] \vec(\ddot\bB) \\
		&= (n \tbH +\bI_T \otimes (\bX_{\hat{\mathscr{S}}}^\top\bX_{\hat{\mathscr{S}}} + n\tau\bP_{\hat{\mathscr{S}}}) )\vec(\ddot\bB)\\
		&= \bM\vec(\ddot\bB).
	\end{align*}
	A similar argument as in Proof of \Cref{lem: Dijlt} leads to 
	\begin{align*}
		\vec(\ddot \bB) = \bM^\dagger\bM \vec(\ddot \bB)
		=\bM^{\dagger} (\be_{t'} \otimes \bX_{\hat{\mathscr{S}}}^\top) \be_i.
	\end{align*}
	
	Therefore, by $\bX\ddot{\bB}=\bX_{\hat{\mathscr{S}}}\ddot{\bB}$,
	\begin{align*}
		\frac{\partial F_{lt}}{\partial E_{it'} } 
		&= \be_l^\top \frac{\partial \bE - \bX(\hbB - \bB^*)}{\partial E_{it'}}\be_t\\
		&= \be_l^\top \big( \be_i\be_{t'}^\top - \bX \ddot{\bB}\big) \be_t\\
		&= \be_l^\top \be_i\be_{t'}^\top\be_t- \be_l^\top\bX \ddot{\bB}\be_t\\
		&= \be_l^\top \be_i\be_{t'}^\top\be_t- (\be_t^\top \otimes \be_l^\top\bX_{\hat{\mathscr{S}}})\vec(\ddot\bB)\\
		&= \be_l^\top \be_i\be_{t'}^\top\be_t- (\be_t^\top \otimes \be_l^\top\bX_{\hat{\mathscr{S}}}) \bM^{\dagger} (\be_{t'} \otimes \bX_{\hat{\mathscr{S}}}^\top) \be_i\\
		&= \be_l^\top \be_i\be_{t'}^\top\be_t- \be_l^\top(\be_t^\top \otimes \bX_{\hat{\mathscr{S}}}) \bM^{\dagger} (\be_{t'} \otimes \bX_{\hat{\mathscr{S}}}^\top) \be_i\\
		&= \be_l^\top \be_i\be_{t'}^\top\be_t- \be_l^\top(\be_t^\top \otimes \bX) \bM^{\dagger} (\be_{t'} \otimes \bX^\top) \be_i,
	\end{align*}
	where the last equality is due to $\bM^\dagger = (\bI_T \otimes \bP_{\hat{\mathscr{S}}}) \bM^\dagger  (\bI_T \otimes \bP_{\hat{\mathscr{S}}})$. 
	
	Now the calculation of $\sum_{i=1}^n\frac{\partial F_{it}}{\partial E_{it'}}$ is straightforward, 
	\begin{align*}
		\sum_{i=1}^n\frac{\partial F_{it}}{\partial E_{it'}} &= \sum_{i=1}^n \big[ \be_i^\top\be_i\be_t^\top\be_{t'}
		- \be_i^\top (\be_t^\top \otimes \bX)\bM^\dagger (\be_{t'} \otimes \bX^\top) \be_i\big]\\
		&= n\be_t^\top\be_{t'} - \trace[(\be_t^\top \otimes \bX)\bM^\dagger (\be_{t'} \otimes \bX^\top)]\\
		&= n\be_t^\top\be_{t'} - \be_t^\top \hbA \be_{t'}\\
		&= \be_t^\top (n\bI_T -\hbA)\be_{t'},
	\end{align*}
	where the third equality is due to the formula of $\hbA$ in \eqref{eq: hbA-matrix}.
	
	Noting that $\bF = \bE - \bZ\bH$, it follows that
	$\sum_{i=1}^n\frac{\partial \be_i^\top\bZ\bH \be_t}{\partial E_{it'}} = \be_t^\top \hbA\be_{t'}$. 
\end{proof}

\subsection{Proofs of results in \saferef{Appendix}{sec:proba-tools}}
\begin{proof}[Proof of \Cref{lem: stein-EF}]
	Let $\bz = \vec(\bE)$, then $\bz\sim \calN(\mathbf{0}, \bK)$ with $\bK = \bS \otimes \bI_n$ by Assumption~\ref{assu: noise}. 
	For each $t_0, t_0' \in [T]$, let 
	$\bG^{(t_0, t_0')} = \bF  \be_{t_0'}\be_{t_0}^\top$,  and $\bff(\bz)^{(t_0, t_0')}  = \vec(\bG) \tD^{-1} $. 
	For convenience, we will drop the superscript ${(t_0, t_0')}$ from $\bG^{(t_0, t_0')}$  and $\bff(\bz)^{(t_0, t_0')}$	in this proof. 
	By $\trace(\bA^\top\bB) = \vec(\bA)^\top \vec(\bB)$, we obtain 
	\begin{align}\label{eq: a1}
		\be_{t_0}^\top \bE^\top \bF \tD^{-1} \be_{t_0'} = \trace(\bE^\top \bF  \be_{t_0'}\be_{t_0}^\top) \tD^{-1}= \trace(\bE^\top \bG\tD^{-1}) = \bz^\top\bff(\bz).
	\end{align}
	By product rule, we have 
	\begin{align}\label{eq: nabla-f}
		\nabla \bff(\bz) = \frac{\partial \vec(\bG) }{\partial \vec(\bE) } \tD^{-1} + \underbrace{\vec(\bG) \frac{\partial \tD^{-1}}{\partial \vec(\bE) }}_{\Rem}, 
	\end{align}
	where $\Rem = \bu\bv^\top$ with $\bu = \vec(\bG)\in \R^{nT\times 1}$, $\bv^\top = \frac{\partial \tD^{-1}}{\partial \vec(\bE) }\in \R^{1\times nT}$. 
	It follows that 
	\begin{align}\label{eq: stein-tr1}
		\trace(\bK \nabla \bff(\bz)) = \trace\Big(\bK \frac{\partial \vec(\bG) }{\partial \vec(\bE)}\Big) \tD^{-1} + \trace(\bK \Rem). 
	\end{align}
	Since $\bK = \bS \otimes \bI_n$ and $\bG = \bF  \be_{t_0'}\be_{t_0}^\top$, 
	$\bK_{it, lt'}= S_{tt'}I(i=l)$, and $G_{it} = F_{it_0'} I(t=t_0)$. It follows 
	\begin{equation}\label{eq: stein-tr2}
		\begin{aligned}
			\trace\Big(\bK \frac{\partial \vec(\bG) }{\partial \vec(\bE)}\Big) 
			= \sum_{i,t}\sum_{l,t'} \bK_{it, lt'} \frac{\partial G_{it}}{\partial E_{lt'} }
			= \sum_{t'} S_{t_0t'} \sum_{i}\frac{\partial F_{it_0'}}{\partial E_{it'}}
			= \be_{t_0}^\top \bS (n\bI_T - \hbA) \be_{t_0'},
		\end{aligned}
	\end{equation}
	where the last equality used \Cref{lem:JacobianE} and that $\hbA$ is symmetric. 
	
	Now we rewrite the quantity we want to bound as 
	\begin{align} 
		&\E\Bigl[
		\fnorm{\bE^\top \bF/\tD - \bS (n\bI_T - \hbA )/\tD}^2\Bigr] \nonumber\\
		=~& \sum_{t_0,t_0'} \E \Big[\Big( \be_{t_0}^\top \bE^\top \bF \tD^{-1} \be_{t_0'} - \be_{t_0}^\top \bS (n\bI_T - \hbA) \be_{t_0'}\tD^{-1} \Big)^2\Big]\nonumber\\
		=~& \sum_{t_0,t_0'} \E \Big[ \big(\bz^\top \bff(\bz) - \trace(\bK \nabla \bff(\bz)) + \trace(\bK\Rem)\big)^2\Big] \nonumber\\
		\le~& 2\sum_{t_0,t_0'} \Big\{ \E \Big[ \big(\bz^\top \bff(\bz) - \trace(\bK \nabla \bff(\bz)) \big)^2\Big] +  \E\Big[ \big(\trace(\bK\Rem)\big)^2 \Big] \Big\}\label{eq: LHS} ,
	\end{align}
	where the second equality follows from \eqref{eq: a1}, \eqref{eq: stein-tr1} and \eqref{eq: stein-tr2}, and the last inequality uses elementary inequality $(a+b)^2 \le 2(a^2 + b^2)$. 
	We next bound the two terms in \eqref{eq: LHS}. 
	\paragraph{First term in \eqref{eq: LHS}.}
	By second-order Stein formula in \Cref{lem: 2nd-Stein}, 
	\begin{equation}\label{eq: stein}
		\sum_{t_0,t_0'} \E \big(\bz^\top \bff(\bz) - \trace(\bK \nabla \bff(\bz))\big)^2 
		= \sum_{t_0,t_0'} 
		\E \Big[\fnorm{\bK^{\frac12}\bff(\bz)}^2  +  \trace\big[\big(\bK \nabla \bff(\bz)\big)^2\big] \Big]. 
	\end{equation}
	Now we bound the two terms in the right-hand side of \eqref{eq: stein}. 
	For the first term, recall  $\bff(\bz) = \vec(\bG) \tD^{-1} $, and $\bG = \bF  \be_{t_0'}\be_{t_0}^\top$, we obtain
	\begin{align*}
		\fnorm{\bK^{\frac12}\bff(\bz)}^2 =  \tD^{-2} \fnorm{(\bS^{\frac12}\otimes \bI_n)\vec(\bG)}^2  =  \tD^{-2} \fnorm{\bG\bS^{\frac12}}^2  = \tD^{-2} \fnorm{\bS^{\frac12}\be_{t_0}}^2  \fnorm{\bF\be_{t_0'}}^2. 
	\end{align*}
	Summing over all $(t_0, t_0')\in [T] \times [T]$, we obtain 
	\begin{equation}\label{eq: stein-RHS1}
		\sum_{t_0, t_0'} \fnorm{\bK^{\frac12}\bff(\bz)}^2 =  \tD^{-2} \fnorm{\bF}^2 \trace(\bS). 
	\end{equation}
	For the second term in RHS of \eqref{eq: stein}, recall $ \nabla\bff(\bz) = \frac{\partial \vec(\bG) }{\partial \vec(\bE) } \tD^{-1} + \Rem$,
	\begin{align}
		&\trace\big[\big(\bK \nabla \bff(\bz)\big)^2\big] \nonumber\\
		=~& \tD^{-2}\trace \Big[\Big( \bK\frac{\partial \vec(\bG) }{\partial \vec(\bE)} \Big)^2\Big]  + \trace[(\bK\Rem)^2] + 2\tD^{-1}\trace\Big[ \bK\frac{\partial \vec(\bG) }{\partial \vec(\bE) } \bK \Rem\Big].\label{eq: *2s}
	\end{align}
	By property of vectorization operation, 
	$\vec(\bG) = \vec(\bF\be_{t_0'}\be_{t_0}^\top) = (\be_{t_0}\be_{t_0'}^\top \otimes \bI_n) \vec(\bF)$, 
	hence 
	$$\frac{\partial \vec(\bG) }{\partial \vec(\bE)} 
	= (\be_{t_0}\be_{t_0'}^\top \otimes \bI_n)  \frac{\partial \vec(\bF) }{\partial \vec(\bE)}, $$
	where $\opnorm{ \frac{\partial \vec(\bF) }{\partial \vec(\bE)} } \le 1$ since the map $\vec(\bE) \mapsto \vec(\bF)$ is 1-Lipschitz by \cite[proposition 3]{bellec2016bounds}. 
	
	Now  we bound the three terms in \eqref{eq: *2s}. For the first term, by Cauchy-Schwarz inequality, 
	\begin{align*}
		&\tD^{-2}\trace \Big[\Big( \bK\frac{\partial \vec(\bG) }{\partial \vec(\bE)} \Big)^2\Big]\\
		=~& \tD^{-2}\trace \Big( \bK (\be_{t_0}\be_{t_0'}^\top \otimes \bI_n) \frac{\partial \vec(\bF) }{\partial \vec(\bE)} \bK (\be_{t_0}\be_{t_0'}^\top \otimes \bI_n)\frac{\partial \vec(\bF) }{\partial \vec(\bE)}\Big)\\
		\le~& \tD^{-2}\fnorm{(\be_{t_0'}^\top \otimes\bI_n)\frac{\partial \vec(\bF) }{\partial \vec(\bE)} \bK (\be_{t_0}\otimes \bI_n)}^2. 
	\end{align*}
	For the second term in \eqref{eq: *2s}, recall $\Rem = \bu\bv^\top$, and $\bu = \vec(\bG)$, $\bv^\top = \frac{\partial \tD^{-1}}{\partial \vec(\bE) }$ from \eqref{eq: nabla-f}, then $\trace[(\bK\Rem)^2] = \trace( \bK \bu\bv^\top\bK \bu\bv^\top)= (\bv^\top\bK \bu)^2$, thus, 
	\begin{align*}
		&\trace[(\bK\Rem)^2] \\
		=~& \Big[ \frac{\partial \tD^{-1}}{\partial \vec(\bE) } \bK \vec(\bG)\Big]^2\\
		=~& \tD^{-6}  \Big[ \vec(\bF)^\top\frac{\partial \vec(\bF)}{\partial \vec(\bE)} \bK (\be_{t_0}\be_{t_0'}^\top \otimes \bI_n) \vec(\bF) \Big]^2\\
		\le~ & \tD^{-6} \norm{\vec(\bF)^\top\frac{\partial \vec(\bF)}{\partial \vec(\bE)} \bK (\be_{t_0} \otimes \bI_n)}^2 \norm{(\be_{t_0'}^\top \otimes \bI_n) \vec(\bF)}^2.
	\end{align*}
	where the first inequality uses $
	\frac{\partial \tD^{-1}}{\partial \vec(\bE)} = -\frac12\tD^{-3}\frac{\partial \fnorm{\bF}^2}{\partial \vec(\bE)} = -\tD^{-3}\vec(\bF)^\top\frac{\partial \vec(\bF)}{\partial \vec(\bE)}$ by chain rule, and the inequality uses Cauchy-Schwarz inequality. 
	
	For the third term in \eqref{eq: *2s}, recall $\Rem = \bu\bv^\top$, and $\bu = \vec(\bG)$, $\bv^\top = \frac{\partial \tD^{-1}}{\partial \vec(\bE) }$, then 
	$\trace\Big[ \bK\frac{\partial \vec(\bG) }{\partial \vec(\bE) } \bK \Rem\Big] = 
	\bv^\top\bK\frac{\partial \vec(\bG) }{\partial \vec(\bE) } \bK  \bu$, hence
	\begin{align*}
		&2\tD^{-1}\trace\Big[ \bK\frac{\partial \vec(\bG) }{\partial \vec(\bE) } \bK \Rem\Big] \\
		=~& 2\tD^{-1}  \frac{\partial \tD^{-1}}{\partial \vec(\bE) }\bK \frac{\partial \vec(\bG) }{\partial \vec(\bE) }\bK \vec(\bG)\\
		=~& -2\tD^{-4} \vec(\bF)^\top \frac{\partial \vec(\bF) }{\partial \vec(\bE)}\bK 
		(\be_{t_0}\be_{t_0'}^\top \otimes \bI_n)\frac{\partial \vec(\bF) }{\partial \vec(\bE)} \bK(\be_{t_0}\be_{t_0'}^\top \otimes \bI_n) \vec(\bF)\\
		\le~&2 \tD^{-4}  \fnorm{(\be_{t_0'}^\top \otimes \bI_n) \vec(\bF) \vec(\bF)^\top \frac{\partial \vec(\bF) }{\partial \vec(\bE)}\bK (\be_{t_0} \otimes \bI_n)} \\
		& \fnorm{(\be_{t_0'}^\top \otimes \bI_n) \frac{\partial \vec(\bF) }{\partial \vec(\bE)} \bK(\be_{t_0}\otimes \bI_n)}, 
	\end{align*}
	where the last inequality uses Cauchy-Schwarz inequality.  
	
	Summing over all $(t_0, t_0')\in [T]\times [T]$ for these three terms in \eqref{eq: *2s}, using $\opnorm{ \frac{\partial \vec(\bF) }{\partial \vec(\bE)} } \le 1$, $\bK = \bS \otimes \bI_n$, and $\fnorm{\bS} \le \fnorm{\bS^{\frac12}}^2= \trace(\bS)$, we obtain 
	\begin{align}
		&\sum_{t_0, t_0'} \trace\big[\big(\bK \nabla \bff(\bz)\big)^2\big] \nonumber\\
		\le~& \tD^{-2}\fnorm{\bK}^2 + \tD^{-6} \fnorm{\bF}^2 \fnorm{\bK}^2 \fnorm{\bF}^2 + 2 \tD^{-4} \fnorm{\bF}^2 \fnorm{\bK}^2\nonumber\\
		=~ & \tD^{-2}n\fnorm{\bS}^2 + \tD^{-6} \fnorm{\bF}^4 n\fnorm{\bS}^2  + 2 \tD^{-4} \fnorm{\bF}^2 n\fnorm{\bS}^2\nonumber\\
		\le~& \big[\tD^{-2}n\trace(\bS) + \tD^{-6} \fnorm{\bF}^4 n\trace(\bS)  + 2 \tD^{-4} \fnorm{\bF}^2 n\trace(\bS)\big] \trace(\bS).\label{eq: stein-RHS2}  
	\end{align}
	\paragraph{Second term in \eqref{eq: LHS}.}
	Recall that $\Rem = \bu\bv^\top$, hence $[\trace(\bK\Rem)]^2 =\trace[(\bK\Rem)^2]$. 
	By calculation of second term in  \eqref{eq: *2s}, we obtain 
	\begin{align}\label{eq: LHS2}
		\sum_{t_0,t_0'} [\trace(\bK\Rem)]^2 
		=\sum_{t_0, t_0'}  \trace[(\bK\Rem)^2] 
		\le  \tD^{-6} \fnorm{\bF}^4 n  \trace(\bS)^2.
	\end{align}
	
	Combining the results \eqref{eq: LHS}, \eqref{eq: stein}, \eqref{eq: stein-RHS1}, \eqref{eq: stein-RHS2},  \eqref{eq: LHS2}, we obtain 
	\begin{align*}
		&\E\Bigl[
		\fnorm{\bE^\top \bF/\tD - \bS (n\bI_T - \hbA )/\tD}^2\Bigr]\\
		\le~&  2\big[\tD^{-2} \fnorm{\bF}^2+\tD^{-2}n\trace(\bS) + 2\tD^{-6} \fnorm{\bF}^4 n\trace(\bS)  + 2 \tD^{-4} \fnorm{\bF}^2 n\trace(\bS)\big] \trace(\bS)\\
		\le~& 4 \trace(\bS), 
	\end{align*}
	thanks to $\tD^{2} = \fnorm{\bF}^2 + n\trace(\bS)$.
\end{proof}

\begin{proof}[Proof of \Cref{lem:steinX}]
	Apply \cite[Proposition 6.3]{bellec2020out} with $\brho = \bU\be_t, \bfeta = \bV\be_{t'}$, we obtain 
	\begin{align*}
		&\E\Big[\norm[\Big]{\bU^\top \bZ \bV - 
			\sum_{j=1}^p\sum_{i=1}^n
			\frac{\partial}{\partial z_{ij} }\Bigl(\bU^\top \be_i \be_j^\top \bV \Bigr)
		}_{\rm F}^2\Big]\\
		=~ &\sum_{t,t'=1}^T\E \big(\be_t^\top\bU^\top \bZ\bV\be_{t'} - \sum_{j=1}^p\sum_{i=1}^n
		\frac{\partial }{\partial z_{ij}} \be_{t}\bU^\top \be_i \be_j^\top \bV\be_{t'} \big)^2\\
		\le~ & \sum_{t,t'=1}^T \left[\E \big[\norm*{\bU\be_t}^2 \norm*{\bV\be_{t'}}^2\big] + \E
		\sum_{ij} \Big[2\norm*{\bV\be_{t'}}^2 \fnorm*{\frac{\partial \bU \be_t}{\partial z_{ij}}} + 2\norm*{\bU\be_{t}}^2 \fnorm*{\frac{\partial \bV \be_{t'}}{\partial z_{ij}}}\Big] \right]\\
		=~ & \E \fnorm*{\bU}^2 \fnorm*{\bV}^2+ \E \sum_{ij}\Big[
		2\fnorm*{\bV}^2\fnorm*{ \frac{\partial \bU}{\partial z_{ij}} }^2
		+ 2\fnorm*{\bU}^2\fnorm*{ \frac{\partial \bV}{\partial z_{ij}} }^2\Big].
	\end{align*}
\end{proof}

	\begin{proof}[Proof of \Cref{cor:steinX}]
	By Kirszbraun's theorem, there exists an $L_1$-Lipschitz function $\bar\bU: \R^{n\times p} \to \R^{n\times T}$ such that $\bar\bU = \bU$ on $\Omega$, and an $L_2$-Lipschitz function $\bar\bV: \R^{n\times p} \to \R^{n\times T}$ such that $\bar\bV = \bV$ on $\Omega$. By projecting $\bar\bU$ and $\bar\bV$ onto the Euclidean ball of radius 1 and $K$ if necessary, we assume without loss of generality that $\fnorm{\bar\bU}\le 1$ and $\fnorm{\bar\bV}\le K$. 
	Therefore, 
	\begin{align*}
		&\E\Big[I(\Omega) \fnorm{\bU^\top \bZ \bV - 
			\sum_{j=1}^p\sum_{i=1}^n
			\frac{\partial}{\partial z_{ij} }\Bigl(\bU^\top \be_i \be_j^\top \bV \Bigr)
		}^2\Big]\\
		=~&\E\Big[I(\Omega) \fnorm{\bar\bU^\top \bZ\bar \bV - 
			\sum_{j=1}^p\sum_{i=1}^n
			\frac{\partial}{\partial z_{ij} }\Bigl(\bar\bU^\top \be_i \be_j^\top \bar\bV \Bigr)
		}^2\Big]\\
		\le~&\E\Big[ \fnorm{\bar\bU^\top \bZ\bar \bV - 
			\sum_{j=1}^p\sum_{i=1}^n
			\frac{\partial}{\partial z_{ij} }\Bigl(\bar\bU^\top \be_i \be_j^\top \bar\bV \Bigr)
		}^2\Big]\\
		\le~& \E \Big[\big(\fnorm{\bar\bU}^2 \fnorm{\bar\bV}^2\big) + 
		2\sum_{ij}\Big(
		\fnorm{\bar\bV}^2\fnorm*{ \frac{\partial \bar\bU}{\partial z_{ij}} }^2
		+ \fnorm{\bar\bU}^2\fnorm*{ \frac{\partial \bar\bV}{\partial z_{ij}}
		}^2\Big)\Big]\\
		\le~& K^2 + 2\E \Big[
		\sum_{ij}\Big( K^2\fnorm*{\frac{\partial \bar\bU}{\partial z_{ij}} }^2 + \fnorm*{\frac{\partial \bar\bV}{\partial z_{ij}} }^2\Big) \Big]\\
		=~& K^2 + 2\E \Big[I(\Omega)
		\sum_{ij}\Big( K^2\fnorm*{\frac{\partial \bar\bU}{\partial z_{ij}} }^2 + \fnorm*{\frac{\partial \bar\bV}{\partial z_{ij}} }^2\Big) + I(\Omega^c)
		\sum_{ij}\Big( K^2\fnorm*{\frac{\partial \bar\bU}{\partial z_{ij}} }^2 + \fnorm*{\frac{\partial \bar\bV}{\partial z_{ij}} }^2\Big) \Big] \\
		=~& K^2 + 2\E \Big[I(\Omega)
		\sum_{ij}\Big( K^2\fnorm*{\frac{\partial \bU}{\partial z_{ij}} }^2 + \fnorm*{\frac{\partial \bV}{\partial z_{ij}} }^2\Big) + I(\Omega^c)
		\sum_{ij}\Big( K^2\fnorm*{\frac{\partial \bar\bU}{\partial z_{ij}} }^2 + \fnorm*{\frac{\partial \bar\bV}{\partial z_{ij}} }^2\Big) \Big] \\
		\le~& K^2 + 2\E \Big[I(\Omega)
		\sum_{ij}\Big( K^2\fnorm*{\frac{\partial \bU}{\partial z_{ij}} }^2 + \fnorm*{\frac{\partial \bV}{\partial z_{ij}} }^2\Big) \Big]
		+ 2C( K^2 L_1^2  + L_2^2 ),
	\end{align*}
	where the last inequality uses $\sum_{ij}\fnorm{\frac{\partial \bar\bU}{\partial z_{ij}} }^2\le nT (n^{-1/2}L_1)^2 = TL_1^2$, $\sum_{ij}\fnorm{\frac{\partial \bar\bV}{\partial z_{ij}} }^2\le TL_2^2$ by Lipschitz properties of $\bar\bU$, $\bar\bV$, and $P(\Omega^c)\le C/T$. 
\end{proof}

\begin{proof}[Proof of \Cref{lem: Chi2type}]
	For each $j\in[p]$, let $\E_j (\cdot) $ denote the conditional expectation $\E[\cdot |\{\bZ e_k, k\ne j\}]$. The left-hand side of the desired inequality can be rewritten as 
	\begin{align*}
		&\E\Big[
		\fnorm[\big]{ 
		p \bU^\top \bV  - \sum_{j=1}^p 
		(\E_j\bU^\top \bZ - \bL^\top) \be_j
		\be_j^\top (\bZ^\top \E_j\bV - \hat \bL)
	}
		\Big] 
	\end{align*}
	with $\bL\in\R^{p\times T}$ defined by
	$\bL^\top \be_j = \E_j\bU^\top \bZ \be_j - \bU^\top \bZ \be_j + \sum_{i=1}^n \partial_{ij} \bU^\top \be_i$
	and $\hat \bL$ defined similarly with $\bU$ replaced by $\bV $.
	We develop the terms in the sum over $j$ as follows:
	\begin{align}
		\nonumber
		&p \bU^\top \bV  - \sum_j
		(\E_j\bU^\top \bZ - \bL^\top) \be_j
		\be_j^\top (\bZ^\top \E_j\bV - \hat \bL)
		\\=\quad&
		\sum_j
		\Bigl(\bU^\top \bV  -  \E_j[\bU^\top]\E_j\bV\Bigr)
		\label{eq:first-term-swap-trick}
		\\&  + \sum_j
		\Bigl(
		\E_j[\bU^\top]\E_j\bV
		-
		\E_j\bU^\top \bZ \be_j \be_j^\top \bZ^\top\E_j\bV
		\Bigr)
		\label{eq:difficult-term}
		\\ & - \bL^\top \hat \bL \label{eq:Xi-hat-Xi-term}
		\\ & + \sum_j \Bigl(\E_j\bU^\top \bZ \be_j \be_j^\top \hat \bL\Bigr) 
		+ 
		\Bigl(\bL^\top \be_j \be_j^\top \bZ^\top \E_j[\bU]\Bigr).
		\label{eq:last-two-terms}
	\end{align}
	First, for \eqref{eq:Xi-hat-Xi-term},
	by the Cauchy-Schwarz inequality
	$\E\bigl[ \|\bL^\top \hat \bL\|_{\rm F}\bigr]
	\le \E\bigl[\|\bL\|_{\rm F}^2\bigr]^{\frac 12}\E\bigl[\|\hat\bL\|_{\rm F}^2\bigr]^{\frac 12}$.
	For a fixed $j\in[p]$ and $t\in [T]$, 
	\begin{align*}
		\E[(\be_j^\top \bL e_t)^2]
		&\le
		\sum_{i=1}^n
		\E[(\be_i^\top(\E_j[\bU] - \bU)e_t )^2]
		+
		\E
		\sum_{i=1}^n
		\sum_{l=1}^n
		\Bigl(
		\frac{\be_i^\top \partial \bU e_t }{\partial z_{lj}}
		\Bigr)^2\\
		&\le
		2
		\E
		\sum_{i=1}^n
		\sum_{l=1}^n
		\Bigl(
		\frac{\be_i^\top \partial \bU e_t }{\partial z_{lj}}
		\Bigr)^2,
	\end{align*}
	where the two inequalities are due to the second-order stein inequality in \Cref{lem: 2nd-Stein}, and Gaussian-Poincar\'e inequality in \Cref{lem: Gaussian-Poincare}, respectively. 
	Summing over $j\in[p]$ and $t\in[T]$ we obtain
	$\E[\|\bL\|_{\rm F}^2] \le 2 \E \sum_{lj} \| \partial_{lj} \bU \|_{\rm F}^2
	= 2 \|\bU\|_\partial^2$.
	Combined with the same bound for $\hat \bL$, we obtain
	$\E[\|\eqref{eq:Xi-hat-Xi-term}\|_{\rm F}^2] 
	\le
	2 \|\bU\|_\partial\|\bV \|_\partial$.
	We now turn to the two terms in \eqref{eq:last-two-terms}.
	By the triangle inequality for the Frobenius norm,
	\begin{align*}
		\E\Bigl[\|\sum_j \E_j\bU^\top \bZ \be_j \be_j^\top \hat \bL\|_{\rm F}\Bigr]
		&\le
		\sum_j \E\Bigl[\|\E_j\bU^\top \bZ \be_j\|_2 \|\be_j^\top \hat \bL\|_2\Bigr]
		\\ &\le \E[
		\sum_j \| \E_j\bU^\top \bZ \be_j\|_2^2
		]^{\frac 12}
		\E[ \sum_j \|\be_j^\top \hat \bL\|_2^2]^{\frac 12}
		\\&\le (p \E[\fnorm{\bU}^2])^{\frac 12}
		\E[ \|\hat \bL\|_{\rm F}^2]^{\frac 12}
	\end{align*}
	where we used that $\| \ba \bb^\top \|_{\rm F} = \|\ba\|_2\|\bb\|_2$ for two vectors $\ba,\bb$,
	the Cauchy-Schwarz inequality, $\E[\|\bA \bz_j\|_2^2|\bA] = \|\bA\|_{\rm F}^2$
	if matrix $\bA$ is independent of $\bz_j\sim \calN(0,I_n)$ (set $\bz_j=\bZ \be_j$),
	and Jensen's inequality.
	
	Next, we decompose \eqref{eq:first-term-swap-trick} as
	$\sum_j \bU^\top(\bV  - \E_j\bV )
	+
	\sum_j(\bU - \E_j\bU)^\top\E_j\bV 
	$.
	We have by the submultiplicativity of the Frobenius norm
	and the Cauchy-Schwarz inequality
	\begin{align*}
		\E[\|\bU^\top\sum_{j}(\bV  - \E_j\bV)\|_{\rm F}]
		& \le \E[\sum_j\fnorm{\bU} \| \bV  - \E_j\bV \|_{\rm F}]
		\\&\le \E[p \fnorm{\bU}^2]^{\frac 12} 
		\E[\sum_j\| \bV  - \E_j\bV \|_{\rm F}^2]^{\frac 12}.
	\end{align*}
	By the Gaussian Poincar\'e inequality applied $p$ times,
	$\E[\sum_j\| \bV  - \E_j\bV \|_{\rm F}^2] \le \|\bV  \|_{\partial}^2$,
	so that the previous display is bounded from above by $\sqrt{p}\|\bV  \|_{\partial}$.
	Similarly, $\E[\|\sum_j(\E_j[\bU] - \bU)^\top \E_j\bV \|_{\rm F}]
	\le \sqrt{p} \| \bU \|_{\partial}$
	and $\E[\|\eqref{eq:first-term-swap-trick}\|_{\rm F}] \le \sqrt p (\|\bU\|_{\partial}
	+ \|\bV \|_{\partial})$.
	
	For the last remaining term, \eqref{eq:difficult-term},
	we first use 
	$\E[\| \eqref{eq:difficult-term} \|_{\rm F}]
	\le
	\E[\| \eqref{eq:difficult-term} \|_{\rm F}^2]^{\frac 12}
	$ by Jensen's inequality and now proceed to bound
	$\| \eqref{eq:difficult-term} \|_{\rm F}^2$. We have
	$$
	\| \eqref{eq:difficult-term} \|_{\rm F}^2
	=
	\|\sum_j \E_j\bU^\top \E_j\bV - \E_j\bU^\top X \be_j \be_j^\top \bZ^\top \E_j\bV\|_{\rm F}^2
	=
	\sum_{j,k} \trace[\bM_j^\top \bM_k],
	$$
	where $\bM_j = \E_j\bU^\top \E_j\bV - \E_j\bU^\top \bZ \be_j \be_j^\top \bZ^\top \E_j\bV$. 
	We first bound $\sum_j \|\bM_j\|_{\rm F}^2$.
	Since the variance of $\ba^\top\bb - \ba^\top \bg \bg^\top \bb$ for standard normal $\bg\sim \calN(0,I_p)$
	is $2\|(\ba \bb^\top + \bb \ba^\top)/2\|_{\rm F}^2\le 2 \|\ba\|_2^2 \|\bb\|_2^2$, 
	applying this variance bound on each pair of coordinates $(t,t')\in[T]\times [T]$ 
	gives $\sum_j \|\bM_j\|_{\rm F}^2 \le 
	\sum_j 2 \|\E_j[\bU]\|_{\rm F}^2\|\E_j\bV\|_{\rm F}^2
	\le 2p$.
	
	We now bound $\sum_{j\ne k} \trace[\bM_j^\top\bM_k]$.
	Setting $\bz_j = \bZ \be_j \sim \calN(0,I_n)$ for every $j\in[p]$,
	we will use many times
	the identity
	\begin{equation}
		\label{stein}
		\E[(\bz_j^\top f(\bZ)  - \sum_i \partial_{ij} f(\bZ)^\top \be_i )g(\bZ)
		=
		\E[
		\sum_{i}
		f(\bZ)^\top \be_i \partial_{ij} g(\bZ)
		]
	\end{equation}
	which follows from Stein's formula
	for $f:\R^{n\times p}\to \R^n$ and $g:\R^{n\times p}\to \R$.
	With $f^{tt'}(\bZ) =(\bz_j^\top \E_j[\bU] \be_{t'}) \E_j\bV e_t$
	and $g^{tt'}(\bZ) = \be_{t'}^\top \bM_k e_t$,
	we find
	\begin{align*}
		&\E\trace[\bM_j^\top \bM_k]
		=\E\trace [\bM_j^\top \sum_t \be_{t'} \be_{t'}^\top \bM_k]
		=\E[\sum_{tt'} e_t^\top  \bM_j^\top \be_{t'} \be_{t'}^\top \bM_k e_t]
		\\&=\E
		\sum_{tt'}
		\Bigl(\bz_j^\top f^{tt'}(\bZ) - \sum_i \be_i^\top \partial_{ij} f^{tt'}(\bZ)\Bigr)
		g^{tt'}(\bZ)
		\\&
		=\E
		\sum_{tt'}\sum_i
		\be_i^\top f^{tt'}(\bZ) \partial_{ij} g^{tt'}(\bZ).
	\end{align*}
	where $g_{tt'}(\bZ) = 
	(
	e_t^\top \E_k \bV  ^\top \E_k \bV  e_t' 
	-
	e_t^\top \E_k   \bU^\top \bz_k \bz_k^\top \E_k \bV  e_t'
	)$ and
	$$\partial_{ij} g_{tt'} =
	\partial_{ij}
	\be_{t'}^\top \bM_k e_t
	=
	\be_{t'}^\top \partial_{ij}[\E_k \bU ^\top \E_k \bV ] e_t 
	-
	\bz_k^\top\partial_{ij}[\E_k \bU \be_{t'} \be_{t}^\top \E_k  \bU^\top] \bz_k
	$$
	Now define   $\tilde f^{tt'}(\bZ) =
	\partial_{ij}[\E_k \bU \be_{t'} \be_{t}^\top \E_k   \bU^\top] \bz_k$
	and $\tilde g^{tt'}(\bZ) = \sum_i \be_i^\top f^{tt'}(\bZ)$.
	By definition of $\tilde f^{tt'}(\bZ)$,
	the previous display is equal to
	$\bz_k^\top\tilde f^{tt'}(\bZ) - \sum_l \partial_{lk} \be_l^\top\tilde f^{tt'}(\bZ)$.
	We apply \eqref{stein} again with respect to $\bz_k$, so that
	\begin{align*}
		\E \trace[\bM_j^\top \bM_k]
		&=
		\sum_{il, tt'}
		\be_i^\top \partial_{lk}[f^{tt'}(\bZ)]
		\be_l^\top \tilde f^{tt'}(\bZ)
		\\&=
		\sum_{il, tt'}
		\Bigl(\be_i^\top \partial_{lk}\Bigl[\E_j\bV e_t \be_{t'}^\top \E_j\bU^\top\Bigr]\bz_j\Bigr)
		\Bigl(\be_l^\top \partial_{ij}\Bigl[\E_k[\bU]      \be_{t'} e_t^\top \E_k[\bV ]^\top\Bigr] \bz_k\Bigr).
	\end{align*}
	To remove the indices $t,t'$, we rewrite the above using $\sum_t \be_t \be_t^\top=I_T$
	and $\sum_{t'} \be_{t'} \be_{t'}^\top = I_T$ so that it equals
	$$
	\E\sum_{il}\trace
	\Bigl\{
	\partial_{lk}\Bigl[
	\E_j\bU^\top \bz_j \be_i^\top \E_j\bV 
	\Bigr]
	\partial_{ij}\Bigl[
	\E_k[\bV ]^\top \bz_k \be_l^\top \E_k[\bU] 
	\Bigr]
	\Bigr\}
	$$
	Summing over $j,k$,
	using $\trace [\bA^\top \bB ] \le \|\bA\|_{\rm F}\|\bB\|_{\rm F})$
	and the Cauchy-Schwarz inequality,
	the above is bounded from above by
	\begin{align*}
		&
		\Bigl\{
		\E\sum_{jk, il}
		\Big\|
		\partial_{lk}
		\Bigl[
		\E_j 
		\bU^\top
		\bz_j
		\be_i^\top
		\E_j
		[\bV ]
		\Bigr]
		\Big\|_{\rm F}^2
		\Bigl\}^{\frac 12}
		\Bigl\{
		\E
		\sum_{jk, il}
		\Big\|
		\partial_{ij}\Bigl[
		\E_k[\bV ]^\top \bz_k
		\be_l^\top
		\E_k[\bU]  
		\Bigr]
		\Big\|_{\rm F}^2
		\Bigr\}^{\frac 12}.
	\end{align*}
	At this point the two factors are symmetric,
	with $(\bV , \bU)$ in the left factor
	replaced with $(\bU,\bV )$ on the right factor.
	We focus on the left factor; similar bound will
	apply to the right one by exchanging the roles of $\bV $ and $\bU$.
	If $\bz_j$ is independent of matrices $A^{(q)}$
	$\E_j[\|\sum_{q=1}^n (\be_q^\top\bz_j) A^{(q)}\|_{\rm F}^2
	=
	\sum_{q=1}^n
	\|A^{(q)}\|_{\rm F}^2
	$ so that
	with $A^{(q)} =\partial_{lk}[\E_j \bU^\top \be_q \be_i^\top \E_j \bU]$, the
	first factor in the above display is equal to
	\begin{align*}
		&\quad~\Bigl\{
		\E\sum_{jk, ilq}
		\Big\|
		\partial_{lk}
		\Bigl(
		\E_j 
		\bU^\top
		\be_q
		\be_i^\top
		\E_j\bV
		\Bigr)
		\Big\|_{\rm F}^2
		\Bigl\}^{\frac 12}
		\\&\stackrel{\text{(i)}}{=}
		\Bigl\{
		\E\sum_{jk, ilq}
		\Big\|
		\partial_{lk}
		\Bigl(
		\E_j 
		\bU^\top
		\Bigr)
		\be_q
		\be_i^\top
		\E_j
		[\bV ]
		+
		\E_j 
		[\bU] ^\top
		\be_q
		\be_i^\top
		\partial_{lk}
		\Bigl(
		\E_j
		[\bV ]
		\Bigr)
		\Big\|_{\rm F}^2
		\Bigl\}^{\frac 12}
		\\&\stackrel{\text{(ii)}}{\le}
		\Bigl\{
		\E\sum_{jk, ilq}
		\Big\|
		\partial_{lk}
		\Bigl(
		\E_j 
		[\bU^\top]
		\Bigr)
		\be_q
		\be_i^\top
		\E_j
		[\bV ]
		\Big\|_{\rm F}^2
		\Bigl\}^{\frac 12}
		+
		\Bigl\{
		\E\sum_{jk, ilq}
		\Big\|
		\E_j 
		\bU^\top
		\be_q
		\be_i^\top
		\partial_{lk}
		\Bigl(
		\E_j
		[\bV ]
		\Bigr)
		\Big\|_{\rm F}^2
		\Bigl\}^{\frac 12}
		\\&\stackrel{\text{(iii)}}{=}
		\Bigl\{
		\E\sum_{jk, ilq}
		\Big\|
		\E_j [ \partial_{lk} \bU]^\top
		\be_q
		\Big\|_2^2
		\Big\|
		\be_i^\top
		\E_j\bV
		\Big\|_2^2
		\Bigl\}^{\frac 12}
		+
		\Bigl\{
		\E\sum_{jk, ilq}
		\Big\|
		\E_j 
		\bU^\top
		\be_q
		\Big\|_2^2
		\Big\|
		\be_i^\top
		\E_j
		[ \partial_{lk} \bV ]
		\Big\|_2^2
		\Bigl\}^{\frac 12}
		\\&\stackrel{\text{(iv)}}{=}
		\Bigl\{
		\E\sum_{jk, l}
		\Big\|
		\E_j [ \partial_{lk} \bU]^\top
		\Big\|_{\rm F}^2
		\Big\|
		\E_j\bV
		\Big\|_{\rm F}^2
		\Bigl\}^{\frac 12}
		+
		\Bigl\{
		\E\sum_{jk, l}
		\Big\|
		\E_j 
		\bU^\top
		\Big\|_{\rm F}^2
		\Big\|
		\E_j
		[ \partial_{lk} \bV ]
		\Big\|_{\rm F}^2
		\Bigl\}^{\frac 12}
	\end{align*}
	where (i) is the chain rule,
	(ii) the triangle inequality,
	(iii) holds
	provided that the order of the derivation $\partial_{lk}$
	and the expectation sign $\E_j$ can be switched and using
	$\|\ba \bb^\top\|_{\rm F}^2 = \|\ba\|_2^2 \|\bb\|_2^2$ for vectors $\ba, \bb$,
	and (iv) holds using
	$\sum_i \|\bA \be_i\|_2^2 = \|\bA\|_{\rm F}^2 = \sum_q \|\bA \be_q\|_2^2$ for a matrix $\bA$ with $n$ columns.
	Finally, by Jensen's inequality, the above display is bounded by
	\begin{equation*}
		\Bigl\{
		\E\sum_{k, l}
		\Big\|
		\partial_{lk}
		\bU
		\Big\|_{\rm F}^2
		\sum_j
		\Big\|
		\E_j
		[\bV ]
		\Big\|_{\rm F}^2
		\Bigl\}^{\frac 12}
		+
		\Bigl\{
		\E\sum_{k, l}
		\Big\|
		\partial_{lk}
		\bV 
		\Big\|_{\rm F}^2
		\sum_j
		\Big\|
		\E_j
		[\bU]
		\Big\|_{\rm F}^2
		\Bigl\}^{\frac 12}
		.
	\end{equation*}
	Since $\fnorm{\bU}\vee \|\bV \|_{\rm F} \le 1$ almost surely,
	the previous display is bounded by
	$\sqrt{p} (\|\bU\|_\partial + \|\bV \|_{\partial})$.
	In summary,
	$$
	\E[
	\|\eqref{eq:difficult-term}
	\|_{\rm F}^2
	]^{\frac 12}
	\le(2p + [2\sqrt{p}(\|\bU\|_{\partial} + \|\bV \|_{\partial})]^2)^{\frac 12}
	\le \sqrt{2p} + 2\sqrt{p}(\|\bU\|_{\partial} + \|\bV \|_{\partial}). 
	$$
	Combining the bounds on the terms
	\eqref{eq:first-term-swap-trick}-\eqref{eq:difficult-term}-\eqref{eq:Xi-hat-Xi-term}-\eqref{eq:last-two-terms}
	with the triangle inequality completes the proof.
\end{proof}

\begin{proof}[Proof of Proposition~\ref{prop: V1}]
Note that $\bQ_1 = \frac{\bE^\top\bF/\tD  - \bS(n\bI_T - \hbA)/\tD}{\fnorm{\bS^{\frac 12}}}$, where $\tD = \big(\fnorm*{\bF}^2 + n \trace(\bS)\big)^{\frac 12}$ is defined in \Cref{lem: stein-EF}.
 Now, Apply \Cref{lem: stein-EF}, we obtain
	\begin{align*}
		\E \| \bQ_1\|^2_{\rm F} 
		= \E \big[\| \bE^\top \bF/\tD - \bS (n\bI_T - \hbA )/\tD\|_{\rm F}^2\big] \frac{1}{\trace(\bS)}
		\le 4.
	\end{align*}
\end{proof}

Before proving Proposition~\ref{prop: V6}], we state the following  \Cref{lem: Rems}, whose proof is deferred to the end of this section. 

\begin{lemma}\label{lem: Rems}
	We have 
	\begin{align*}
		\sum_{ij}\bF^\top\bZ\be_j\be_i^\top\frac{\partial \bF}{\partial z_{ij}} 
		&= \bJ_1 - \bF^\top\bZ\bH (n\bI_T -\hbA),\\
		\sum_{ij}\Big(\frac{\partial \bF}{\partial z_{ij} }\Big)^\top
		\bZ \be_j \be_i^\top \bF
		&= \bJ_2 -  \hbA \bF^\top\bF.
	\end{align*}
	where $\bJ_1 = \sum_{ij}\bF^\top\bZ\be_j \sum_t \Delta_{ij}^{it} \be_t^\top$ with $\fnorm*{\bJ_1}\le n^{\frac12} \fnorm*{\bF}^2$, and  
	$\bJ_2=\sum_{ijlt}\be_t D_{ij}^{lt}\be_l^\top\bZ\be_j \be_i^\top\bF$ with $\fnorm*{\bJ_2}\le n^{\frac 12}\opnorm*{\bZ}\fnorm{\bH} \fnorm*{\bF}$.
%
\end{lemma}

\begin{proof}[Proof of Proposition~\ref{prop: V6}]\label{proof: V6}
	We first apply \Cref{lem:steinX}. To be more specific, let $\bU = n^{-\frac 12}\bF/D$ and $\bV =n^{-\frac 12} \bZ^\top\bU$ with $D = (\fnorm*{\bF}^2/n + \fnorm{\bH}^2)^{\frac 12}$, then $\fnorm*{\bU}\le 1$, $\fnorm*{\bV}\le n^{-\frac12}\opnorm*{\bZ}$. \Cref{lem:steinX} yields  
	\begin{align}\label{eq: UZV}
		&\E \Big(\Big\|\bU^\top \bZ \bV - 
		\sum_{j=1}^p\sum_{i=1}^n
		\frac{\partial}{\partial z_{ij} }\Bigl(\bU^\top \be_i \be_j^\top \bV \Bigr)\Big\|_{\rm F}^2\Big)\\
		\le~ & \E [\fnorm*{\bU}^2 \fnorm*{\bV}^2]+ \E \sum_{ij}\Big[
		2\fnorm*{\bV}^2\fnorm*{ \frac{\partial \bU}{\partial z_{ij}} }^2
		+ 2\fnorm*{\bU}^2\fnorm*{ \frac{\partial \bV}{\partial z_{ij}} }^2\Big]\notag\\
		\le~&\E \big(n^{-1}\opnorm*{\bZ}^2\big) + 
		2\E \Big(n^{-1}\opnorm*{\bZ}^2\sum_{ij}
		\fnorm*{ \frac{\partial \bU}{\partial z_{ij}} }^2\Big)
		+ 2 \E \Big(\sum_{ij} \fnorm*{ \frac{\partial \bV}{\partial z_{ij}} }^2 \Big)\\ \notag 
		\le~&\frac{4p}{n} 
		+  \E \big(\frac 1n\opnorm{\bZ}^2\big) 
		+ 6\E \Big(n^{-1}\opnorm*{\bZ}^2\sum_{ij}
		\fnorm*{ \frac{\partial \bU}{\partial z_{ij}} }^2\Big),
	\end{align}
where the last inequality uses the following bound derived using $\bV = n^{-\frac 12} \bZ^\top\bU$, and $\fnorm{\bU}\le1$,
	\begin{align}
	\sum_{ij}
	\fnorm*{ \frac{\partial \bV}{\partial z_{ij}} }^2\notag
	=~& \sum_{ij}
	n^{-1}\fnorm*{\frac{\partial \bZ^\top}{\partial z_{ij}} \bU + \bZ^\top \frac{\partial \bU^\top}{\partial z_{ij}} }^2\notag\\
	\le~& 2{n}^{-1} \Big(p\fnorm*{\bU}^2 + \opnorm*{\bZ}^2 \sum_{ij} \fnorm*{ \frac{\partial \bU}{\partial z_{ij}} }^2\Big)\notag\\
	\le~& \frac{2p}{n} + 2n^{-1}\opnorm*{\bZ}^2 \sum_{ij} \fnorm*{ \frac{\partial \bU}{\partial z_{ij}} }^2. \label{eq: Vzij}
\end{align}
Note that $\E [(n^{-\frac12}\opnorm*{\bZ})^2] \le (1 + \sqrt{p/n})^2 + 1/n$ by \eqref{eq: opnorm-Z}. 
Now we establish the connection between $\bQ_2$ and the term inside Frobenius norm in \eqref{eq: UZV}. By definitions of $\bU$ and $\bV$, 
\begin{equation}\label{eq: uxv}
	\bU^\top \bZ \bV = n^{-\frac 32}D^{-2}\bF^\top \bZ \bZ^\top\bF.
\end{equation}
Next, by product rule,
\begin{align*}
	&\sum_{j=1}^p\sum_{i=1}^n\frac{\partial}{\partial z_{ij} }\Bigl(\bU^\top \be_i \be_j^\top \bV \Bigr)\\
	=~& n^{-\frac 32}\sum_{j=1}^p\sum_{i=1}^n\frac{\partial}{\partial z_{ij} }\Bigl(\bF^\top \be_i \be_j^\top \bZ^\top\bF D^{-2} \Bigr)\\
	=~& n^{-\frac 32}\sum_{j=1}^p\sum_{i=1}^n\Big( \underbrace{\frac{\partial \bF^\top}{\partial z_{ij}} \be_i \be_j^\top \bZ^\top\bF D^{-2}}_{(i)} +\underbrace{\bF^\top\be_i \be_j^\top\frac{\partial \bZ^\top\bF}{\partial z_{ij}}  D^{-2}}_{(ii)} +  \underbrace{\bF^\top\be_i \be_j^\top\bZ^\top\bF\frac{\partial D^{-2}}{\partial z_{ij}}}_{(iii)} \Big).
\end{align*}
We now rewrite the above three terms $(i), (ii)$ and $(iii)$.
\begin{enumerate}
	\item[(i)] For term ($i$), by \Cref{lem: Rems}, 
	\begin{align*}
		&n^{-\frac 32}D^{-2}\sum_{j=1}^p\sum_{i=1}^n\frac{\partial \bF^\top}{\partial z_{ij}} \be_i \be_j^\top \bZ^\top\bF\\
		=~& n^{-\frac 32}D^{-2} \bigl[ \bJ_1 - \bF^\top\bZ\bH(n\bI_T -\hbA) \bigr]^\top\\
		=~& n^{-\frac 32}D^{-2} \bigl[ \bJ_1^\top - (n\bI_T -\hbA) \bH^\top \bZ^\top \bF \bigr].
	\end{align*}
	
	\item[(ii)] For term ($ii$), by product rule and \Cref{lem: Rems}, 
	\begin{align*}
		&n^{-\frac 32}D^{-2}\sum_{j=1}^p\sum_{i=1}^n \bF^\top\be_i \be_j^\top\frac{\partial \bZ^\top\bF}{\partial z_{ij}} \\
		=~&n^{-\frac 32}D^{-2}\Big( p\bF^\top\bF +
		\sum_{j=1}^p\sum_{i=1}^n \bF^\top\be_i \be_j^\top\bZ^\top\frac{\partial \bF}{\partial z_{ij}} \Big)\\
		=~&n^{-\frac 32}D^{-2}\Big( p\bF^\top\bF +
		(\bJ_2 - \hbA\bF^\top\bF)^\top \Big)\\
		=~&n^{-\frac 32}D^{-2}\big( p\bF^\top\bF - \bF^\top\bF\hbA + \bJ_2^\top \big).
	\end{align*}
	
	\item[(iii)] For term ($iii$), by chain rule, 
		\begin{align*}
		&n^{-\frac 32}\sum_{j=1}^p\sum_{i=1}^n \bF^\top\be_i \be_j^\top\bZ^\top\bF\frac{\partial D^{-2}}{\partial z_{ij}}\\
		=~& -2n^{-\frac 32}D^{-3}\sum_{j=1}^p\sum_{i=1}^n \bF^\top\be_i \be_j^\top\bZ^\top\bF\frac{\partial D}{\partial z_{ij}}\\
		=~& -n^{-\frac 32}D^{-2} \underbrace{\Big(2D^{-1}\sum_{j=1}^p\sum_{i=1}^n \bF^\top\be_i \be_j^\top\bZ^\top\bF\frac{\partial D}{\partial z_{ij}}\Big)}_{\bJ_3}
	\end{align*}
\end{enumerate}
Combining \eqref{eq: uxv} and the above three expressions for (i)-(ii)-(iii), 
	\begin{align*}
		&\bU^\top \bZ \bV - \sum_{j=1}^p\sum_{i=1}^n \frac{\partial}{\partial z_{ij} }\Bigl(\bU^\top \be_i \be_j^\top \bV \Bigr)\\
		=~& n^{-\frac 32}D^{-2}\Bigl[\bF^\top \bZ \bZ^\top\bF  + (n\bI_T -\hbA) \bH^\top \bZ^\top \bF - p\bF^\top\bF + \bF^\top\bF\hbA\\
		& -\bJ_1^\top - \bJ_2^\top + \bJ_3\Bigr] \\
		=~& \bQ_2 - n^{-\frac 32}D^{-2} (\bJ_1^\top + \bJ_2^\top - \bJ_3).
	\end{align*}
That is, 
\begin{equation}\label{eq: V6=A + B}
	\bQ_2 =  \bU^\top \bZ \bV - \sum_{j=1}^p\sum_{i=1}^n \frac{\partial}{\partial z_{ij} }\Bigl(\bU^\top \be_i \be_j^\top \bV \Bigr) + n^{-\frac 32}D^{-2} (\bJ_1^\top + \bJ_2^\top - \bJ_3). 
\end{equation}
Note that \Cref{lem: Rems} implies that
\begin{align}\label{eq: steinX_4}
	n^{-\frac 32}D^{-2} \fnorm*{\bJ_1} \le \frac{\fnorm*{\bF}^2/n}{\fnorm*{\bF}^2/n + \fnorm{\bH}^2} \le 1, 
\end{align}
and 
\begin{align}\label{eq: steinX_3}
	n^{-\frac 32}D^{-2} \fnorm*{\bJ_2} \le \Big[n^{-\frac 12} \opnorm*{\bZ} \frac{\fnorm*{\bF}n^{-\frac 12} \fnorm{\bH}}{\fnorm*{\bF}^2/n + \fnorm{\bH}^2} \Big]
	\le \frac12 (n^{-\frac 12} \opnorm*{\bZ}). 
\end{align}
Since $\bJ_3=2D^{-1}\sum_{j=1}^p\sum_{i=1}^n \bF^\top\be_i \be_j^\top\bZ^\top\bF\frac{\partial D}{\partial z_{ij}}$, by Cauchy-Schwarz inequality
\begin{align*}
	n^{-3} D^{-4}\fnorm{\bJ_3}^2 = 
	&\sum_{t,t'} ([\bJ_3]_{t,t'})^2\\
	&= 4 n^{-3} D^{-6}\sum_{t,t'} \big[\sum_{i,j} \be_t^\top\bF^\top\be_i \be_j^\top\bZ^\top\bF\be_{t'}\frac{\partial D}{\partial z_{ij}}\big]^2\\
	&\le 4 n^{-3} D^{-6} \sum_{t,t'}  \Big[\sum_{i,j} [\be_t^\top\bF^\top\be_i \be_j^\top\bZ^\top\bF\be_{t'}]^2 \sum_{i,j}\big(\frac{\partial D}{\partial z_{ij}}\big)^2\Big]\\
	&= 4 n^{-3} D^{-6}
	\fnorm{\bF}^2 
	\fnorm{\bZ^\top\bF}^2  \sum_{i,j}\big(\frac{\partial D}{\partial z_{ij}}\big)^2\\
	&= 4 n^{-3} D^{-6}
	\fnorm{\bF}^4 \opnorm{\bZ}^2 \sum_{i,j}\big(\frac{\partial D}{\partial z_{ij}}\big)^2.
\end{align*}

\begin{itemize}
	\item[(1)] Under \Cref{assu: tau}(i) that $\tau>0$. By \Cref{cor: partialD}, we have $\sum_{i,j}\big(\frac{\partial D}{\partial z_{ij}}\big)^2\le C(\tau')n^{-1}D^2$. Then, 
	\begin{equation}\label{eq: steinX_5}
		n^{-3} D^{-4}\fnorm{\bJ_3}^2 
		\le C(\tau')n^{-4}D^2 \fnorm{\bF}^4 \opnorm{\bZ}^2 
		\le C(\tau')n^{-2} \opnorm{\bZ}^2,
	\end{equation}
	by $\fnorm*{\bF}^2/(nD^2) \le 1$. 
	
	By \eqref{eq: V6=A + B} and triangular inequality, we have
	\begin{equation}\label{eq: steinX_2}
		\begin{aligned}
			\E[\fnorm{\bQ_2}^2]
			\le~& 2 \E \Big[ \fnorm{\bU^\top \bZ \bV - \sum_{j=1}^p\sum_{i=1}^n \frac{\partial}{\partial z_{ij} }\Bigl(\bU^\top \be_i \be_j^\top \bV \Bigr)}^2 \Big]\\
			&+ 2\E [\fnorm{n^{-\frac 32}D^{-2} (\bJ_1^\top + \bJ_2^\top- \bJ_3)}^2].
		\end{aligned}
	\end{equation}
By \eqref{eq: steinX_4}, \eqref{eq: steinX_3}, \eqref{eq: steinX_5}, the second term in \eqref{eq: steinX_2} can be upper bounded by 
\begin{equation}
	6\E [( 1 + \frac 14 n^{-1}\opnorm{\bZ}^2) + C(\tau') n^{-2}\opnorm{\bZ}^2) ] \le C(\tau') (1 + \frac pn),
\end{equation}
where the last inequality uses \eqref{eq: opnorm-Z}. 

For the first term in \eqref{eq: steinX_2}, since $\sum_{ij} \fnorm{ \frac{\partial \bU}{\partial z_{ij}} }^2 \le C(\tau') (T\wedge \frac pn)$ by \Cref{lem: partialF}, we have 
\begin{align*}
	&\E \Big(\Big\|\bU^\top \bZ \bV - 
	\sum_{j=1}^p\sum_{i=1}^n
	\frac{\partial}{\partial z_{ij} }\Bigl(\bU^\top \be_i \be_j^\top \bV \Bigr)\Big\|_{\rm F}^2\Big) \\ 
	\le~&\frac{4p}{n} 
	+  \E \big(n^{-1}\opnorm{\bZ}^2\big) 
	+ 6\E \Big(n^{-1}\opnorm{\bZ}^2\sum_{ij}
	\fnorm*{ \frac{\partial \bU}{\partial z_{ij}} }^2\Big)\\
	\le~& \frac{4p}{n}  + [1 + C(\tau') (T\wedge \frac pn) ]\E \big(n^{-1}\opnorm{\bZ}^2\big) \\
	\le~& \frac{4p}{n}  + [1 + C(\tau') (T\wedge \frac pn) ] [(1 + \sqrt{p/n})^2 + 1/n]\\
	\le~& C(\tau')(T \wedge (1 + \frac pn))(1 + \frac pn). 
\end{align*}
Therefore, under \Cref{assu: tau}(i), we obtain
\begin{align*}
	\E [\fnorm{\bQ_2}^2] \le C(\tau')(T \wedge (1 + \frac pn))(1 + \frac pn).
\end{align*}
	
	\item[(2)] Under \Cref{assu: tau}(ii), let $\Omega = U_1\cap U_2\cap U_3$, then we have 
	$\P(\Omega^c) \le C(\gamma, c) \frac 1T$ by  \eqref{eq: P-Omega}. 
	 By 
	\Cref{lem: lipschitz-Lasso}, on $\Omega$, we have 
	(i) The map $\bZ\mapsto \bU$ is $n^{-1/2} L_1$-Lipschitz, where $L_1=8\max(1, (2\eta)^{-1})$, and $\fnorm{\bU}\le 1$, 
	(ii) The map $\bZ \mapsto \bV$ is $n^{-1/2} L_2$-Lipschitz, where $L_2 = (1 + (2 +\sqrt{p/n})L_1)$, and $\fnorm{\bV}\le (2 + \sqrt{p/n})$. 
	Applying \Cref{cor:steinX} with $K = 2 + \sqrt{p/n}$ yields
	\begin{align*}
		&\E\Big[I(\Omega) \fnorm{\bU^\top \bZ \bV - 
			\sum_{j=1}^p\sum_{i=1}^n
			\frac{\partial}{\partial z_{ij} }\Bigl(\bU^\top \be_i \be_j^\top \bV \Bigr)
		}^2\Big]\\
		\le~& K^2 + C(\gamma, c)( K^2 L_1^2  + L_2^2 ) + 2\E \Big[I(\Omega)
		\sum_{ij}\Big( K^2\fnorm*{\frac{\partial \bU}{\partial z_{ij}} }^2 + \fnorm*{\frac{\partial \bV}{\partial z_{ij}} }^2\Big) \Big]\\
		\le~& C(\gamma, c), 
	\end{align*}
where the last inequality holds because $K \le C(\gamma), L_1 = C(\gamma, c), L_2 = C(\gamma, c)$, and on $\Omega$, $\sum_{ij} \fnorm{\frac{\partial \bU}{\partial z_{ij}} }^2\le C(\gamma, c)$ from \Cref{lem: partialF}, and $\E [\fnorm{\frac{\partial \bV}{\partial z_{ij}} }^2 ]\le C(\gamma, c)$ by product rule. Therefore, under \Cref{assu: tau}(i), we obtain
\begin{align*}
	\E [I(\Omega)\fnorm{\bQ_2}^2] 
		\le~& 2 \E \Big[ \fnorm{\bU^\top \bZ \bV - \sum_{j=1}^p\sum_{i=1}^n \frac{\partial}{\partial z_{ij} }\Bigl(\bU^\top \be_i \be_j^\top \bV \Bigr)}^2 \Big]\\
	&+ 2\E [\fnorm{n^{-\frac 32}D^{-2} (\bJ_1^\top + \bJ_2^\top- \bJ_3)}^2]\\
		\le~& C(\gamma, c),
\end{align*}
where the last inequality used \eqref{eq: steinX_4}, \eqref{eq: steinX_3}, and that
$n^{-3} \E[D^{-4}\fnorm{\bJ_3}^2] \le  C(\gamma, c)$ in analogy to  \eqref{eq: steinX_5}. 

\end{itemize}
\end{proof}

\begin{proof}[Proof of Proposition~\ref{prop: V8}]\label{proof: V8}
	We will apply \Cref{lem: Chi2type}. 
	Let $\bU = \bV = n^{-\frac12} D^{-1} \bF$ with $D = (\fnorm*{\bF}^2/n + \fnorm{\bH}^2)^{\frac 12}$, then $\fnorm*{\bU}=\fnorm*{\bV}\le 1$. Let $\bW_0 = p \bU^\top \bU - \sum_{j=1}^p \big(\sum_{i=1}^n \frac{\partial \bU^\top\be_i}{\partial z_{ij}} - \bU^\top \bZ\be_j\big) \big(\sum_{i=1}^n \frac{\partial \bU^\top\be_i}{\partial z_{ij}} - \bU^\top \bZ\be_j\big)^\top$, 
	then \Cref{lem: Chi2type} gives 
	\begin{align*}
		\E [\fnorm*{\bW_0}] &\le 2 \|\bU\|_\partial \|\bV\|_\partial
		+ \sqrt p \bigl( \sqrt 2 + (3+\sqrt{2})(\|\bU\|_{\partial} + \|\bV\|_{\partial}) \bigr) \\
		&= 2\|\bU\|_\partial^2 + \sqrt{p} (\sqrt 2 + 2(3+\sqrt{2}) \|\bU\|_{\partial}).
	\end{align*}
We will prove under \Cref{assu: tau}(i), and the proof under \Cref{assu: tau}(ii) on set $\Omega = U_1\cap U_2\cap U_3$ follows from almost similar arguments with $\tau'$ replaced by $\eta$, which is a constant that depends only on $\gamma, c$.
	
	By definition of $\norm*{\bU}_{\partial}$ and \Cref{lem: partialF}, 
	$\norm{\bU}_{\partial}^2= \sum_{ij}\E  \fnorm{\frac{\partial \bU}{\partial z_{ij}} }^2 \le C(\gamma, \tau')$. Thus $\E [\fnorm*{\bW_0}] \le C(\tau', \gamma) \sqrt{p}$. 
	
	Now we establish the connection between $\bW_0$ and $\bQ_3$. Since $\bU = n^{-\frac12} D^{-1} \bF$, by product rule, 
	\begin{align*}
		&\sum_{i=1}^n \frac{\partial \bU^\top\be_i}{\partial z_{ij}} - \bU^\top \bZ\be_j \\
		=~& n^{-\frac12} \Big(\sum_{i=1}^n \frac{\partial D^{-1}\bF^\top \be_i}{\partial z_{ij}} - D^{-1}\bF^\top \bZ\be_j\Big)\\
		=~& n^{-\frac12}\Big(\sum_{i=1}^n D^{-1}\frac{\partial \bF^\top \be_i}{\partial z_{ij}} - D^{-1}\bF^\top \bZ\be_j\Big) + n^{-\frac12}\sum_{i=1}^n \bF^\top \be_i\frac{\partial D^{-1}}{\partial z_{ij}} \\
		=~& n^{-\frac12}D^{-1}\Big(\sum_{i=1}^n \frac{\partial \bF^\top \be_i}{\partial z_{ij}} - \bF^\top \bZ\be_j \Big) - n^{-\frac12}D^{-2}\sum_{i=1}^n\bF^\top \be_i\frac{\partial D}{\partial z_{ij}}.
	\end{align*}
	For the first term in the last display, we have by \Cref{lem: Dijlt} 
	\begin{align*}
		\sum_{i=1}^n \frac{\partial \be_i^\top \bF}{\partial z_{ij}} = \sum_{i=1}^n \sum_{t=1}^T \frac{\partial \be_i^\top \bF\be_t}{\partial z_{ij}}\be_t^\top = \sum_{it} (D_{ij}^{it} + \Delta_{ij}^{it})\be_t^\top = -\be_j^\top \bH (n\bI_T - \hbA) + 
		\sum_{it} \Delta_{ij}^{it}\be_t^\top. 
	\end{align*}
	Hence,
	\begin{align}
		& \sum_{i=1}^n \frac{\partial \bU^\top\be_i}{\partial z_{ij}} - \bU^\top \bZ\be_j \nonumber\\
		=~& n^{-\frac12}D^{-1}\big[ -(n\bI_T - \hbA) \bH^\top \be_j + \sum_{it} \Delta_{ij}^{it}\be_t - \bF^\top\bZ\be_j \big] - n^{-\frac12}D^{-2}\sum_{i=1}^n\bF^\top \be_i\frac{\partial D}{\partial z_{ij}}\nonumber\\
		=~& -n^{-\frac12}D^{-1}\big[ (n\bI_T - \hbA) \bH^\top + \bF^\top\bZ\big]\be_j + n^{-\frac12}D^{-1} \sum_{it} \Delta_{ij}^{it}\be_t - n^{-\frac12}D^{-2}\sum_{i=1}^n\bF^\top \be_i\frac{\partial D}{\partial z_{ij}}.\label{eq: V8_proof}
	\end{align}
	Let $\bW_1 = -n^{-\frac12}D^{-1}\big[ (n\bI_T - \hbA) \bH^\top + \bF^\top\bZ\big]$ be the first term in \eqref{eq: V8_proof}. 
	For the second term in \eqref{eq: V8_proof}, recall 
	$\Delta_{ij}^{lt} = -(\be_t^\top \otimes \be_l^\top)(\bI_T\otimes \bX)
	\bM^\dagger (\bI_T\otimes \bSigma^{\frac12})
	\bigl(\bF^\top \otimes \bI_{p}\bigr)(\be_i \otimes\be_j)$ in \Cref{lem: Dijlt},
	\begin{align*}
		& n^{-\frac12}D^{-1} \sum_{it} \Delta_{ij}^{it}\be_t \\
		=~& - n^{-\frac12}D^{-1}\sum_{it} \be_t (\be_t^\top \otimes \be_i^\top)
		(\bI_T\otimes \bX)
		\bM^\dagger (\bI_T\otimes \bSigma^{\frac12}) \bigl(\bF^\top\be_i \otimes \be_j \bigr)\\
		=~& - n^{-\frac12}D^{-1}\sum_{i} (\bI_T \otimes \be_i^\top\bX)
		\bM^\dagger (\bI_T\otimes \bSigma^{\frac12}) \bigl(\bF^\top\be_i \otimes  \bI_p\bigr)\be_j\\
		=~& -n^{-\frac12}D^{-1} \sum_{i=1}^n(\bI_T \otimes \be_i^\top\bX) \bM^\dagger (\bI_T\otimes \bSigma^{\frac12}) (\bF^\top \be_i\otimes \bI_p) \be_j\\
		=~& \bW_2 \be_j,
	\end{align*}
	where $\bW_2 =  -n^{-\frac12}D^{-1} \sum_{i=1}^n(\bI_T \otimes \be_i^\top\bX) \bM^\dagger (\bI_T\otimes \bSigma^{\frac12}) (\bF^\top \be_i\otimes \bI_p)$. 
	For the third term in \eqref{eq: V8_proof}, 
	$$- n^{-\frac12}D^{-2}\bF^\top\sum_{i=1}^n \be_i\frac{\partial D}{\partial z_{ij}} = \bW_3 \be_j,$$
	where $\bW_3 = - n^{-\frac12}D^{-2} \bF^\top \frac{\partial D}{\partial \bZ}$, here we slightly abuse the notation and let $\frac{\partial D}{\partial \bZ}$ denote the $n\times p$ matrix with $(i,j)$-th entry being $\frac{\partial D}{\partial z_{ij}}$.
	Therefore, \eqref{eq: V8_proof} can be simplified as 
	\begin{align*}
		\sum_{i=1}^n \frac{\partial \bU^\top\be_i}{\partial z_{ij}} - \bU^\top \bZ\be_j = \big[\bW_1 + \bW_2 + \bW_3 \big]\be_j.
	\end{align*}
	Furthermore,
	\begin{align*}
		\bW_0 &= p \bU^\top \bU - \sum_{j=1}^p\Big(\sum_{i=1}^n \frac{\partial \bU^\top\be_i}{\partial z_{ij}} - \bU^\top \bZ\be_j\Big) \Big(\sum_{i=1}^n \frac{\partial \bU^\top\be_i}{\partial z_{ij}} - \bU^\top \bZ\be_j\Big)^\top\\
		&= p \bU^\top \bU - [\bW_1 + \bW_2 + \bW_3][\bW_1 + \bW_2 + \bW_3]^\top\\
		&= n^{\frac 12} \bQ_3 - \bW_1(\bW_2+\bW_3)^\top -(\bW_2+\bW_3)\bW_1^\top - (\bW_2+\bW_3) (\bW_2+\bW_3)^\top,
	\end{align*}
	where the last equality is due to 
	\begin{align*}
		&p \bU^\top \bU - \bW_1\bW_1^\top \\
		=~& n^{-1} D^{-2} \Big[p\bF^\top\bF -\bF^\top \bZ \bSigma^{-1}\bZ^\top \bF  - (n\bI_T -\hbA)\bH^\top\bH (n\bI_T -\hbA) \\
		& - (n\bI_T -\hbA)\bH^\top\bZ^\top\bF - \bF^\top\bZ\bH(n\bI_T -\hbA)\Big]\\
		=~& n^{\frac12} \bQ_3.
	\end{align*}
	Therefore, 
	\begin{align}\label{eq: V8-W}
		\bQ_3 = n^{-\frac 12} \Big[ \bW_0 + 
		\bW_1(\bW_2+\bW_3)^\top + (\bW_2+\bW_3)\bW_1^\top + (\bW_2+\bW_3) (\bW_2+\bW_3)^\top
		\Big].
	\end{align}
	We then bound the norms of $\bW_1, \bW_2, \bW_3$. For $\bW_1$, 
	\begin{align*}
		\fnorm*{\bW_1} &= n^{-\frac12}D^{-1} \fnorm{(n\bI_T - \hbA) \bH^\top + \bF^\top\bZ} \\
		&\le n^{\frac12} \big(D^{-1} \fnorm{\bH} + n^{-1}D^{-1} \fnorm*{\bF}\opnorm*{\bZ} \big)\\
		&\le n^{\frac12} + n^{\frac12}  \big( D^{-1} \fnorm*{\bF}/\sqrt{n}\opnorm*{\bZ/\sqrt{n}}\big)\\
		&\le n^{\frac12} ( 1 + \opnorm{\bZ/\sqrt{n}}),
	\end{align*}
	where we used $\opnorm*{\bI_T - \hbA/n}\le 1$ by \Cref{lem: aux-tau>0} , $D^{-1} \fnorm{\bH} \le 1$, and $D^{-1} \fnorm*{\bF}/\sqrt{n} \le 1$. 
	
For $\bW_2$,  
		\begin{align*}
		\opnorm*{\bW_2} 
		&= n^{-\frac12} D^{-1}\opnorm[\Big]{\sum_{i=1}^n(\bI_T \otimes \be_i^\top\bX) \bM^\dagger (\bI_T\otimes \bSigma^{\frac12}) (\bF^\top \be_i\otimes \bI_p)}\\
		&\le n^{-\frac12} D^{-1}\sum_{i=1}^n\opnorm{(\bI_T \otimes \be_i^\top\bX) \bM^\dagger (\bI_T\otimes \bSigma^{\frac12}) (\bF^\top \be_i\otimes \bI_p)}\\
		&\le n^{-\frac12} D^{-1} \sum_{i=1}^n  \opnorm{(\bI_T \otimes \be_i^\top\bX) \bM^\dagger (\bI_T\otimes \bSigma^{\frac12}) }
		\opnorm{(\bF^\top \be_i\otimes \bI_p)}\\
		&\le n^{-\frac12} \opnorm{\bSigma}^{\frac12}D^{-1} \sum_{i=1}^n  \opnorm{(\bI_T \otimes \be_i^\top\bX) (\bI_T \otimes (\bX_{\hat{\mathscr{S}}}^\top\bX_{\hat{\mathscr{S}}} + n\tau\bP_{\hat{\mathscr{S}}})^{\dagger})}
		\opnorm{(\bF^\top \be_i\otimes \bI_p)}\\
		&= n^{-\frac12} \opnorm{\bSigma}^{\frac12}D^{-1} \sum_{i=1}^n  \opnorm{\bI_T \otimes [\be_i^\top\bX_{\hat{\mathscr{S}}}(\bX_{\hat{\mathscr{S}}}^\top\bX_{\hat{\mathscr{S}}} + n\tau\bP_{\hat{\mathscr{S}}})^{\dagger}]}
		\norm{\bF^\top \be_i}\\
		&\le n^{-\frac12} \opnorm{\bSigma}^{\frac12}D^{-1} \opnorm{\bX_{\hat{\mathscr{S}}}(\bX_{\hat{\mathscr{S}}}^\top\bX_{\hat{\mathscr{S}}} + n\tau\bP_{\hat{\mathscr{S}}})^{\dagger}}
		\sum_{i=1}^n \norm{\bF^\top \be_i}\\
		&\le n^{-\frac12}D^{-1} n^{-\frac12} (\tau/\opnorm{\bSigma})^{-\frac12}
		n^{\frac 12} \fnorm{\bF}\\
		&\le (\tau')^{-\frac12},
	\end{align*}
	where the third inequality uses $\bM^\dagger \preceq \bI_T \otimes (\bX_{\hat{\mathscr{S}}}^\top\bX_{\hat{\mathscr{S}}} + n\tau\bP_{\hat{\mathscr{S}}})^{\dagger}$, 
	the fourth inequality uses the result that $\opnorm{\be_i^\top\bA} \le \opnorm{\bA}$, 
	the penultimate inequality uses $\opnorm{\bX_{\hat{\mathscr{S}}}(\bX_{\hat{\mathscr{S}}}^\top\bX_{\hat{\mathscr{S}}} + n\tau\bP_{\hat{\mathscr{S}}})^{\dagger}} \le (n\tau)^{-1/2}$ and the Cauchy-Schwarz inequality, 
	the last inequality follows from $n^{-1/2}D^{-1}\fnorm{\bF}\le 1$. 
	It immediately follows that $\fnorm{\bW_2} \le \sqrt{T}C(\tau')$ since the rank of $\bW_2$ is at most $T$. 
	
	For $\bW_3$, using $n^{-\frac12} D^{-1} \fnorm*{\bF}\le 1$, and $\fnorm*{ \frac{\partial D}{\partial \bZ}}\le n^{-\frac12} D C(\tau')$ from \Cref{cor: partialD}, we obtain
	\begin{align*}
		\fnorm*{\bW_3}
		= n^{-\frac12} D^{-2} \fnorm{\bF^\top  \frac{\partial D}{\partial \bZ}} \le n^{-\frac12} D^{-2} \fnorm*{\bF}\fnorm*{ \frac{\partial D}{\partial \bZ}} 
		\le D^{-1}\fnorm*{ \frac{\partial D}{\partial \bZ}}
		\le n^{-1/2} C(\tau'),
	\end{align*}

	The desired inequality follows by combining \eqref{eq: V8-W} and the bounds for $\bW_0, \bW_1, \bW_2, \bW_3$.
\end{proof}

\begin{proof}[Proof of \Cref{lem: Rems}]
	By \Cref{lem: Dijlt}, $\frac{\partial F_{lt}}{\partial z_{ij}} = D_{ij}^{lt} + \Delta_{ij}^{lt}$, where 
	\begin{align}
		D_{ij}^{lt} &= -(\be_j^\top\bH \otimes \be_i^\top)(\bI_{nT} - \bN)(\be_t\otimes \be_l),\label{eq: D_ijlt}\\
		\Delta_{ij}^{lt} &= -(\be_t^\top \otimes \be_l^\top)(\bI_T\otimes \bX)
		\bM^\dagger (\bI_T\otimes \bSigma^{\frac12})
		\bigl(\bF^\top \otimes \bI_{p}\bigr)(\be_i \otimes\be_j). \label{eq: Delta_ijlt}
	\end{align}
	
	For the first equality, since $\be_i^\top \frac{\partial \bF}{\partial z_{ij} }=\sum_t (D_{ij}^{it} + \Delta_{ij}^{it}) \be_t^\top$, we have
	\begin{align*}
		&\sum_{i=1}^n \sum_{j=1}^p\bF^\top\bZ\be_j\be_i^\top\frac{\partial \bF}{\partial z_{ij}}\\
		=~& \sum_{ij}\bF^\top\bZ\be_j \sum_{t=1}^T (D_{ij}^{it} + \Delta_{ij}^{it}) \be_t^\top\\
		=~& \sum_{ij}
		\bF^\top \bZ \be_j \sum_{t} D_{ij}^{it} \be_t^\top 
		+ \underbrace{\sum_{ij} \bF^\top \bZ \be_j \sum_{t} p \Delta_{ij}^{it} \be_t^\top}_{\bJ_1},
	\end{align*}
	where the first term can be simplified as below
	\begin{align*}
		&\sum_{j=1}^p\sum_{i=1}^n
		\bF^\top \bZ \be_j \sum_{t=1}^T D_{ij}^{it} \be_t^\top\\
		=~&-\sum_{j=1}^p\sum_{i=1}^n
		\bF^\top \bZ \be_j \sum_{t=1}^T(\be_j^\top\bH \otimes \be_i^\top)(\bI_{nT} - \bN)(\be_t\otimes \be_i)\be_t^\top\\
		=~& -
		\bF^\top \bZ \bH
		\Bigl[
		\sum_i
		(
		\bI_T \otimes \be_i^\top
		)
		(\bI_{nT} - \bN)
		(\bI_T \otimes \be_i)
		\Bigr]\\
		=~& -
		\bF^\top \bZ\bH(n\bI_T - \hbA).
	\end{align*}
	
	For the second equality, since $\frac{\partial \bF}{\partial z_{ij} }
	=\sum_{lt} \be_l (D_{ij}^{lt} + \Delta_{ij}^{lt}) \be_t^\top$,
	\begin{align*}
		\sum_{ij}\Big(\frac{\partial \bF}{\partial z_{ij} }\Big)^\top
		\bZ \be_j \be_i^\top \bF
		= \underbrace{\sum_{ijtl} \be_t D_{ij}^{lt} \be_l^\top\bZ\be_j \be_i^\top \bF}_{\bJ_2} +\sum_{ijtl} \be_t \Delta_{ij}^{lt} \be_l^\top\bZ\be_j \be_i^\top \bF, 
	\end{align*}
	where the second term can be simplified as below,
	\begin{align*}
		&\sum_{ijtl} \be_t \Delta_{ij}^{lt} \be_l^\top\bZ\be_j \be_i^\top \bF\\
		~=& \sum_{ijtl} \be_t\Delta_{ij}^{lt}(\be_i\otimes \be_j)^\top (\bF \otimes \bZ^\top\be_l)\\
		~=& -\sum_{ijtl} \be_t(\be_t^\top \otimes \be_l^\top)(\bI_T\otimes \bX)
		\bM^\dagger (\bI_T\otimes \bSigma^{\frac12})\bigl(\bF^\top \otimes \bI_{p}\bigr) (\be_i \otimes\be_j)(\be_i\otimes \be_j)^\top (\bF \otimes \bZ^\top\be_l)\\
		~=& -\sum_{l} (\bI_T \otimes \be_l^\top)(\bI_T\otimes \bX)
		\bM^\dagger (\bI_T\otimes \bSigma^{\frac12}) (\bF^\top\otimes \bI_p) (\bF \otimes \bZ^\top\be_l)\\
		~=& -\sum_{l} (\bI_T \otimes \be_l^\top)(\bI_T\otimes \bX)
		\bM^\dagger \bigl(\bF^\top\bF \otimes \bX^\top\be_l\bigr) \\
		~=& -\sum_{l} (\bI_T \otimes \be_l^\top)(\bI_T\otimes \bX)
		\bM^\dagger \bigl(\bI_T \otimes \bX^\top\be_l\bigr) \bF^\top\bF\\
		~=& - \hbA \bF^\top \bF,
	\end{align*}
	where the last line uses the expression of $\hbA$ in \eqref{eq: hbA-matrix}.
	
	It remains to bound the norm of $\bJ_1$ and $\bJ_2$. 
	To bound $\fnorm{\bJ_2}$, recall the definition of $\bJ_2$, 
	\begin{align*}
		\bJ_2 &= \sum_{ijlt} \be_t D_{ij}^{lt} \be_l^\top\bZ\be_j \be_i^\top\bF\\
		&= - \sum_{ijlt} \be_t (\be_j^\top\bH \otimes \be_i^\top) (\bI_{nT} - \bN ) (\be_t\otimes \be_l) \be_l^\top\bZ\be_j \be_i^\top\bF\\
		&= - \sum_{ijlt} \be_t (\be_t^\top\otimes \be_l^\top) (\bI_{nT} - \bN) (\bH^\top\be_j \otimes \be_i) \be_l^\top\bZ\be_j \be_i^\top\bF\\
		&= - \sum_{ijl} (\bI_T\otimes \be_l^\top) (\bI_{nT} - \bN) (\bH^\top\be_j \otimes \be_i) \be_l^\top\bZ\be_j \be_i^\top\bF\\
		&= - \sum_{l} (\bI_T\otimes \be_l^\top) (\bI_{nT} - \bN) (\bH^\top\bZ^\top\be_l \otimes \bF),
	\end{align*}
	Since $\bN$ is non-negative definite with $\opnorm*{\bN} \le 1$, $ \opnorm*{\bI_{nT} - \bN} \le 1$, 
	\begin{align*}
		\fnorm*{\bJ_2} &\le \sum_{l} \fnorm*{(\bI_T\otimes \be_l^\top) (\bI_{nT} - \bN) (\bH^\top\bZ^\top\be_l \otimes \bF)}\\
		&\le \sum_{l} \opnorm*{(\bI_T\otimes \be_l^\top) (\bI_{nT} - \bN)}\fnorm*{(\bH^\top\bZ^\top\be_l \otimes \bF)}\\
		&\le \sum_{l} \fnorm*{(\bH^\top\bZ^\top\be_l \otimes \bF)}\\
		&= \sum_{l} \fnorm*{\bH^\top\bZ^\top\be_l} \fnorm*{\bF}\\
		&\le n^{\frac 12} \fnorm{\bZ\bH} \fnorm{\bF}\\
		&\le n^{\frac 12} \opnorm*{\bZ}\fnorm{\bH} \fnorm*{\bF}.
	\end{align*}
	
	To bound $\fnorm{\bJ_1}$, recall the definition of $\bJ_1$, 
	$$\bJ_1 = \sum_{ij}\bF^\top\bZ\be_j \sum_t \Delta_{ij}^{it} \be_t^\top.$$
	For each $t, t'\in [T]$, 
	\begin{align*}
		\be_{t'}^\top\bJ_1 \be_{t} & = \be_{t'}^\top \big[\sum_{ij}\bF^\top\bZ\be_j \sum_t \Delta_{ij}^{it} \be_t^\top\big]\be_t\\
		&=\sum_{j=1}^p\sum_{i=1}^n
		\be_{t'}^\top \bF^\top \bZ \be_j  \Delta_{ij}^{it}
		\\&
		=
		-\sum_i
		(\be_i^\top\bF \otimes \be_{t'}^\top \bF^\top)
		\bN
		(\be_t \otimes \be_i)\\
		&=
		-\be_{t'}^\top\sum_i
		(\be_i^\top\bF \otimes \bF^\top)
		\bN
		( \bI_T\otimes \be_i)\be_t,
	\end{align*}
	Thus,  
	$
	\bJ_1 = \sum_i
	(\be_i^\top\bF \otimes \bF^\top)
	\bN
	( \bI_T\otimes \be_i),
	$
	and hence
	\begin{align*}
		\fnorm{\bJ_1} &= \fnorm{\sum_i(\be_i^\top\bF \otimes \bF^\top)\bN(\bI_T\otimes \be_i)}\\
		&\le \sum_i\fnorm{(\be_i^\top\bF \otimes \bF^\top)\bN(\bI_T\otimes \be_i)}\\
		&\le \sum_i\fnorm{(\be_i^\top\bF \otimes \bF^\top)}\opnorm{\bN(\bI_T\otimes \be_i)}\\
		&\le \sum_i\norm{\be_i^\top\bF} \fnorm{\bF}\\
		&\le n^{\frac 12}\fnorm{\bF}^2,
	\end{align*}
	where the first inequality is by sub-additivity of Frobenius norm, the second inequality uses $\fnorm{\bA_1\bA_2}\le \fnorm{\bA_1}\opnorm{\bA_2}$ for any matrices $\bA_1,\bA_2$ with appropriate dimensions, the third inequality is by $\opnorm{\bN}\le1$ from \Cref{lem: MN}, the last inequality is by Cauchy-Schwarz inequality. 
\end{proof}

\end{document}